\documentclass[a4paper]{amsart}
\usepackage[T1]{fontenc}
\usepackage[latin9]{inputenc}
\usepackage{mathrsfs}
\usepackage{amstext}
\usepackage{amsthm}
\usepackage{amssymb}
\usepackage{stackrel}
\usepackage{graphicx}
\usepackage[pdfusetitle,
 bookmarks=true,bookmarksnumbered=false,bookmarksopen=false,
 breaklinks=false,pdfborder={0 0 0},pdfborderstyle={},backref=false,colorlinks=false]
 {hyperref}

\makeatletter

\pdfpageheight\paperheight
\pdfpagewidth\paperwidth

\numberwithin{equation}{section}
\numberwithin{figure}{section}

\usepackage{tikz-cd}
\usepackage{mathtools}
\usepackage{adjustbox}
\usepackage{enumitem}
\tikzcdset{scale cd/.style={every label/.append style={scale=#1},
    cells={nodes={scale=#1}}}}
\usepackage{quiver}
\usepackage{eucal}
\usepackage{mleftright}
\mleftright

\makeatother

\theoremstyle{plain}
\newtheorem{thm}{\protect\theoremname}[section]
\theoremstyle{definition}
\newtheorem{defn}[thm]{\protect\definitionname}
\newtheorem{question}[thm]{\protect\questionname}
\theoremstyle{remark}
\newtheorem{rem}[thm]{\protect\remarkname}
\theoremstyle{definition}
\newtheorem{warning}[thm]{\protect\warningname}
\theoremstyle{plain}
\newtheorem{lem}[thm]{\protect\lemmaname}
\theoremstyle{definition}
\newtheorem{example}[thm]{\protect\examplename}
\theoremstyle{plain}
\newtheorem{prop}[thm]{\protect\propositionname}
\newtheorem{cor}[thm]{\protect\corollaryname}
\theoremstyle{definition}
\newtheorem{notation}[thm]{\protect\notationname}
\newtheorem{variant}[thm]{\protect\variantname}
\newtheorem{construction}[thm]{\protect\constructionname}

\providecommand{\constructionname}{Construction}
\providecommand{\corollaryname}{Corollary}
\providecommand{\definitionname}{Definition}
\providecommand{\examplename}{Example}
\providecommand{\lemmaname}{Lemma}
\providecommand{\notationname}{Notation}
\providecommand{\propositionname}{Proposition}
\providecommand{\questionname}{Question}
\providecommand{\remarkname}{Remark}
\providecommand{\theoremname}{Theorem}
\providecommand{\variantname}{Variant}
\providecommand{\warningname}{Warning}

\begin{document}
\global\long\def\sf#1{\mathsf{#1}}%

\global\long\def\scr#1{\mathscr{{#1}}}%

\global\long\def\cal#1{\mathcal{#1}}%

\global\long\def\bb#1{\mathbb{#1}}%

\global\long\def\bf#1{\mathbf{#1}}%

\global\long\def\frak#1{\mathfrak{#1}}%

\global\long\def\fr#1{\mathfrak{#1}}%

\global\long\def\u#1{\underline{#1}}%

\global\long\def\tild#1{\widetilde{#1}}%

\global\long\def\mrm#1{\mathrm{#1}}%

\global\long\def\pr#1{\left(#1\right)}%

\global\long\def\abs#1{\left|#1\right|}%

\global\long\def\inp#1{\left\langle #1\right\rangle }%

\global\long\def\br#1{\left\{  #1\right\}  }%

\global\long\def\norm#1{\left\Vert #1\right\Vert }%

\global\long\def\hat#1{\widehat{#1}}%

\global\long\def\opn#1{\operatorname{#1}}%

\global\long\def\bigmid{\,\middle|\,}%

\global\long\def\Top{\sf{Top}}%

\global\long\def\Set{\sf{Set}}%

\global\long\def\SS{\sf{sSet}}%

\global\long\def\Kan{\sf{Kan}}%

\global\long\def\Cat{\mathcal{C}\sf{at}}%

\global\long\def\Grpd{\mathcal{G}\sf{rpd}}%

\global\long\def\Res{\mathcal{R}\sf{es}}%

\global\long\def\imfld{\cal M\mathsf{fld}}%

\global\long\def\ids{\cal D\sf{isk}}%

\global\long\def\D{\sf{Disk}}%

\global\long\def\ich{\cal C\sf h}%

\global\long\def\SW{\mathcal{SW}}%

\global\long\def\SHC{\mathcal{SHC}}%

\global\long\def\Fib{\mathcal{F}\mathsf{ib}}%

\global\long\def\Bund{\mathcal{B}\mathsf{und}}%

\global\long\def\Fam{\cal F\sf{amOp}}%

\global\long\def\B{\sf B}%

\global\long\def\Spaces{\sf{Spaces}}%

\global\long\def\Mod{\sf{Mod}}%

\global\long\def\Nec{\sf{Nec}}%

\global\long\def\Fin{\sf{Fin}}%

\global\long\def\Ch{\sf{Ch}}%

\global\long\def\Ab{\sf{Ab}}%

\global\long\def\SA{\sf{sAb}}%

\global\long\def\P{\mathsf{POp}}%

\global\long\def\Op{\mathcal{O}\mathsf{p}}%

\global\long\def\Opg{\mathcal{O}\mathsf{p}^{\mathrm{gn}}_{\infty}}%

\global\long\def\Tup{\mathsf{Tup}}%

\global\long\def\H{\cal H}%

\global\long\def\Mfld{\cal M\mathsf{fld}}%

\global\long\def\Disk{\cal D\mathsf{isk}}%

\global\long\def\Acc{\mathcal{A}\mathsf{cc}}%

\global\long\def\Pr{\mathcal{P}\mathrm{\mathsf{r}}}%

\global\long\def\Del{\mathbf{\Delta}}%

\global\long\def\id{\operatorname{id}}%

\global\long\def\Aut{\operatorname{Aut}}%

\global\long\def\End{\operatorname{End}}%

\global\long\def\Hom{\operatorname{Hom}}%

\global\long\def\Ext{\operatorname{Ext}}%

\global\long\def\sk{\operatorname{sk}}%

\global\long\def\ihom{\underline{\operatorname{Hom}}}%

\global\long\def\N{\mathrm{N}}%

\global\long\def\-{\text{-}}%

\global\long\def\op{\mathrm{op}}%

\global\long\def\To{\Rightarrow}%

\global\long\def\rr{\rightrightarrows}%

\global\long\def\rl{\rightleftarrows}%

\global\long\def\mono{\rightarrowtail}%

\global\long\def\epi{\twoheadrightarrow}%

\global\long\def\comma{\downarrow}%

\global\long\def\ot{\leftarrow}%

\global\long\def\corr{\leftrightsquigarrow}%

\global\long\def\lim{\operatorname{lim}}%

\global\long\def\colim{\operatorname{colim}}%

\global\long\def\holim{\operatorname{holim}}%

\global\long\def\hocolim{\operatorname{hocolim}}%

\global\long\def\Ran{\operatorname{Ran}}%

\global\long\def\Lan{\operatorname{Lan}}%

\global\long\def\Sk{\operatorname{Sk}}%

\global\long\def\Sd{\operatorname{Sd}}%

\global\long\def\Ex{\operatorname{Ex}}%

\global\long\def\Cosk{\operatorname{Cosk}}%

\global\long\def\Sing{\operatorname{Sing}}%

\global\long\def\Sp{\operatorname{Sp}}%

\global\long\def\Spc{\operatorname{Spc}}%

\global\long\def\Fun{\operatorname{Fun}}%

\global\long\def\map{\operatorname{map}}%

\global\long\def\diag{\operatorname{diag}}%

\global\long\def\Gap{\operatorname{Gap}}%

\global\long\def\cc{\operatorname{cc}}%

\global\long\def\ob{\operatorname{ob}}%

\global\long\def\Map{\operatorname{Map}}%

\global\long\def\Rfib{\operatorname{RFib}}%

\global\long\def\Lfib{\operatorname{LFib}}%

\global\long\def\Tw{\operatorname{Tw}}%

\global\long\def\Equiv{\operatorname{Equiv}}%

\global\long\def\Arr{\operatorname{Arr}}%

\global\long\def\Cyl{\operatorname{Cyl}}%

\global\long\def\Path{\operatorname{Path}}%

\global\long\def\Alg{\operatorname{Alg}}%

\global\long\def\ho{\operatorname{ho}}%

\global\long\def\Comm{\operatorname{Comm}}%

\global\long\def\Triv{\operatorname{Triv}}%

\global\long\def\triv{\operatorname{triv}}%

\global\long\def\Env{\operatorname{Env}}%

\global\long\def\Act{\operatorname{Act}}%

\global\long\def\loc{\operatorname{loc}}%

\global\long\def\Assem{\operatorname{Assem}}%

\global\long\def\Nat{\operatorname{Nat}}%

\global\long\def\Conf{\operatorname{Conf}}%

\global\long\def\Rect{\operatorname{Rect}}%

\global\long\def\Emb{\operatorname{Emb}}%

\global\long\def\Homeo{\operatorname{Homeo}}%

\global\long\def\mor{\operatorname{mor}}%

\global\long\def\Germ{\operatorname{Germ}}%

\global\long\def\Post{\operatorname{Post}}%

\global\long\def\Sub{\operatorname{Sub}}%

\global\long\def\Shv{\operatorname{Shv}}%

\global\long\def\Cov{\operatorname{Cov}}%

\global\long\def\Disc{\operatorname{Disc}}%

\global\long\def\Open{\operatorname{Open}}%

\global\long\def\Disj{\operatorname{Disj}}%

\global\long\def\Exc{\operatorname{Exc}}%

\global\long\def\Homog{\operatorname{Homog}}%

\global\long\def\PL{\mathrm{PL}}%

\global\long\def\sm{\mathrm{sm}}%

\global\long\def\istp{\mathrm{istp}}%

\global\long\def\exh{\mathrm{exh}}%

\global\long\def\disj{\mathrm{disj}}%

\global\long\def\Top{\mathrm{Top}}%

\global\long\def\lax{\mathrm{lax}}%

\global\long\def\weq{\mathrm{weq}}%

\global\long\def\fib{\mathrm{fib}}%

\global\long\def\fin{\mathrm{fin}}%

\global\long\def\inert{\mathrm{inert}}%

\global\long\def\act{\mathrm{act}}%

\global\long\def\cof{\mathrm{cof}}%

\global\long\def\inj{\mathrm{inj}}%

\global\long\def\univ{\mathrm{univ}}%

\global\long\def\Ker{\opn{Ker}}%

\global\long\def\Coker{\opn{Coker}}%

\global\long\def\Im{\opn{Im}}%

\global\long\def\Coim{\opn{Im}}%

\global\long\def\coker{\opn{coker}}%

\global\long\def\im{\opn{\mathrm{im}}}%

\global\long\def\coim{\opn{coim}}%

\global\long\def\gn{\mathrm{gn}}%

\global\long\def\Mon{\opn{Mon}}%

\global\long\def\Un{\opn{Un}}%

\global\long\def\St{\opn{St}}%

\global\long\def\cun{\widetilde{\opn{Un}}}%

\global\long\def\cst{\widetilde{\opn{St}}}%

\global\long\def\Sym{\operatorname{Sym}}%

\global\long\def\CA{\operatorname{CAlg}}%

\global\long\def\Int{\operatorname{Int}}%

\global\long\def\SP{\operatorname{SP}}%

\global\long\def\Nbd{\operatorname{Nbd}}%

\global\long\def\rd{\mathrm{rd}}%

\global\long\def\xmono#1#2{\stackrel[#2]{#1}{\rightarrowtail}}%

\global\long\def\xepi#1#2{\stackrel[#2]{#1}{\twoheadrightarrow}}%

\global\long\def\adj{\stackrel[\longleftarrow]{\longrightarrow}{\bot}}%

\global\long\def\btimes{\boxtimes}%

\global\long\def\ps#1#2{\prescript{}{#1}{#2}}%

\global\long\def\ups#1#2{\prescript{#1}{}{#2}}%

\global\long\def\hofib{\mathrm{hofib}}%

\global\long\def\cofib{\mathrm{cofib}}%

\global\long\def\Vee{\bigvee}%

\global\long\def\w{\wedge}%

\global\long\def\t{\otimes}%

\global\long\def\bp{\boxplus}%

\global\long\def\rcone{\triangleright}%

\global\long\def\lcone{\triangleleft}%

\global\long\def\S{\mathsection}%

\global\long\def\p{\prime}%

\global\long\def\pp{\prime\prime}%

\global\long\def\W{\overline{W}}%

\global\long\def\o#1{\overline{#1}}%

\global\long\def\fp{\overrightarrow{\times}}%

\title{A Context for Manifold Calculus}
\begin{abstract}
We develop Weiss's manifold calculus in the setting of $\infty$-categories,
where we allow the target $\infty$-category to be any $\infty$-category
with small limits. We will establish the connection between polynomial
functors, Kan extensions, and Weiss sheaves, and will classify homogeneous
functors. We will also generalize Weiss and Boavida de Brito's theorem
to functors taking values in arbitrary $\infty$-categories with small
limits.
\end{abstract}

\author{Kensuke Arakawa}
\email{arakawa.kensuke.22c@st.kyoto-u.ac.jp}
\address{Department of Mathematics, Kyoto University, Kyoto, 606-8502, Japan}

\maketitle
\tableofcontents{}

\section*{Introduction}

Let $M$ be a smooth $n$-manifold, and let $\Open\pr M$ be the poset
of open sets of $M$. The idea of Weiss's \textit{manifold calculus}
is to study space-valued presheaves $F:\Open\pr M^{\op}\to\sf{Spaces}$
on $M$ by constructing a tower of presheaves
\[
F\to\cdots\to T_{k}F\to\cdots\to T_{1}F\to T_{0}F.
\]
Each $T_{k}F$ is the best ``polynomial functor'' of degree $\leq k$
that approximates $F$. Borrowing an analogy from calculus, the tower
is called the \textbf{Taylor tower }of $F$ and the functors $T_{k}F$
the \textbf{polynomial approximations} of $F$. 

The Taylor tower is determined by the definition of polynomials, so
we must make this explicit. There is a heuristic in homotopy theory
that homology behaves like linear (or affine) functions. A homology
theory is more or less characterized by the excision axiom, so we
make the following definition:
\begin{defn}
\cite[Definition 2.2]{Weiss99} We say that $F$ is\textbf{ polynomial
of degree $\leq k$}, or \textbf{$k$-excisive}, if for each open
set $U\subset M$ and pairwise disjoint closed subsets $A_{0},\dots,A_{k}$
of $U$, the map
\[
F\pr U\to\underset{\emptyset\subsetneq S\subset\{0,\dots,k\}}{\holim}F\pr{U\setminus\bigcup_{i\in S}A_{i}}
\]
is a homotopy equivalence.
\end{defn}

For example, $1$-excisivity is equivalent to the following condition:
For every open set $U\subset M$ and every pair of disjoint closed
sets $A_{0},A_{1}\subset U$, the square 
\[\begin{tikzcd}
	{F(U)} & {F(U\setminus A_1)} \\
	{F(U\setminus A_0)} & {F(U\setminus (A_0\cup A_1))}
	\arrow[from=1-1, to=1-2]
	\arrow[from=1-2, to=2-2]
	\arrow[from=2-1, to=2-2]
	\arrow[from=1-1, to=2-1]
\end{tikzcd}\]is homotopy cartesian. This corresponds to the conventional usage
of the term ``excisivity,'' hence justifying our terminology. 

We also want to ensure that $F$ interacts well with the smooth structure
on $M$, and also that $F$ is ``continuous.'' We thus make the
following definition:
\begin{defn}
\cite[$\S$1]{Weiss99} A space-valued presheaf $F$ on $M$ is said
to be \textbf{good} if it satisfies the following conditions:
\begin{itemize}
\item Let $U\subset V$ be an inclusion of open sets of $M$ which is a
smooth isotopy equivalence. Then the map $F\pr V\to F\pr U$ is a
homotopy equivalence.
\item Let $U_{0}\subset U_{1}\subset\cdots$ be an increasing sequence of
open subsets of $M$. Then the map $F(\bigcup_{i\geq0}U_{i})\to\holim_{i}F\pr{U_{i}}$
is a homotopy equivalence.
\end{itemize}
\end{defn}

With these definitions, we can now define polynomial approximations.
\begin{defn}
Let $F,G:\Open\pr M^{\op}\to\sf{Spaces}$ be functors that carry smooth
isotopy equivalences to homotopy equivalences. A natural transformation
$\alpha:F\to G$ is said to exhibit $G$ as a \textbf{$k$th polynomial
approximation} of $F$ if $G$ is polynomial of degree $\leq k$ and
$\alpha$ is homotopically initial among the natural transformations
to good polynomial functors of degree $\leq k$. In this case, we
will write $G=T_{k}F$. 
\end{defn}

At this point, several questions come up naturally:
\begin{question}
\label{que:1}When does a good presheaf on $M$ admit polynomial approximations?
If it does, how do we construct the approximations?
\end{question}

\begin{question}
\label{que:2}Can we classify polynomial functors?
\end{question}

Suppose now that $F$ takes values in the category of pointed spaces
and that we have constructed its polynomial approximations. In good
cases, the analysis of $F$ reduces, up to extension problems, to
studying the fibers $L_{k}F=\hofib\pr{T_{k}F\to T_{k-1}F}$. The functor
$L_{k}F$ has the special property that it is \textbf{homogeneous}:
It is polynomial of degree $\leq k$, and its $\pr{k-1}$th polynomial
approximation vanishes. This leads to the following question:
\begin{question}
\label{que:3}Can we classify homogeneous functors?
\end{question}

In \cite{Weiss99}, Weiss gave complete answers to these questions.
We now review his answers as two separate theorems (Theorems \ref{thm:Weiss1}
and \ref{thm:Weiss2}).

The first theorem answers Questions \ref{que:1} and \ref{que:2}.
We will let $\Disj^{\leq k}_{{\rm sm}}\pr M$ denote the subset of
$\Open\pr M$ consisting of the elements that are diffeomorphic to
$\bb R^{n}\times S$, where $S$ is a finite set of cardinality at
most $k$.
\begin{thm}
\cite[Lemma 3.8, Theorem 4.1, Theorem 5.1]{Weiss99}\label{thm:Weiss1}
Let $G:\Disj^{\leq k}_{{\rm sm}}\pr M^{\op}\to\Spaces$ be a functor
carrying isotopy equivalences to homotopy equivalences. The homotopy
right Kan extension of $G$ along the inclusion $\Disj^{\leq k}_{{\rm sm}}\pr M^{\op}\hookrightarrow\Open\pr M^{\op}$
is good and $k$-excisive. Moreover, every good polynomial functor
of degree $\leq k$ arises in this way.

In particular, every good space-valued presheaf $F$ on $M$ admits
polynomial approximations in all degrees, constructed by homotopy
right Kan extensions.
\end{thm}

And the second theorem addresses Question \ref{que:3}:
\begin{thm}
\cite[Theorem 8.5]{Weiss99}\label{thm:Weiss2} Homogeneous functors
of degree $k$ taking values in pointed spaces admit a classification
in terms of fibrations over $B_{k}\pr M$ equipped with sections near
the ``fat diagonal'', where $B_{k}\pr M$ denotes the space of unordered
configurations of $k$ points in $M$. 
\end{thm}

Most of the things (perhaps except for Theorem \ref{thm:Weiss2})
we have discussed so far make sense for presheaves taking values in
categories other than that of spaces. Moreover, it seems quite natural
to consider this generalization. For example, if one is interested
in (co)homology, then it is not hard to imagine that presheaves of
(co)chain complexes or spectra would be very useful. As such, numerous
studies consider such generalized manifold calculus; see \cite{Wei04,RW14,HK23,TS_poly,TS_homog}
for instance. In this paper, we will answer Questions \ref{que:1},
\ref{que:2}, and \ref{que:3} in the generalized setting, and aim
to set a foundation of generalized manifold calculus with arbitrary
targets.

Here is what we will do. We will generalize Theorems \ref{thm:Weiss1}
and \ref{thm:Weiss2} to presheaves on $M$ taking values in an arbitrary
$\infty$-category $\cal C$ with small limits. It is straightforward
to state a generalization of Theorem \ref{thm:Weiss1}: The definitions
of good functors, polynomial functors, polynomial approximations,
and homogeneous functors directly carry over to $\cal C$-valued presheaves;
we just have to replace homotopy limits by $\infty$-categorical limits.
So a generalization of Theorem \ref{thm:Weiss1} will be as follows.
\begin{thm}
[Theorem \ref{thm:best_approx_exists}]\label{thm:main1}The right
Kan extension functor restricts to a categorical equivalence
\[
\Fun_{{\rm istp}}\pr{\Disj^{\leq k}_{\sm}\pr M^{\op},\cal C}\xrightarrow{\simeq}\Exc^{k}_{{\rm good}}\pr{M^{\op};\cal C},
\]
where:
\begin{itemize}
\item $\Fun_{{\rm istp}}\pr{\Disj^{\leq k}_{\sm}\pr M^{\op},\cal C}$ denotes
the $\infty$-category of functors $\Disj^{\leq k}_{\sm}\pr M\to\cal C$
carrying smooth isotopy equivalences to equivalences; and
\item $\Exc^{k}_{{\rm good}}\pr{M^{\op};\cal C}$ denotes the $\infty$-category
of $k$-excisive good functors $\Open\pr M^{\op}\to\cal C$.
\end{itemize}
In particular, every good functor $F:\Open\pr M^{\op}\to\cal C$ admits
a $k$th polynomial approximation, which is the right Kan extension
of $F\vert\Disj^{\leq k}_{\sm}\pr M^{\op}$.
\end{thm}

To generalize Theorem \ref{thm:Weiss2}, recall that a fibration over
$B_{k}\pr M$ is equivalent (via the $\infty$-categorical Grothendieck
construction \cite[2.2.1.2]{HTT}) to a functor $\Sing B_{k}\pr M\to\cal S$,
where $\cal S$ denotes the $\infty$-category of spaces. With this
in mind, we generalize Theorem \ref{thm:Weiss2} as follows:
\begin{thm}
[Corollary \ref{cor:homog_class}, Theorem \ref{thm:inv}]\label{thm:main2}Suppose
that $\cal C$ is pointed. There is an equivalence of $\infty$-categories
\[
\Fun\pr{\Sing B_{k}\pr M,\cal C}\simeq\opn{Homog}^{k}_{{\rm good}}\pr{M^{\op};\cal C},
\]
where $\opn{Homog}^{k}_{{\rm good}}\pr{M^{\op};\cal C}$ denotes the
$\infty$-category of good functors $\Open\pr M^{\op}\to\cal C$ that
are of homogeneous of degree $k$. 
\end{thm}

\begin{rem}
Theorem \ref{thm:main2} says nothing about the fat diagonal, but
a more precise statement involving the fat diagonal statement will
follow. (See Theorem \ref{thm:inv}.) In particular, when applied
to the case where $\cal C$ is the $\infty$-category $\cal S_{\ast}$
of pointed spaces, the theorem recovers Weiss's classification of
homogeneous functors.
\end{rem}

\begin{rem}
Most of the statements of Theorems \ref{thm:main1} and \ref{thm:main2}
make sense even when $M$ does not have a smooth structure. Because
of this, we will work mostly with topological manifolds, explaining
how various statements will be affected in the presence of smooth
and PL structures here and there.
\end{rem}

\begin{rem}
In \cite{BW13}, Boavida de Brito and Weiss established an analog
of Theorem \ref{thm:Weiss2} for \textit{context-free}\footnote{The word ``context'' here has no relation with the title of this
paper.}\textit{ manifold calculus}, i.e., calculus of functors defined on
the category of \textit{all} smooth manifolds of a fixed dimension,
not just on the poset of open sets on a single manifold. We will also
state and prove a generalization of this in Section \ref{sec:context_free}.
\end{rem}

\begin{rem}
Our paper is thematically very similar to the works of Songhafouo
Tsopm\'{e}n\'{e}, Arnaud, and Stanley\cite{TS_poly,TS_vgh,TS_homog},
in which the authors ask the extent to which manifold calculus can
be developed for presheaves taking values in categories other than
spaces. (They mainly work in the setting of model categories and employ
very different technique than ours.) Our theorems will recover the
results in their paper and solve their conjecture. More specifically:
\begin{itemize}
\item By applying Theorem \ref{thm:main1} and \ref{thm:main2} to the case
where $\cal C$ is the underlying $\infty$-category of a simplicial
model category, we recover the main results of \cite{TS_poly}. 
\item By applying Theorem \ref{thm:main2} to the case where $\cal C$ is
an ordinary category, we recover the main result of \cite{TS_vgh}.
\item By applying Theorem \ref{thm:main2} to the case where $\cal C$ is
the underlying $\infty$-category of a simplicial model category $\bf C$,
we recover the main result and a conjecture of \cite{TS_homog}: For
each object $X\in\bf C$, the homotopy classes of maps $B_{k}\pr M\to B\mathrm{hAut}\pr X$
are in bijection with the weak equivalence class good functors $\Open\pr M^{\op}\to\bf C$
that are of homogeneous of degree $k$ and carrying disjoint union
of $k$ open balls to objects weakly equivalent to $X$. 
\end{itemize}
\end{rem}

At first glance, one might expect that proofs of Theorem \ref{thm:main1}
and \ref{thm:main2} are straightforward, requiring only a minor modification
of Weiss's argument. While this is true to some extent, several parts
of Weiss\textquoteright s original proof rely critically on concrete
constructions of spaces, which do not generalize easily. Moreover,
our theorems is slightly stronger than Weiss's results, in that it
gives an equivalence of $\infty$-categories, and this added generality
introduces further complications. To overcome these challenges, we
develop and apply more sophisticated $\infty$-categorical techniques.

Theorem \ref{thm:main1} and \ref{thm:main2} have a wide array of
applications outside what has traditionally been regarded as part
of manifold calculus. For example, (part of) the theory of \textit{factorization
homology} \cite{FHprimer} deals with functors of the form
\[
H:\Mfld_{\sm,n}\to\cal C,
\]
where $\Mfld_{\sm,n}$ denotes the $\infty$-category of smooth $n$-manifolds
and smooth embeddings. By precomposing the functor $\Open\pr M\to\Mfld_{\sm,n}$,
we obtain a $\cal C^{\op}$-valued precosheaf $F:\Open\pr M^{\op}\to\cal C^{\op}$.
An observation of Ayala and Francis \cite[Corollary 2.2.24]{FHprimer}
(which will be reproved in Subsection \ref{sec:taylor}) essentially
asserts that the functor $F$ is the limit of its Taylor tower, in
the sense that the map $F\to\lim_{k}T_{k}F$ is an equivalence. This
means that invariants of manifolds realized by factorization homology,
which include a very interesting class of invariants \cite{FHprimer},
can be approximated by their Taylor approximations. It is crucial
here to have a flexibility in the target category, which this paper
offers, because invariants can arise in many different forms, such
as spectra, chain complexes, or even $\infty$-operads \cite{A24d}.

\subsection*{Outline of the Paper}

This paper consists of 5 sections. Section \ref{sec:poly} concerns
the existence and classification of polynomial functors. In Section
\ref{sec:taylor}, we construct and consider the convergence problem
of Taylor towers. Section \ref{sec:homog} classifies homogeneous
functors. In Section \ref{sec:context_free}, we prove a generalization
of the results in Sections \ref{sec:poly} and \ref{sec:taylor} for
context-free manifold calculus. Section \ref{sec:bd}, discusses how
these ideas extend to manifolds with boundary

The paper is also accompanied by an appendix, consisting of three
sections. The purpose of the appendix is to record miscellaneous results
used in the main body of the paper, so it should be consulted only
when the need arises. 

\subsection*{Acknowledgment}

I thank Michael Weiss for carefully explaining some of the key ideas
of his papers, David Ayala for kindly sharing the proof of Theorem
\ref{thm:localizing_wrt_istpy}, and Alexander Kupers for generously
answering my question on MathOverflow \cite{MO452495}. I am also
indebted to Daisuke Kishimoto and Mitsunobu Tsutaya for their constant
support and encouragement. Last, but not least, I express my sincerest
gratitude to the referee for carefully reading the manuscript and
making valuable feedback, which has significantly improved this paper. 

\subsection*{Notation and Terminology}
\begin{itemize}
\item The term ``$\infty$-category'' will be used as a synonym of quasi-category
of Joyal \cite{Joyal_qcat_Kan}. We will mainly follow \cite{HTT}
for notation and terminology for $\infty$-categories, with the following
exceptions:
\begin{itemize}
\item We say that a morphism $f:S\to T$ of simplicial sets is\textbf{ final
}if it is cofinal in the sense of \cite[Definition 4.1.1.1 ]{HTT}.
We say that $f$ is \textbf{initial }if its opposite is final.
\item We say that a functor $f:\cal C\to\cal D$ of $\infty$-categories
is a \textbf{localization }with respect to a set $S$ of morphisms
of $\cal C$ if it is initial among functors inverting the morphisms
in $S$. (See Subsection \ref{subsec:loc} for a precise definition.)
This definition is more general than the definition in \cite{HTT}.
\item Except in Appendix \ref{App:(Non)-Loc}, we will not distinguish between
ordinary categories and their nerves.
\end{itemize}
\item If $\cal C$ is an $\infty$-category, we denote its maximal sub Kan
complex by $\cal C^{\simeq}$ and refer to it as the \textbf{core
}of $\cal C$.
\item Let $n\geq0$. An \textbf{$n$-manifold}, or a\textbf{ manifold of
dimension} $n$, will always mean a topological manifold without boundary,
i.e., a second countable, Hausdorff topological space which admits
an open cover by open sets homeomorphic to $\bb R^{n}$. If $n\geq1$,
we define an \textbf{$n$-manifold with boundary} to be a second countable,
Hausdorff topological space which admits an open cover by open sets
homeomorphic to $\bb R^{n}$ or $[0,\infty)\times\bb R^{n-1}$. In
particular, \textit{not every $n$-manifold with boundary is an $n$-manifold}.
\item Let $n\geq0$. Given $n$-manifolds $M$ and $N$, we let $\Emb\pr{M,N}$
denote the topological space of embeddings $M\to N$, topologized
by the compact-open topology. We let $\Mfld^{T}_{n}$ denote the topological
category whose objects are $n$-manifolds and whose hom spaces are
given by $\Emb\pr{-,-}$, and let $\Mfld^{\Delta}_{n}$ denote the
simplicial category obtained from $\Mfld^{T}_{n}$ by applying the
singular complex functor to the hom spaces. We let $\Mfld_{n}$ denote
the homotopy coherent nerve of the topological category $\Mfld^{T}_{n}$,
and let $\sf{Mfld}_{n}$ denote the nerve of the the \textit{ordinary}
category obtained from $\Mfld^{T}_{n}$ by forgetting the topology
of the mapping spaces. Equivalences of $\Mfld_{n}$ are called \textbf{isotopy
equivalences}.

We let $\Disk_{n}\subset\Mfld_{n}$ and $\sf{Disk}_{n}\subset\sf{Mfld}_{n}$
denote the full subcategories spanned by the objects that are homeomorphic
to $\bb R^{n}\times S$ for some finite set $S$. For each $k\geq0$,
we let $\Disk^{\leq k}_{n}\subset\Disk_{n}$ and $\sf{Disk}^{\leq k}_{n}\subset\sf{Disk}_{n}$
denote the full subcategories spanned by the objects with at most
$k$ components.
\item Let $n\geq0$. We let $\Mfld_{\sm,n}$ denote the homotopy coherent
nerve of the topological category of smooth $n$-manifolds and smooth
embeddings, whose mapping spaces $\Emb_{\sm}\pr{-,-}$ are topologized
by the weak topology (also called the compact-open $C^{\infty}$ topology)
of \cite[Chapter 2]{HirschDT}. We let $\Disk_{\sm,n}\subset\Mfld_{\sm,n}$
denote the full subcategory spanned by the smooth manifolds diffeomorphic
to a finite disjoint union of $\bb R^{n}$ (with the standard smooth
structure). Equivalences of $\Mfld_{\sm,n}$ will be called \textbf{smooth
isotopy equivalences}. The $\infty$-categories $\Disk^{\leq k}_{\sm,n}$,
$\sf{Mfld}_{\sm,n}$, $\sf{Disk}_{\sm,n}$, and $\sf{Disk}^{\leq k}_{\sm,n}$
are defined as in the previous point. (We will not define a PL version
of these $\infty$-categories because the author is not aware of a
good topology on PL embeddings whose paths correspond to PL isotopies.)
\item Let $n,k\geq0$. If $M$ is an $n$-manifold, we will write $\Disj\pr M\subset\Open\pr M$
for the full subcategory spanned by the open sets homeomorphic to
$\bb R^{n}\times S$ for some finite set $S$, and let $\Disj^{\leq k}\pr M$
denote its full subposet spanned by the open sets with at most $k$
components. If $M$ is smooth, we let $\Disj_{\sm}\pr M\subset\Disj\pr M$
denote the full subcategory spanned by the open sets diffeomorphic
to $\bb R^{n}\times S$ for some finite set $S$ (with the standard
smooth structure), and write $\Disj^{\leq k}_{\sm}\pr M=\Disj_{\sm}\pr M\cap\Disj^{\leq k}\pr M$.
\item Let $n\ge0$ and let $M$ be an $n$-manifold. For each finite set
$S$, we let $\Conf\pr{S,M}\subset M^{S}$ denote the subspace consisting
of the injections $S\to M$. For each $k\geq0$, we will write $\Conf\pr{k,M}=\Conf\pr{\{1,\dots,k\},M}$.
(Thus $\Conf\pr{0,M}$ is a point.) The $k$th symmetric group $\Sigma_{k}$
acts on $\Conf\pr{k,M}$ by precomposition. We will write $B_{k}\pr M=\Conf\pr{k,M}/\Sigma_{k}$
for the orbit space of this action. The points of $B_{k}\pr M$ will
be identified with subsets of $M$ of cardinality $k$.
\item Given a category $\cal C$ and a functor $F:\cal C\to\SS$ carrying
each object to an $\infty$-category, we denote its \textbf{relative
nerve} \cite[$\S$ 3.2.5]{HTT} by $\int F=\int^{\cal C}F=\int^{C\in\cal C}F\pr C$.
Recall that the projection $\int F\to\cal C$ is the cocartesian fibration
associated to the functor $F:\cal C\to\Cat_{\infty}$.
\item For each finite set $S$, we will write $\cal P\pr S$ for the poset
of subsets of $S$, and $\cal P_{0}\pr S$ for the poset of nonempty
subsets of $S$. 
\end{itemize}

\section{\label{sec:poly}Polynomial Functors}

In this section, we will show that polynomial approximations exist
and that they are constructed by Kan extending along $\Disj^{\leq k}\pr M\to\Open\pr M$
(Theorem \ref{thm:best_approx_exists}).
\begin{warning}
\label{warn:op}To avoid writing ``$\op$'' repeatedly, we will
work with \textit{co}variant functors in this section, so that polynomial
functors are described in terms of \textit{co}sheaves and \textit{left}
Kan extensions. However, in later section, we sometimes work with
the dual theory of presheaves and right Kan extensions (such as Theorems
\ref{thm:analyticity} and \ref{thm:inv}). The choice will be dictated
by the ease it takes to state and prove the relevant results.
\end{warning}

Let us start by stating the main result of this section (Theorem \ref{thm:best_approx_exists}).
\begin{defn}
\label{def:istp_exh_poly}Let $k\geq0$, let $M$ be a manifold, and
let $\cal C$ be an $\infty$-category with finite colimits. Let $F:\Open\pr M\to\cal C$
be a functor.
\begin{enumerate}
\item Let $\cal J\subset\Open\pr M$ be a subcategory, and let $G:\cal J\to\cal C$
be a functor. We say that $G$ is \textbf{isotopy invariant} if it
carries each morphism in $\cal J$ which is an isotopy equivalence
to an equivalence of $\cal C$. Isotopy invariant functors will be
called \textbf{isotopy functors}. We will write $\Fun_{\istp}\pr{\cal J,\cal C}\subset\Fun\pr{\cal J,\cal C}$
for the full subcategory spanned by the isotopy functors.
\item We say that $F$ is\textbf{ exhaustive} if for each increasing sequence
$U_{0}\subset U_{1}\subset\cdots$ of open sets of $M$, the map
\[
\colim_{i}F\pr{U_{i}}\to F\pr{\bigcup_{i\geq0}U_{i}}
\]
is an equivalence.
\item We say that $F$ is \textbf{$k$-excisive}, or \textbf{polynomial
of degree $\leq k$}, if for each open set $U$ of $M$ and each pairwise
disjoint closed subsets $A_{0},\dots,A_{k}$ of $U$, the map
\[
\colim_{\emptyset\neq S\subset\{0,\dots,k\}}F\pr{U\setminus\bigcup_{i\in S}A_{i}}\to F\pr U
\]
is an equivalence. 
\item We let $\Exc^{k}_{\istp,\exh}\pr{M;\cal C}\subset\Fun\pr{\Open\pr M,\cal C}$
denote the full subcategory spanned by the $k$-excisive, exhaustive
isotopy functors. 
\end{enumerate}
Here is the main result of this section.
\end{defn}

\begin{thm}
\label{thm:best_approx_exists}Let $k\geq0$, let $M$ be a manifold,
and let $\cal C$ be an $\infty$-category with small colimits. The
left Kan extension functor $\Fun\pr{\Disj^{\leq k}\pr M,\cal C}\to\Fun\pr{\Open\pr M,\cal C}$
restricts to a categorical equivalence 
\[
\Fun_{\istp}\pr{\Disj^{\leq k}\pr M,\cal C}\xrightarrow{\simeq}\Exc^{k}_{\istp,\,\exh}\pr{M;\cal C}.
\]
In particular, the inclusion
\[
\Exc^{k}_{\istp,\exh}\pr{M;\cal C}\hookrightarrow\Fun_{\istp}\pr{\Open\pr M,\cal C}
\]
admits a right adjoint, which carries each object $F\in\Fun_{\istp}\pr{\Open\pr M,\cal C}$
to a left Kan extension of $F\vert\Disj^{\leq k}\pr M$.
\end{thm}

We will deduce Theorem \ref{thm:best_approx_exists} from two results:
The first one is the following lemma, which is a consequence of Ayala--Francis's
localization theorem (Theorem \ref{thm:localizing_wrt_istpy}). Recall
that the theorem says that the functor $\Disj^{\leq k}\pr M\to\Disk^{\leq k}_{n/M}$
is a localization.
\begin{lem}
\label{lem:istp_Kan}Let $k\geq0$, let $M$ be a manifold, let $\cal C$
be an $\infty$-category with small colimits, and let $F:N\pr{\Open\pr M}\to\cal C$
be a functor which is a left Kan extension of $F\vert\Disj^{\leq k}\pr M$.
If $F\vert\Disj^{\leq k}\pr M$ is an isotopy functor, so is $F$.
\end{lem}

\begin{proof}
Since localizations are final \cite[\href{https://kerodon.net/tag/02N9}{Tag 02N9}]{kerodon},
Theorem \ref{thm:localizing_wrt_istpy} gives an equivalence
\[
F\pr U\simeq\colim_{V\in\Disj^{\leq k}\pr U}F\pr V\xrightarrow{\simeq}\colim_{V\in\Disk^{\leq k}_{n/U}}F\pr V
\]
natural in $U\in\Open\pr M$. The right-hand side is manifestly isotopy
invariant, and we are done.
\end{proof}

The second one concerns the identification of polynomial functors
with functors enjoying certain gluing properties. The gluing properties
are best expressed using cosheaves with respect to Grothendieck topologies
on $\Open\pr M$, which we now introduce.

\begin{defn}
Let $\cal C$ be an $\infty$-category, and let $\cal A$ be another
$\infty$-category equipped with a Grothendieck topology \cite[Definition 6.2.2.1]{HTT}.
We will say that a functor $F:\cal A\to\cal C$ is a \textbf{cosheaf}
(with respect to the Grothendieck topology) if for each object $A\in\cal A$
and each covering sieve $\cal A^{0}_{/A}\subset\cal A_{/A}$ of $A$,
the composite
\[
\pr{\cal A^{0}_{/A}}^{\rcone}\to\cal A\xrightarrow{F}\cal C
\]
is a colimit diagram.
\end{defn}

\begin{defn}
\label{def:Weiss_cover}Let $M$ be a manifold, and let $k\geq0$.
A subset $\cal U\subset\Open\pr M$ of $M$ is called a \textbf{Weiss
$k$-cover} if for each subset $S\subset M$ of cardinality $\leq k$,
there is an element of $\cal U$ which contains $S$. The \textbf{Weiss
$k$-topology} on $M$ is the Grothendieck topology on $\Open\pr M$
whose covering sieves of each open set $U\in\Open\pr M$ are precisely
the Weiss $k$-covers $\cal U$ of $U$ that are also sieves on $U$.
If $\cal C$ is an $\infty$-category, a functor $F:\Open\pr M\to\cal C$
is called a \textbf{Weiss $k$-cosheaf} if it is a cosheaf with respect
to the Weiss $k$-topology. The \textbf{Weiss topology }on $M$ is
the intersection of the Weiss $k$-topologies as $k$ ranges over
all nonnegative integers. A \textbf{Weiss cosheaf} is a cosheaf with
respect to the Weiss topology.
\end{defn}

\begin{example}
\label{exa:Weiss_k-cosheaf}Let $M$ be a manifold, and let $\cal C$
be an $\infty$-category. 
\begin{enumerate}
\item Any nonempty collection of open sets of $M$ is a Weiss $0$-cover
of $X$. Consequently, a functor $F:\Open\pr M\to\cal C$ is a Weiss
$0$-cosheaf if and only if it is \textbf{essentially constant}, i.e.,
it factors through a contractible Kan complex. Since $\Open\pr M$
is already weakly contractible, this is equivalent to the requirement
that $F$ carries all morphisms to equivalences.
\item A functor $F:\Open\pr M\to\cal C$ is a Weiss $1$-cosheaf if and
only if for each nonempty sieve $\cal U\subset\Open\pr M$, the map
\[
\colim_{U\in\cal U}F\pr U\to F\pr{\bigcup_{U\in\cal U}U}
\]
is an equivalence. Thus a Weiss $1$-cosheaf is a cosheaf on $X$
(that is, a cosheaf with respect to the standard Grothendieck topology
on $\Open\pr M$) if and only if $F\pr{\emptyset}$ is an initial
object.
\end{enumerate}
\end{example}

\begin{defn}
\label{def:Weiss_condition}Let $k\geq0$, let $M$ be a manifold,
and let $\chi:\cal I\to\Open\pr M$ be a functor of small $\infty$-categories.
For each finite set $S\subset M$, let $\cal I_{S}\subset\cal I$
denote the full subcategory spanned by the objects $I\in\cal I$ such
that $S\subset\chi\pr I$.
\begin{enumerate}
\item We say that $\chi$ satisfies the \textbf{Weiss $k$-condition} if
for each subset $S\subset M$ of cardinality at most $k$, the $\infty$-category
$\cal I_{S}$ is weakly contractible.
\item We say that $\chi$ satisfies the \textbf{Weiss condition} if for
every finite subset $S\subset M$, the $\infty$-category $\cal I_{S}$
is weakly contractible.
\end{enumerate}
\end{defn}

We now arrive at the identification result of polynomial functors
with cosheaves and Kan extensions, which we prove at the end of this
section.
\begin{thm}
\label{thm:Weiss_k-cosheaf}Let $k\geq0$, let $M$ be a manifold,
let $\cal C$ be an $\infty$-category with small colimits, and let
$F:\Open\pr M\to\cal C$ be an isotopy functor. The following conditions
are equivalent:
\begin{enumerate}
\item The functor $F$ is a left Kan extension of $F\vert\Disj^{\leq k}\pr M$.
\item Let $\cal I$ be a small $\infty$-category, let $U\subset M$ be
an open set, and let $\chi:\cal I\to\Open\pr U$ be a functor satisfying
the Weiss $k$-condition. Then the map
\[
\colim_{I\in\cal I}F\pr{\chi\pr I}\to F\pr U
\]
is an equivalence of $\cal C$.
\item The functor $F$ is a Weiss $k$-cosheaf.
\item The functor $F$ is $k$-excisive and exhaustive.
\end{enumerate}
\end{thm}

For later discussions, we will also prove the following limit case
of Theorem \ref{thm:Weiss_k-cosheaf}:
\begin{thm}
\label{thm:Weiss_cosheaf}Let $M$ be a manifold, let $\cal C$ be
an $\infty$-category with small colimits, and let $F:\Open\pr M\to\cal C$
be an isotopy functor. The following conditions are equivalent:
\begin{enumerate}
\item The functor $F$ is a left Kan extension of $F\vert\Disj\pr M$.
\item Let $\cal I$ be a small $\infty$-category, let $U\subset M$ be
an open set, and let $\chi:\cal I\to\Open\pr U$ be a functor which
satisfies the Weiss condition. Then the map
\[
\colim_{I\in\cal I}F\pr{\chi\pr I}\to F\pr U
\]
is an equivalence of $\cal C$.
\item The functor $F$ is a Weiss cosheaf.
\end{enumerate}
\end{thm}

Using Lemma \ref{lem:istp_Kan} and Theorem \ref{thm:Weiss_k-cosheaf},
we can prove Theorem \ref{thm:best_approx_exists} as follows:
\begin{proof}
[Proof of Theorem \ref{thm:best_approx_exists}]By Lemma \ref{lem:istp_Kan},
the left Kan extension functor restricts to a left adjoint
\[
L:\Fun_{\istp}\pr{\Disj^{\leq k}\pr M,\cal C}\to\Fun_{\istp}\pr{\Open\pr M,\cal C},
\]
which is fully faithful by \cite[Proposition 4.3.2.15]{HTT}. By Theorem
\ref{thm:Weiss_k-cosheaf}, the essential image of $L$ is $\Exc^{k}_{\istp,\exh}\pr{M;\cal C}$.
The claim follows.
\end{proof}

We now turn to the proofs of Theorems \ref{thm:Weiss_k-cosheaf} and
\ref{thm:Weiss_cosheaf}. We need a few lemmas.

The first lemma gives a sufficient condition for a functor into $\Disk_{n/M}$
(or $\Disk^{\leq k}_{n/M}$ to be final.
\begin{lem}
\label{lem:final_slice_mfd}Let $n,k\geq0$, let $\cal I$ be a (small)
category and let $\overline{f}:\cal I^{\rcone}\to\Mfld_{n}$ be a
functor. Set $\overline{f}\pr I=U_{I}$ for $I\in\cal I$ and $\overline{f}\pr{\infty}=M$.
For each subset $S\subset M$, let $\cal I_{S}\subset\cal I$ denote
the full subcategory spanned by the objects $I\in\cal I$ such that
$U_{I}$ contains $S$.
\begin{enumerate}
\item Suppose that $U_{I}\in\Disk_{n}$ for every $I\in\cal I$. If $\cal I_{S}$
is weakly contractible for every finite set $S\subset M$, then the
functor $\cal I\to\Disk_{n}{}_{/M}$ is final.
\item Suppose that $U_{I}\in\Disk^{\leq k}_{n}$ for every $I\in\cal I$.
If $\cal I_{S}$ is weakly contractible for every finite set $S\subset M$
of cardinality $\leq k$, then the functor $\cal I\to\Disk^{\leq k}_{n/M}$
is final.
\end{enumerate}
\end{lem}

\begin{proof}
We prove part (1); the proof of part (2) is similar. By Proposition
\ref{prop:final_map_into_slice} and \cite[Theorem 4.2.4.1]{HTT},
it will suffice to show that, for each object $V\in\Disk_{n}$, the
map
\[
\hocolim_{I\in\cal I}\Sing\Emb\pr{V,U_{I}}\to\Sing\Emb\pr{V,M}
\]
is a weak homotopy equivalence of simplicial sets. If $V=\emptyset$,
the claim is obvious because $\cal I=\cal I_{\emptyset}$ is weakly
contractible by hypothesis. So suppose that $V$ is nonempty. Let
$p$ denote the cardinality of the set of components of $V$, and
fix a homeomorphism $V\cong\bb R^{n}\times\{1,\dots,p\}$. By Proposition
\ref{prop:emb_conf_cart}, the evaluation at the origin of $\bb R^{n}$
determines a homotopy cartesian square 
\[\begin{tikzcd}
	{\operatorname{Sing}\operatorname{Emb}(V,U_I)} & {\operatorname{Sing}\operatorname{Emb}(V,M)} \\
	{\operatorname{Sing}\operatorname{Conf}(p,U_I)} & {\operatorname{Sing}\operatorname{Conf}(p,M)}
	\arrow[from=2-1, to=2-2]
	\arrow[from=1-1, to=1-2]
	\arrow[from=1-1, to=2-1]
	\arrow[from=1-2, to=2-2]
\end{tikzcd}\]for every $I\in\cal I$. Since colimits in $\cal S$ are universal,
we are reduced to showing that the map
\[
\hocolim_{I\in\cal I}\Sing\Conf\pr{p,U_{I}}\to\Sing\Conf\pr{p,M}
\]
is a weak homotopy equivalence. This is a direct consequence of our
hypothesis and \cite[Theorem A.3.1]{HA}.
\end{proof}

Our next lemma is completely elementary in the topological case and
the PL case, but we include a proof using Morse theory because we
need it in the smooth version of Lemma \ref{lem:MV}.
\begin{lem}
\label{lem:removing_disks}Let $n\geq1$, $k\geq0$, and let $-1<c_{0}<\cdots<c_{k}<1$
and $r>0$ be real numbers. Set
\[
C_{i}=\{x\in\bb R^{n}\mid\pr{x_{1}-c_{i}}^{2}+\sum_{1<i\leq n}x^{2}_{i}\leq r^{2}\},
\]
and suppose that the sets $C_{0},\dots,C_{k}$ are mutually disjoint
and are contained in the interior of $D^{n}$. Then $\bb R^{n}\setminus\opn{Int}\pr{C_{0}\cup\cdots\cup C_{k}}$
is obtained from $\{x\in\bb R^{n}\mid\norm x\geq1\}$ by attaching
an $\pr{n-1}$-handle $k$ times. (See Figure \ref{fig:handle}.)
\end{lem}

\begin{figure}[h]
\includegraphics[scale=0.6]{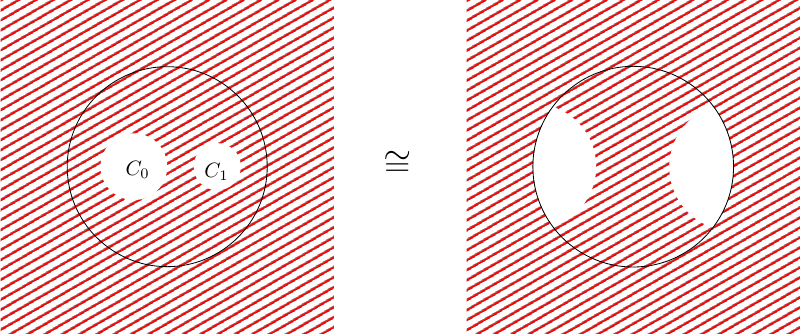}

\caption{\label{fig:handle}Picture of Lemma \ref{lem:removing_disks}.}
\end{figure}

\begin{proof}
Using bump functions, construct a smooth function $\phi:\bb R\to\bb R$
with the following properties (Figure \ref{fig:phi}):
\begin{enumerate}
\item For each $0\leq i\leq k$, the restriction $\phi\vert[c_{i}-r,c_{i}+r]$
is given by $\phi\pr x=-\pr{x-c_{i}}^{2}$. 
\item The derivative of $\phi$ is positive on $\pr{-\infty,c_{1}}$ and
is negative on $\pr{c_{k},\infty}$.
\item The function $\phi$ agrees with $-x^{2}$ outside $[-1,1]$
\item For each $1\leq i\leq k$, there is a unique point $a_{i}\in\pr{c_{i-1},c_{i}}$
such that $\phi'\pr{a_{i}}=0$. 
\item For each $0<i\leq k$, we have $\phi''\pr{a_{i}}>0$.
\item We have $-1<\phi\pr{a_{1}}<\cdots<\phi\pr{a_{k}}$.
\end{enumerate}
\begin{figure}[h]
\includegraphics[scale=0.6]{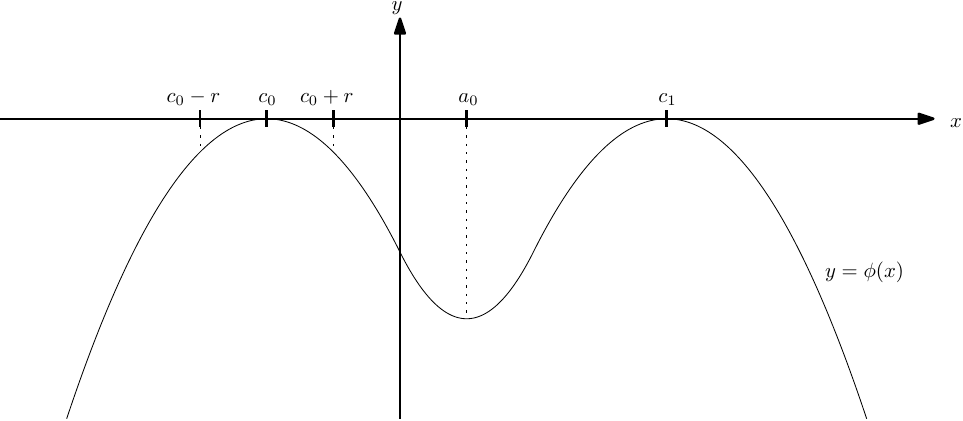}

\caption{\label{fig:phi}Graph of $\phi$.}
\end{figure}

We then define $F:\bb R^{n}\to\bb R$ by $F\pr x=\phi\pr{x_{1}}-\sum_{1<i\leq n}x^{2}_{i}$.
By construction, the critical points of $F$ in $F^{-1}\pr{(-\infty,r^{2}]}$
are the points $\{\pr{a_{i},0,\dots,0}\}_{1\leq i\leq k}$, and all
of them are nondegenerate and have index $n-1$. Therefore, Morse
theory (see, e.g., \cite[Proposition VII.2.2]{Kosinski}) shows that
the set $\bb R^{n}\setminus\opn{Int}\pr{C_{0}\cup\cdots\cup C_{k}}=F^{-1}\pr{(-\infty,-r^{2}]}$
is obtained from $F^{-1}\pr{(-\infty,-1]}=\{x\in\bb R^{n}\mid\norm x\geq1\}$
by attaching $\pr{n-1}$-handles $k$ times. 
\end{proof}

The next lemma is principle for a variant of Mayer--Vietoris argument
(\cite[Chapter 1, $\S$5]{BottTu}) for $k$-excisive functors.
\begin{lem}
\label{lem:MV}Let $n,k\geq0$, and let $M$ be an $n$-manifold.
Let $\cal U$ be a set of open sets of $M$ which satisfies the following
conditions:
\begin{enumerate}
\item The set $\cal U$ contains $\Disj^{\leq k}\pr M$.
\item Let $U_{0}\subset U_{1}\subset\cdots$ be an increasing sequence of
elements in $\cal U$. Then $\bigcup_{i\geq0}U_{i}$ belongs to $\cal U$.
\item Let $U\subset M$ be an open set and let $A_{0},\dots,A_{k}$ be pairwise
disjoint closed sets of $U$. Suppose that, for each nonempty subset
$S\subset\{0,\dots,k\}$, the open set $U\setminus\bigcup_{i\in S}A_{i}$
belongs to $\cal U$. Then $U$ belongs to $\cal U$.
\end{enumerate}
Then $\cal U=\Open\pr M$.
\end{lem}

\begin{proof}
The proof proceeds in several steps.

\begin{enumerate}[label=(\textit{Step\arabic*}), wide, itemsep = 1em]

\item We show that $\cal U$ contains $\Disj\pr M$. It suffices
to show that $\Disj^{\leq p}\pr M\subset\cal U$ for every $p\geq0$.
We prove this by induction on $p$. If $p\leq k$, the claim follows
from our assumption (1). Suppose we have proved that $\Disj^{\leq p-1}\pr M\subset\cal U$
for some $p>k$, and let $U\in\Disj^{\leq p}\pr M$. We wish to show
that $U\in\cal U$. Let $A_{0},\dots,A_{k}$ be distinct components
of $U$. (There may be other components, but choose $k+1$ one of
them.) For every nonempty subset $S\subset\{0,\dots,k\}$, the open
set $U\setminus\bigcup_{i\in S}A_{i}$ belongs to $\Disj^{\leq p-1}\pr M$
and hence to $\cal U$. It follows from the induction hypothesis and
condition (3) that $U\in\cal U$, completing the induction.

\item Let $U\subset M$ be an open set which has the form $U=\opn{Int}N$
for some compact manifold $N$ with boundary admitting a handle decomposition.
We will show that $U$ belongs to $\cal U$. 

Let $\pr{a_{0},\dots,a_{n}}\in\bb Z^{n+1}_{\geq0}$ be an $\pr{n+1}$-tuple
of nonnegative integers. We say that a handle decomposition $\emptyset=N_{0}\subset\cdots\subset N_{\sum^{n}_{i=0}a_{i}}=N$
of $N$ is of \textbf{type} $\pr{a_{0},\dots,a_{n}}$ if for each
$0\leq d\leq n$, there are exactly $a_{d}$ integers $i$ such that
$N_{i-1}$ is obtained from $N_{i}$ by attaching a $d$-handle. We
define the \textbf{handle type} of $N$ to be the minimal element
$\pr{a_{0},\dots,a_{n}}\in\bb Z^{n+1}_{\geq0}$ for which $P$ admits
a handle decomposition of type $\pr{a_{0},\dots,a_{n}}$, where $\bb Z^{n+1}_{\geq0}$
is endowed with the lexicographic ordering read from right to left.
(In other words, given distinct elements $\pr{a_{0},\dots,a_{n}},\pr{b_{0},\dots,b_{n}}\in\bb Z^{n+1}_{\geq0}$,
we declare that $\pr{b_{0},\dots,b_{n}}<\pr{a_{0},\dots,a_{n}}$ if
and only if $b_{i}<a_{i}$, where $i$ is the maximal integer such
that $a_{i}\neq b_{i}$.) We will show by transfinite induction on
the handle type of $N$ that $U$ belongs to $\cal U$.

Let $\pr{a_{0},\dots,a_{n}}\in\bb Z^{n+1}_{\geq0}$ be the handle
type of $N$. Suppose that the claim has been proved for every element
$\pr{a'_{0},\dots,a'_{n}}\in\bb Z^{n+1}_{\geq0}$ smaller than $\pr{a_{0},\dots,a_{n}}$.
We must show that $U$ belongs to $\cal U$. If $a_{1}=\cdots=a_{n}=0$,
then $U$ belongs to $\Disj\pr M$, and the claim follows from Step
1. So suppose that $a_{i}>0$ for some $i$. Let 
\[
\emptyset=N_{0}\subset\cdots\subset N_{a}=N
\]
be a handle decomposition of $N$ of type $\pr{a_{0},\dots,a_{n}}$,
where we wrote $a=\sum^{n}_{i=0}a_{i}$. Since attaching a $0$-handle
is equivalent to adding a disjoint copy of $D^{n}$, by rearranging
the order of the handle attachment, we may assume that $N_{a}$ is
obtained from $N_{a-1}$ by attaching a handle of positive index,
say $\lambda$. Let $e:D^{\lambda}\times D^{n-\lambda}\to N_{a}$
denote the corresponding embedding. Choose disjoint closed disks $C_{0},\dots,C_{k}\subset\opn{Int}D^{\lambda}$
and set $A_{j}=C_{j}\times D^{n-\lambda}$. If $S\subset\{0,\dots,k\}$
is a nonempty subset, then we have
\begin{align*}
U\setminus\bigcup_{j\in S}A_{j} & =\opn{Int}\pr{N_{a}\setminus\bigcup_{j\in S}e\pr{\opn{Int}\pr{C_{j}}\times D^{n-\lambda}}}\\
 & =\opn{Int}\pr{N_{a-1}\amalg_{S^{\lambda-1}\times D^{n-\lambda}}e\pr{\pr{D^{\lambda}\setminus\bigcup_{j\in S}\opn{Int}\pr{C_{j}}}\times D^{n-\lambda}}}.
\end{align*}
Now according to Lemma \ref{lem:removing_disks}, the manifold with
boundary $N_{a}\pr S=N_{a-1}\amalg_{S^{\lambda-1}\times D^{n-\lambda}}e\pr{\pr{D^{\lambda}\setminus\bigcup_{j\in S}\opn{Int}\pr{C_{j}}}\times D^{n-\lambda}}$
admits a handle decomposition of type $\pr{b_{0},\dots,b_{n}}$, where
\[
b_{i}=\begin{cases}
a_{i} & \text{if }i\neq\lambda,\lambda-1,\\
a_{\lambda-1}+\abs S-1 & \text{if }i=\lambda-1,\\
a_{\lambda}-1 & \text{if }i=\lambda.
\end{cases}
\]
Since $\pr{b_{0},\dots,b_{n}}<\pr{a_{0},\dots,a_{n}}$, the induction
hypothesis implies that $\opn{Int}\pr{N_{a}\pr S}$ belongs to $\cal U$.
It follows from condition (3) that $U$ belongs to $\cal U$, as desired.

\item We show that if $U$ is smoothable or $n\neq4$, then $U$
belongs to $\cal U$. By Step 2 and assumption (2), it suffices to
show that $U$ has a (possibly infinite) handle decomposition. If
$U$ is smoothable, then it has a handle decomposition by Morse theory
\cite[Lemma 5.1.8]{WallDT}. If $n=5$, this is \cite[Theorem 2.3.1]{Quinn82},
and if $n\geq6$, this is \cite[Essay III, Theorem 2.1]{KS_FETM}.
If $n\leq3$, then it suffices to show that $U$ has a smooth structure.
If $n=1$, this follows from the classification of $1$-manifolds
\cite[Theorem 5.27]{LeeITM}. If $n=2$, this is \cite[Theorem A]{Hatcher22}.
If $n=3$, then by \cite[Theorem 3.10.8]{Thurston97}, it suffices
to show that $U$ has a PL structure. This is \cite[Theorem 2]{Hamilton76}.

\item We finish off the proof by considering the case where $n=4$.
Let $U\subset M$ be a nonempty finite set. For each path component
$V\subset U$, choose $\pr{k+1}$ points $p^{V}_{0},\dots,p^{V}_{k}\in V$,
and set $A_{j}=\{p^{V}_{j}\mid V\in\pi_{0}\pr U\}$ for $0\leq j\leq k$.
Then $A_{0},\dots,A_{k}$ are disjoint closed subsets of $U$, and
moreover for each nonempty finite set $S\subset\{0,\dots,k\}$, the
components of the open set $U\setminus\bigcup_{i\in S}A_{i}$ are
noncompact. Since connected, non-compact $4$-manifolds are smoothable
\cite[8.2]{FQ90}, we deduce from Step 2 that $\bigcap_{i\in S}U\setminus A_{i}$
belongs to $\cal U$. Condition (3) now shows that $U$ belongs to
$\cal U$, and the proof is complete.\qedhere

\end{enumerate}
\end{proof}

\begin{rem}
Our proof of Lemma \ref{lem:MV} is a modification of Weiss's proof
of \cite[Theorem 5.1]{Weiss99}, where he essentially established
the lemma in the case where $M$ is smoothable. The topological case
is a bit harder because we do not have a control of the order of the
indices of handles.
\end{rem}

\begin{rem}
\label{rem:MV}We have stated Lemma \ref{lem:MV} to highlight its
relation to exhaustivity and excisivity of functors. However, a closer
inspection of the proof actually shows that we can replace condition
(2) as follows:
\begin{itemize}
\item [(2$'$)]Let $U_{0}\subset U_{1}\subset\cdots$ be an increasing sequence
of elements in $\cal U$, where each $U_{i}$ is has compact closure
in $U$. Then $\bigcup_{i\geq0}U_{i}$ belongs to $\cal U$.
\end{itemize}
If $M$ is a PL manifold and $\Disj\pr M\subset\cal U$, then we can
also replace condition (3) by a similar condition in which the $A_{i}$'s
are compact subcomplexes of $U$ for some triangulation of $U$ coming
from the PL structure, by using a handle decomposition associated
to the triangulation. (See \cite[6.9]{RS72}.)

\end{rem}

We now arrive at the proof of Theorems \ref{thm:Weiss_k-cosheaf}
and \ref{thm:Weiss_cosheaf}.
\begin{proof}
[Proof of Theorem \ref{thm:Weiss_k-cosheaf}]Clearly (2)$\implies$(3)$\implies$(4).
It will therefore suffice to show that (1)$\implies$(2) and (4)$\implies$(1).

First we show that (1)$\implies$(2). Suppose that $F$ is a left
Kan extension of $F\vert\Disj^{\leq k}\pr M$. Let $\cal I'$ be a
small $\infty$-category, let $U\subset M$ be an open set, and let
$\chi:\cal I'\to\Open\pr U$ be a functor satisfying the Weiss $k$-condition.
For each $I\in\cal I'$, set $U_{I}=\chi\pr I$. We must show that
the map 
\[
\theta:\colim_{I\in\cal I'}F\pr{U_{I}}\to F\pr U
\]
is an equivalence of $\cal C$. According to \cite[\href{https://kerodon.net/tag/02MD}{Tag 02MD}]{kerodon}
and \cite[\href{https://kerodon.net/tag/02N9}{Tag 02N9}]{kerodon},
there is a small category $\cal I$ and a final functor $\cal I\to\cal I'$
whose pullbacks are all final. Then the composite $\chi\circ p$ satisfies
the Weiss $k$-condition, so it suffices to consider the case where
$\cal I'=\cal I$.

Set $n=\dim M$. Using Theorem \ref{thm:localizing_wrt_istpy} (and
the fact that localizations are final), we can identify $\theta$
with the map
\[
\colim_{\pr{I,V}\in\int^{I\in\cal I}\Disk^{\leq k}_{n/U_{I}}}F\pr V\to\colim_{V\in\Disk^{\leq k}_{n/U}}F\pr V
\]
induced by the functor $\phi:\int^{I\in\cal I}\Disk^{\leq k}_{n/U_{I}}\to\Disk^{\leq k}_{n/U}$.
It will therefore suffice to show that $\phi$ is final. For this,
we observe that the map $\psi:\int^{I\in\cal I}\Disj^{\leq k}\pr{U_{I}}\to\int^{I\in\cal I}\Disk^{\leq k}_{n/U_{I}}$
is final (Proposition \ref{prop:fiberwise_finality}). Since final
maps have the right cancellation property \cite[Corollary 4.1.9]{CisinskiHCHA},
it suffices to show that the composite $\phi\circ\psi$ is final.
We prove this by using Lemma \ref{lem:final_slice_mfd}: We must show
that, for each finite set $S\subset U$, the category $\int^{I\in\cal I_{S}}\Disj^{\leq k}\pr{U_{I}}_{S}$
is final. For this, observe that the projection $\int^{I\in\cal I_{S}}\Disj^{\leq k}\pr{U_{I}}_{S}\to\cal I_{S}$
is a cocartesian fibration whose fibers are weakly contractible (as
they are cofiltered), and that its base is also weakly contractible
by hypothesis. It follows from Quillen's theorem B (Proposition \ref{prop:Theorem_B})
that $\int^{I\in\cal I_{S}}\Disj^{\leq k}\pr{U_{I}}_{S}$ is also
weakly contractible, as claimed.

We now complete the proof by showing that (4)$\implies$(1). Let $G$
be a left Kan extension of $F\vert\Disj\pr M$, and let $\alpha:G\to F$
be a natural transformation which extends the identity natural transformation
of $F\vert\Disj\pr M$. We must show that $\alpha$ is a natural equivalence. 

Call an open set $U\subset M$ \textbf{good} if the map $\alpha_{U}$
is an equivalence. We wish to show that every open set of $M$ is
good. Since $F$ and $G$ are both $k$-excisive and exhaustive (for
we already know that (1)$\implies$(4)), the collection of good open
subsets of $M$ satisfies the hypotheses of Lemma \ref{lem:MV}. Therefore,
every open set of $M$ is good, and we are done.
\end{proof}

\begin{proof}
[Proof of Theorem \ref{thm:Weiss_cosheaf}]Clearly (2)$\implies$(3).
Also, we can prove (1)$\implies$(2) in exactly the same way as we
proved the implication (1)$\implies$(2) of Theorem \ref{thm:Weiss_k-cosheaf}.
It will therefore suffice to show that (3)$\implies$(1).

Suppose that $F$ is a Weiss cosheaf. Let $G$ be a left Kan extension
of $F\vert\Disj\pr M$. Note that $G$ is a Weiss cosheaf because
we have already shown that (1)$\implies$(3). Let $\alpha:G\to F$
be a natural transformation which extends the identity natural transformation
of $F\vert\Disj\pr M$. We must show that $\alpha$ is a natural equivalence. 

Call an open set $U\subset M$ \textbf{good} if the map $\alpha_{U}$
is an equivalence. We wish to show that every open set of $M$ is
good. We prove this in two steps.

Let $U\subset M$ be an open set homeomorphic to an open subset of
$\bb R^{n}$, where $n=\dim M$. We claim that $U$ is good. Choose
an embedding $U\hookrightarrow\bb R^{n}$ and identify $U$ with a
subset of $\bb R^{n}$. Let $\opn{Cube}\pr U$ denote the poset of
open subsets of $U$ that are finite disjoint unions of open cubes
lying in $U$, and let $\overline{\opn{Cube}\pr U}$ denote the poset
of open subsets of elements of $\opn{Cube}\pr U$. Since $\opn{Cube}\pr U$
is closed under finite intersection, the inclusion $\opn{Cube}\pr U\to\overline{\opn{Cube}\pr U}$
is final. Moreover, $\overline{\opn{Cube}\pr U}$ is a covering sieve
of $U$ in the Weiss topology on $\Open\pr M$. Since every element
of $\opn{Cube}\pr U$ is good, it follows that $\alpha_{U}$ is an
equivalence.

Next, let $U\subset M$ be an arbitrary subset. We claim that $U$
is good. Let $\cal U$ denote the set of all open sets of $U$ that
are homeomorphic to an open subset of $\bb R^{n}$. By the result
in the previous paragraph, every element in $\cal U$ is good. Moreover,
$\cal U$ is a covering sieve of $U$ in the Weiss topology. It follows
that $\alpha_{U}$ is an equivalence, and the proof is complete.
\end{proof}

\begin{rem}
Using Variant \ref{var:o-small}, we can show that everything in this
section is valid in the smooth case if we make the following replacements:
\begin{itemize}
\item manifolds by smooth manifolds;
\item isotopy equivalences by smooth isotopy equivalences;
\item $\Disj\pr -$ and $\Disj^{\leq k}\pr -$ by $\Disj_{\sm}\pr -$ and
$\Disj^{\leq k}_{\sm}\pr -$;
\end{itemize}
A similar remark applies to the PL case.
\end{rem}

\begin{rem}
Let $M$ be a manifold, and let $\cal C$ be a locally small $\infty$-category
with small colimits. Isotopy invariance, exhaustivity, and $k$-excisivity
of a functor $F:\Open\pr M\to\cal C$ can be detected jointly by the
hom-functors (i.e., by the composites $\Open\pr M^{\op}\xrightarrow{F}\cal C^{\op}\xrightarrow{\cal C\pr{-,C}}\cal S$,
where $C$ ranges over the objects in $\cal C$). Therefore, some
of the key results of this section (like Lemma \ref{lem:istp_Kan})
could have been proved by outsourcing everything to the case of space-valued
presheaves, covered in \cite{Weiss99}. We have opted for the current
approach because we believe that it is more conceptual, and also because
it clarifies the connection between polynomials and Weiss sheaves.
\end{rem}

\section{\label{sec:taylor}Taylor Towers and Their Convergence}

In the previous section, we saw that polynomial approximations exist,
and that they are constructed by Kan extensions. In this section,
we use it to construct the Taylor (co)tower of an isotopy functor
and discuss when the tower converges to the original functor.

Let $M$ be a manifold, and let $\cal C$ be an $\infty$-category
with small colimits. Theorem \ref{thm:best_approx_exists} says that
for each isotopy functor $F:\Open\pr M\to\cal C$ and each $k\geq0$,
there is a best approximation of $F$ by $k$-excisive, exhaustive
isotopy functor, namely, the left Kan extension $T_{k}F$ of $F\vert\Disj^{\leq k}\pr M$.
This is called the \textbf{$k$th polynomial approximation} of $F$.
We now organize these approximations into a cotower
\[
T_{0}F\to T_{1}F\to\cdots\to F.
\]

For later discussions, we work in a slightly more general setting.
\begin{defn}
\label{def:Taylor}Let $\cal A$ be an $\infty$-category equipped
with a sequence $\cal A^{0}\subset\cal A^{1}\subset\cdots$ of full
subcategories. Set $\cal A^{\infty}=\bigcup_{i}\cal A^{i}$. We let
$\bb T\pr{\cal A}\subset\pr{\bb Z_{\geq0}\star\{\infty\}}\times\cal A$
denote the full subcategory spanned by the objects $\pr{i,A}$, where
$i\in\bb Z_{\geq0}$ and $A\in\cal A_{i}$, together with the objects
$\pr{\infty,A}$, where $A\in\cal A$. 

Let $\cal C$ be another $\infty$-category with small colimits, and
let $F:\cal A\to\cal C$ be a functor. A functor $\bb T\pr F:\pr{\bb Z_{\geq0}\star\{\infty\}}\times\cal A\to\cal C$
is called a \textbf{Taylor cotower} of $F$ (with respect to the subcategories
$\{\cal A^{i}\}_{i\geq0}$) if it is a left Kan extension of the composite
\[
\bb T\pr{\cal A}\xrightarrow{\text{projection}}\cal A\xrightarrow{F}\cal C.
\]
For each $k\in\bb Z_{\geq0}\star\{\infty\}$, we set $T_{k}F=\bb T\pr F\vert\{k\}\times\cal A$.
Note that for each $k\in\bb Z_{\geq0}$ and $A\in\cal A$, the inclusion
\[
\{\id_{k}\}\times\cal A^{k}_{/A}\hookrightarrow\bb T\pr{\cal A}_{/\pr{k,A}}
\]
is final (for it is a right adjoint), so the restriction $T_{k}F$
is a left Kan extension of $F\vert\cal A^{k}$. By convention, we
will write $T_{-1}F$ for any functor which carries all objects of
$\cal A$ to an initial object of $\cal C$. 

A Taylor cotower $\bb T\pr F$ of $F$ is said to be\textbf{ convergent}
if the functor $\bb Z_{\geq0}\star\{\infty\}\to\Fun\pr{\cal A,\cal C}$
adjoint to $\bb T\pr F$ is a colimit diagram.
\end{defn}

\begin{example}
\label{exa:taylor}Let $M$ be a manifold, let $\cal C$ be an $\infty$-category
with small colimits, and let $F:\Open\pr M\to\cal C$ be a functor.
A \textbf{Taylor cotower of $F$ }is a Taylor cotower of $F$ with
respect to the subcategories $\{\Disj^{\leq i}\pr M\}_{i\geq0}$.
If $F$ is an isotopy functor, then for each $k\in\bb Z$, the $k$th
level $T_{k}F$ of a Taylor cotower of $F$ is polynomial of degree
$\leq k$ (Theorem \ref{thm:best_approx_exists}).
\end{example}

\begin{warning}
Despite the name, the constituents of Taylor cotowers of \textit{non-}isotopy
functors may \textit{fail} to be polynomial. For example, consider
the functor $F:\Open\pr{\bb R}\to\Set$ defined by $F\pr U=U\times U$.
It is easy to check that $T_{1}F$ is not excisive.
\end{warning}

The proposition below gives the characterization of convergent functors.
\begin{prop}
\label{prop:taylor}Let $\cal A$ be an $\infty$-category, and let
$\cal A_{0}\subset\cal A_{1}\subset\cdots$ be a sequence of full
subcategories of $\cal A$. Set $\cal A_{\infty}=\bigcup_{i\geq0}\cal A_{i}$.
Let $\cal C$ be an $\infty$-category with small colimits, and let
$F:\cal A\to\cal C$ be a functor. The following conditions are equivalent:
\begin{enumerate}
\item The functor $F$ has a convergent Taylor cotower.
\item The functor $F$ is a left Kan extension of $F\vert\cal A_{\infty}$.
\end{enumerate}
\end{prop}

\begin{proof}
We will use the transitivity of Kan extensions \cite[Proposition 4.3.2.8]{HTT}.
Let $T=\bb T\pr F$ be a Taylor cotower of $F$. Since colimits in
functor categories can be formed pointwise \cite[\href{https://kerodon.net/tag/02X9}{Tag 02X9}]{kerodon},
an easy finality argument shows that condition (1) is equivalent to
the following condition:
\begin{itemize}
\item [(1$'$)]$T$ is a left Kan extension of $T\vert\bb Z_{\geq0}\times\cal A$. 
\end{itemize}
Let $\cal X\subset\pr{\bb Z_{\geq0}\star\{\infty\}}\times\cal A$
denote the full subcategory spanned by the objects in $\bb Z_{\geq0}\times\cal A$
and the objects in $\{\infty\}\times\cal A_{\infty}$. Since every
object $A\in\cal A_{\infty}$ belongs to some $\cal A_{i}$, the functor
$T\vert\cal X$ is a left Kan extension of $T\vert\bb Z_{\geq0}\times\cal A$.
Therefore, condition (1$'$) is equivalent to the following condition:
\begin{itemize}
\item [(1$''$)]$T$ is a left Kan extension of $T\vert\cal X$.
\end{itemize}
On the other hand, $T\vert\cal X$ is a left Kan extension of $T\vert\pr{\bb Z_{\geq0}\star\{\infty\}}\times\cal A_{\infty}$,
so (1$''$) is equivalent to the following condition:
\begin{itemize}
\item [(1$'''$)]$T$ is a left Kan extension of $T\vert\pr{\bb Z_{\geq0}\star\{\infty\}}\times\cal A_{\infty}$.
\end{itemize}
Since $T\vert\bb Z_{\geq0}\times\cal A$ is a left Kan extension of
$T\vert\bb Z_{\geq0}\times\cal A_{\infty}$, condition (1$'''$) is
equivalent to condition (2), and we are done.
\end{proof}

\begin{cor}
Let $M$ be a manifold, and let $\cal C$ be an $\infty$-category
with small colimits. An isotopy functor $F:\Open\pr M\to\cal C$ has
a convergent Taylor cotower if and only if $F$ satisfies the equivalent
conditions of Theorem \ref{thm:Weiss_cosheaf}.
\end{cor}

\begin{proof}
This follows from Proposition \ref{prop:taylor} and Theorem \ref{thm:Weiss_cosheaf}.
\end{proof}

When we can talk about connectivity of morphisms, we can give a better
convergence result. We will focus on one such instance, where the
target $\infty$-category is a stable $\infty$-category with t-structures.
(See Section \ref{subsec:t-st} for a brief review of t-structures.)
The following definition follows \cite[Theorem 2.3]{GW99}. 
\begin{defn}
\label{def:analyticity}Let $M$ be a smooth manifold, let $\cal C$
be a stable $\infty$-category with $t$-structures, let $F:\Open\pr M^{\op}\to\cal C$
be an exhaustive isotopy functor, and let $\rho$ and $c$ be integers.
We say that $G$ is \textbf{$\rho$-analytic with excess $c$} if
it satisfies the following condition:
\begin{itemize}
\item [($\ast$)]Let $P\subset M$ be a smooth compact submanifold with
boundary, and let $Q_{0},\dots,Q_{r}\subset M\setminus\Int\pr P$
be pairwise disjoint, compact, codimension $0$ smooth submanifolds
with boundary. Suppose that each $Q_{i}$ has handle index $q_{i}<\rho$
(i.e., it is obtained from $\partial Q_{i}$ by attaching handles
of indices $\leq q_{i}$). Then the $\pr{r+1}$-cube 
\[
\cal P\pr{\{0,\dots,r\}}^{\op}\to\cal C,\,S\mapsto F\pr{\Int\pr{P\cup\bigcup_{s\in S}Q_{s}}}
\]
is $\pr{c+\sum^{r}_{i=0}\pr{\rho-q_{i}}}$-cartesian.
\end{itemize}
\end{defn}

The following result should be compared with \cite[Theorem 2.3]{GW99}.
It also generalizes \cite[Theorem E.5]{RW14}.
\begin{thm}
\label{thm:analyticity}Let $\cal C$ be a stable $\infty$-category
with a t-structure, let $M$ be a smooth manifold, let $\rho$ and
$c$ be integers, and let $F:\Open\pr M^{\op}\to\cal C$ be an exhaustive
isotopy functor. Suppose that the following conditions are satisfied:
\begin{enumerate}
\item $F$ is $\rho$-analytic with excess $c$.
\item $\cal C$ has small limits.
\item The functor $\pi_{0}:\cal C\to\cal C^{\heartsuit}$ preserves countable
products.
\end{enumerate}
Then for every $k\geq1$ and every open set $W\subset M$ admitting
a proper Morse function whose critical points have indices less than
$q$, the map
\[
\eta_{k-1,W}:F\pr W\to T_{k-1}F\pr W
\]
is $\pr{c+k\pr{\rho-q}-1}$-connected.\footnote{This is slightly weaker than the estimate of \cite{GW99}, where it
says $c+k\pr{\rho-q}$ instead of $c+k\pr{\rho-q}-1$. However, a
close inspection of the proof of loc. cit. shows that it only gives
the same estimate as ours. Michael Weiss informed me that this issue
can probably be fixed, but the details are not written down yet. We
believe that his technique will be applicable to our setting too,
so we should be able to say that $\eta_{k-1,W}$ is $c+k\pr{\rho-q}$-connected.} In particular, if $\rho>q$ and $\bigcap_{n}\cal C_{\geq n}=0$,
then the map
\[
F\pr W\to\lim_{k}T_{k}F\pr W
\]
is an equivalence.
\end{thm}

\begin{proof}
We follow \cite[Theorem 2.3]{GW99}. Shifting degrees if necessary,
we may assume that $c=0$. Suppose first that $\overline{W}$ is a
compact smooth submanifold with boundary of $M$ and admits a handle
decomposition with handles of indices at most $q$. We will show that
$\eta_{k-1,W}$ is $k\pr{\rho-q}$-connected by induction on $q$
and the number of $q$-handles.

If $q=0$, then $W$ is the disjoint union of open balls, say $W_{1},\dots,W_{l}$.
If $l\leq k-1$, then $\eta_{k-1,W}$ is an equivalence, and we are
done. If $l\geq k$, it suffices to show that for each $k\leq t\leq l$,
the map
\[
\theta_{t}:T_{t}F\pr W\to T_{t-1}F\pr W
\]
is $t\rho$-connected. For this, for each subset $S\subset\pi_{0}\pr W$,
let $W_{S}\subset W$ denote the union of the elements in $S$. Consider
the following commutative diagram
\[\begin{tikzcd}
	{\coprod_{S\subset \pi_0(W)}\mathcal{P}_0(S)} & {\coprod_{S\subset\pi_0(W),\,|S|=t}\operatorname{Disj}^{\leq t-1}(W_S)} & {\operatorname{Disj}^{\leq t-1}(W)} \\
	{\coprod_{S\subset \pi_0(W)}\mathcal{P}(S)} & {\coprod_{S\subset\pi _0(W),\,|S|=t}\operatorname{Disj}^{\leq t}(W_S)} & {\operatorname{Disj}^{\leq t}(W),}
	\arrow["{\phi_0}"', from=1-1, to=1-2]
	\arrow[from=1-1, to=2-1]
	\arrow[from=1-2, to=1-3]
	\arrow[from=1-2, to=2-2]
	\arrow[from=1-3, to=2-3]
	\arrow["\phi"', from=2-1, to=2-2]
	\arrow[from=2-2, to=2-3]
\end{tikzcd}\]where the map $\phi\vert\cal P\pr S$ is given by $R\mapsto W_{R}$.
The right-hand square is a pushout in $\Cat_{\infty}$ (because the
Joyal model structure is left proper), and the maps $\phi_{0}$ and
$\phi$ are final by inspection. It follows from Remark \ref{rem:dec_lim}
that $\theta_{t}$ is a pullback of the map
\[
\prod_{S\subset\pi_{0}\pr W}G\pr{W_{S}}\to\prod_{S\subset\pi_{0}\pr W}\lim_{R\in\cal P_{0}\pr S}G\pr{W_{R}},
\]
which is $t\rho$-connected by the analyticity assumption (applied
to $P=\emptyset$ and $Q_{i}=\overline{W_{i}}$). Hence $\theta_{t}$
is $t\rho$-connected, as required.

For the inductive step, suppose we can write $\overline{W}=N\cup H$,
where $H\cong D^{q}\times D^{n-q}$ is a $q$-handle and $N$ has
one less $q$-handle than $\overline{W}$. Choose disjoint closed
disks $C_{0},\dots,C_{k}\subset\Int D^{q}$, and set $A_{i}=C_{i}\times D^{n-q}$.
By the induction hypothesis, for each $S\in\cal P_{0}\pr{\{0,\dots,k\}}$,
the map
\[
G\pr{W\setminus\bigcup_{i\in S}A_{i}}\to T_{k-1}\pr{W\setminus\bigcup_{i\in S}A_{i}}
\]
is $k\pr{\rho-q+1}$-connected. It follows from Corollary \ref{cor:GW99_1.22}
that the map
\[
\lim_{S\in\cal P_{0}\pr S}G\pr{W\setminus\bigcup_{i\in S}A_{i}}\to\lim_{S\in\cal P_{0}\pr S}T_{k-1}\pr{W\setminus\bigcup_{i\in S}A_{i}}\simeq T_{k-1}\pr W
\]
is $k\pr{\rho-q}$-connected. Moreover, the cube $S\mapsto G\pr{W\setminus\bigcup_{i\in S}A_{i}}$
is $k\pr{\rho-q}$-cartesian by the analyticity assumption. Hence
the map $G\pr W\to T_{k-1}\pr W$ is $k\pr{\rho-q}$-connected (Proposition
\ref{prop:1.5}), completing the induction.

The general case and the last claim are taken care of by the Milnor
exact sequence (Lemma \ref{lem:Milnor}; note that conditions (2)
and (3) ensure that the hypotheses of the lemma are satisfied), and
we are done.
\end{proof}

\section{\label{sec:homog}Classification of Homogeneous Functors}

In this section, we classify homogeneous functors, the building blocks
of Taylor towers in manifold calculus. More formally, let $k\geq0$,
let $M$ be a manifold, let $\cal C$ be a pointed $\infty$-category
with small colimits, and let $F:\Open\pr M\to\cal C$ be a functor
which is polynomial of degree $\leq k$. We say that $F$ is \textbf{homogeneous}
of degree $k$ if $F$ evaluates to a zero object at each element
$U\in\Disj\pr M$ with fewer than $k$ components; or equivalently,
if the $\pr{k-1}$th level $T_{k-1}F$ of a Taylor cotower of $F$
(Example \ref{exa:taylor}) carries every object to a zero object.
For example, every functor of polynomial of degree $\leq0$ is homogeneous
of degree $0$; a $1$-excisive exhaustive isotopy functor is homogeneous
of degree $1$ if and only if it is a cosheaf (by Theorem \ref{thm:Weiss_k-cosheaf}
and Example \ref{exa:Weiss_k-cosheaf}).

We are interested in homogeneous functors for the following reason:
Suppose that $F$ is an isotopy functor. A cofiber of $T_{k-1}F\to T_{k}F$,
called a \textbf{$k$th homogeneous layer} of $F$, is then an isotopy
functor which is homogeneous of degree $\leq k$: It is an isotopy
functor by Lemma \ref{lem:istp_Kan}; it is polynomial of degree $\leq k$
by Theorem \ref{thm:Weiss_k-cosheaf}; homogeneity is clear. Therefore,
if $F$ has a convergent Taylor cotower, then we can reduce the analysis
of $F$ to those of the homogeneous layers of $F$, at least up to
extension problems.

In this section, we prove that homogeneous functors are classified
by unordered configuration spaces and ``objects of sections trivial
near the fat diagonal'' (Corollary \ref{cor:homog_class} and Theorem
\ref{thm:inv}). To illustrate the utility of the classification,
we also give a connectivity estimate of these approximations (Corollary
\ref{cor:RW14}).
\begin{notation}
Let $M$ be a manifold, and let $\cal C$ be a pointed $\infty$-category
with small colimits. We let $\Homog^{k}_{\istp,\exh}\pr{M;\cal C}$
denote the full subcategory of $\Fun\pr{\Open\pr M,\cal C}$ spanned
by the exhaustive isotopy functors that are homogeneous of degree
$k$.
\end{notation}

Recall that if $M$ is an $n$-manifold, then $I^{=k}_{M}$ denotes
the (non-full) subcategory of $\Disj\pr M$ spanned by the isotopy
equivalences between the elements of $\Disj\pr M$ with exactly $k$
components (Notation \ref{nota:B'_k}). The poset $I^{=k}_{M}$ classifies
homogeneous functors in the following sense:
\begin{prop}
\label{prop:homog}Let $M$ be a manifold, let $\cal C$ be a pointed
$\infty$-category with small colimits, and let $k\geq0$. The inclusion
$I^{=k}_{M}\subset\Open\pr M$ induces a trivial fibration
\[
\theta:\Homog^{k}_{\istp,\exh}\pr{M;\cal C}\xrightarrow{\simeq}\Fun_{\istp}\pr{I^{=k}_{M},\cal C}.
\]
\end{prop}

\begin{proof}
By Theorem \ref{thm:best_approx_exists}, the restriction functor
\[
p:\Exc^{k}_{\istp,\exh}\pr{M;\cal C}\to\Fun_{\istp}\pr{\Disj^{\leq k}\pr M,\cal C}
\]
is a trivial fibration. Let $\Fun'_{\istp}\pr{\Disj^{\leq k}\pr M,\cal C}\subset\Fun_{\istp}\pr{\Disj^{\leq k}\pr M,\cal C}$
denote the full subcategory spanned by the functors which carry every
object in $\Disj^{\leq k-1}\pr M$ to a zero object. The functor
\[
\Homog^{k}_{\istp,\exh}\pr{M;\cal C}\to\Fun'_{\istp}\pr{\Disj^{\leq k}\pr M,\cal C}
\]
is a pullback of $p$, so it is a trivial fibration. On the other
hand, Proposition \ref{prop:zero_extension} shows that the functor
\[
\Fun'_{\istp}\pr{\Disj^{\leq k}\pr M,\cal C}\to\Fun_{\istp}\pr{I^{=k}_{M},\cal C}
\]
is a trivial fibration. Therefore, $\theta$ is a composition of trivial
fibrations and hence is itself a trivial fibration.
\end{proof}

Now if $M$ is a manifold and $k\geq0$, then a functor on $I^{=k}_{M}$
is an isotopy functor if and only if it maps every morphism of $I^{=k}_{M}$
to an equivalence. Therefore, Proposition \ref{prop:homog} says that
homogeneous functors of degree $k$ on a manifold $M$ can be classified
in terms of the classifying space of $I^{=k}_{M}$, i.e., a localization
of $I^{=k}_{M}$ with respect to all morphisms \cite[\href{https://kerodon.net/tag/01MY}{Tag 01MY}]{kerodon}.
The following proposition identifies the homotopy type of this classifying
space.
\begin{prop}
\label{prop:I_M=00003Dk_htpytype}Let $M$ be a manifold. For each
$k\geq0$, there is a weak homotopy equivalence
\[
I^{=k}_{M}\xrightarrow{\simeq}\Sing B_{k}\pr M
\]
which carries an object $U\in I^{=k}_{M}$ to a point $\{p_{1},\dots,p_{k}\}\in B_{k}\pr M$
which intersects every component of $U$.
\end{prop}

\begin{rem}
The identification of the homotopy type of $I^{=k}_{M}$ dates back
at least to \cite[Lemma 3.5]{Weiss99} and has been reproved again
and again (e.g., \cite[Lemma 2.12]{FHTM}). We include a proof for
later reference.
\end{rem}

\begin{proof}
We consider the maps
\[
I^{=k}_{M}\xleftarrow{\pi}\int^{U\in I^{=k}_{M}}\Sing B'_{k}\pr U\xrightarrow{\phi}\Sing B_{k}\pr M,
\]
where $B'_{k}\pr -$ is defined in Notation \ref{nota:B'_k}. The
map $\pi$ is a trivial fibration because it is a left fibration with
contractible fibers. So it has a section $\sigma$, and $\phi\circ\sigma$
carries each element $U\in I^{=k}_{M}$ to a point in $B'_{k}\pr U$.
It will therefore suffice to show that $\phi$ is a weak homotopy
equivalence. According to (the argument of) \cite[Corolary 3.3.4.6]{HTT},
it will suffice to show that the map 
\[
\hocolim_{U\in I^{=k}_{M}}\Sing B'_{k}\pr U\to\Sing B_{k}\pr M
\]
is a weak homotopy equivalence. This follows from \cite[Theorem A.3.1]{HA}.
\end{proof}

Combining Propositions \ref{prop:homog} and \ref{prop:I_M=00003Dk_htpytype},
we obtain:
\begin{cor}
\label{cor:homog_class}Let $M$ be a manifold, let $\cal C$ be a
pointed $\infty$-category with small colimits, and let $k\geq0$.
There is a zig-zag of categorical equivalences
\[
\Homog^{k}_{\istp,\exh}\pr{M;\cal C}\xrightarrow{\simeq}\Fun_{\istp}\pr{I^{=k}_{M},\cal C}\xleftarrow{\simeq}\Fun\pr{\Sing B_{k}\pr M,\cal C}.
\]
\end{cor}

\begin{rem}
\label{rem:interpretation}Propositions \ref{prop:homog} and \ref{prop:I_M=00003Dk_htpytype}
lead to an interesting observation if we apply it to homogeneous layers
of isotopy functors.\footnote{The author learned this interpretation from \cite{Munson10}.}
Let $n\geq0$, let $k\ge1$, let $M$ be an $n$-manifold, let $\cal C$
be a pointed $\infty$-category with small colimits, and let $F:\Open\pr M\to\cal C$
be an isotopy functor. Let $L_{k}F$ denote a cofiber of the natural
transformation $T_{k-1}F\to T_{k}F$ and let $\overline{L}_{k}F:\Sing B_{k}\pr M\to\cal C$
denote the corresponding functor. Let us compute the value of $\overline{L}_{k}F$
at a point$\{p_{1},\dots,p_{k}\}\in B_{k}M$. First, choose pairwise
disjoint open sets $B_{1},\dots,B_{k}\subset M$ that are homeomorphic
to $\bb R^{n}$ and such that $p_{i}\in B_{i}$, and set $U=B_{1}\cup\dots\cup B_{k}$.
Then 
\begin{align*}
\overline{L}_{k}F\pr{\{p_{1},\dots,p_{k}\}} & \simeq L_{k}F\pr U\\
 & \simeq\cofib\pr{T_{k-1}F\pr U\to T_{k}F\pr U}.
\end{align*}
Since $T_{k-1}F$ is $\pr{k-1}$-excisive, the map $T_{k-1}F\pr U\to T_{k}F\pr U$
can be identified with the map
\[
\theta:\colim_{S\subsetneq\{1,\dots,k\}}F\pr{\bigcup_{i\in S}B_{i}}\to F\pr{\bigcup^{k}_{i=1}B_{i}}.
\]
A cofiber of $\theta$ is called a \textbf{total cofiber} of the $k$-dimensional
cubical diagarm $S\mapsto F\pr{\bigcup_{i\in S}B_{i}}$. Total cofibers
admit an interpretation in terms of a derivative: For example if $k=2$,
then the total homotopy cofiber can be written as
\[
\cofib\pr{\cofib\pr{F\pr{\emptyset}\to F\pr{B_{1}}}\to\opn{cofib}\pr{F\pr{B_{2}}\to F\pr{B_{1}\cup B_{2}}}}
\]
which resembles the classical formula
\[
f''\pr 0=\lim_{h_{1},h_{2}\to0}\frac{f\pr{h_{1}+h_{2}}-f\pr{h_{1}}-f\pr{h_{2}}+f\pr 0}{h_{1}h_{2}}
\]
of the second derivative. Therefore, the values of $\overline{L}_{k}F$
are these ``$k$th derivatives'' of $F$, highlighting the connection
with the Taylor expansion in ordinary calculus.
\end{rem}

In the situation of Corollary \ref{cor:homog_class}, it is relatively
easy to pass from $\Homog^{k}_{\istp,\exh}\pr{M;\cal C}$ to $\Fun\pr{\Sing B_{k}\pr M,\cal C}$:
We just have to restrict a homogenous functor to $I^{=k}_{M}$, and
then pass to the localization. We now explain how to go in the reverse
direction (Theorem \ref{thm:inv}).

The starting point is the observation that for a manifold $M$, the
equivalence $\Fun\pr{\Sing B_{k}\pr M,\cal C}\xrightarrow{\simeq}\Fun_{\istp}\pr{I^{=k}_{M},\cal C}$
coming from Proposition \ref{prop:I_M=00003Dk_htpytype} is given
by 
\[
F\mapsto\pr{U\mapsto\colim_{\Sing B'_{k}\pr U}F}.
\]
(This follows from the proof of Proposition \ref{prop:I_M=00003Dk_htpytype}.
Also, we used the informal notation of Remark \ref{rem:02ZM}, applied
to the cocartesian fibration $\int^{U\in I^{=k}_{M}}\Sing B'_{k}\pr U$).
A bit more consideration then leads to the following proposition:
\begin{prop}
\label{prop:inverse}Let $M$ be a manifold, let $\cal C$ be a pointed
$\infty$-category with small colimits, let $k\geq0$, and let $F:\Sing B_{k}\pr M\to\cal C$
be a functor. Define another functor $\widetilde{F}:\Open\pr M\to\cal C$
by 
\[
\widetilde{F}\pr U=\colim_{\Sing B_{k}\pr U}F.
\]
Then:
\begin{enumerate}
\item $\widetilde{F}$ is $k$-excisive, exhaustive, and isotopy invariant.
\item The assignment $F\mapsto L_{k}\widetilde{F}$ determines an inverse
equivalence
\[
\Fun\pr{\Sing B_{k}\pr M,\cal C}\xrightarrow{\simeq}\Homog^{k}_{\istp,\exh}\pr{M;\cal C}
\]
of the equivalence of Proposition \ref{prop:homog}.
\end{enumerate}
\end{prop}

We need a few lemmas for the proof of Proposition \ref{prop:inverse}.
\begin{notation}
\label{nota:Bdubprime}Let $M$ be a manifold and let $k\geq0$. We
write $B''_{k}\pr M$ for the complement of $B'_{k}\pr M$ in $B_{k}\pr M$.
That is, $B''_{k}\pr M$ is the subspace of $B_{k}\pr M$ consisting
of the points $\{p_{1},\dots,p_{k}\}$ such that the map $\{p_{1},\dots,p_{k}\}\to\pi_{0}\pr M$
is \textit{not} injective.
\end{notation}

\begin{lem}
\label{lem:homog_inv}Let $U$ be a manifold whose components are
homeomorphic to $\bb R^{n}$, and let $k\geq1$. The map
\[
\int^{V\in\Disj^{\leq k-1}\pr U}\Sing B_{k}\pr V\to\Sing\pr{B''_{k}\pr U}
\]
is a weak homotopy equivalence.
\end{lem}

\begin{proof}
By \cite[Lemma 3.3.4.1 and Proposition 3.3.4.2]{HTT}, it suffices
to show that the diagram $\{\Sing B_{k}\pr V\to\Sing B''_{k}\pr U\}_{V\in\Disj^{\leq k-1}\pr U}$
is a homotopy colimit diagram of simplicial sets. According to \cite[Theorem A.3.1]{HTT},
it suffices to show that, for each point $S\in B''_{k}\pr U$, the
full subcategory $\Disj^{\leq k-1}\pr U_{S}\subset\Disj^{\leq k-1}\pr U$
spanned by the objects $V\in\Disj^{\leq k-1}\pr U$ such that $S\subset V$
is weakly contractible. Since $S$ lies in no more than $\pr{k-1}$
components of $U$, such a $V$ necessarily lies in the union of the
components of $U$ containing points in $S$. Thus $\Disj^{\leq k-1}\pr U_{S}$
has a final object, and is in particular weakly contractible.
\end{proof}

\begin{proof}
[Proof of Proposition \ref{prop:inverse}]The claim is trivial if
$k=0$, so we will assume that $k\geq1$ throughout.

We start with (1). By Proposition \ref{prop:dec_colim}, it suffices
to show that the functor
\[
\Sing B_{k}\pr -:\Open\pr M\to\cal S
\]
is $k$-excisive, exhaustive, and isotopy invariant. The first two
follows from \cite[Theorem A.3.1]{HA}, while isotopy invariance is
obvious.\footnote{Incidentally, $B_{k}\pr -$ is one of the ``classic'' polynomial
functors that appeared in the original paper of Weiss; see \cite[Example 2.4]{Weiss99}.} 

For (2), let $U\in I^{=k}_{M}$ and $F\in\Fun\pr{\Sing B_{k}\pr M}$.
Since $B_{k}\pr U$ is the disjoint union of $B'_{k}\pr U$ and $B''_{k}\pr U$,
we have an equivalence
\[
\pr{\colim_{\Sing B'_{k}\pr U}F}\vee\pr{\colim_{\Sing B''_{k}\pr U}F}\xrightarrow{\simeq}F\pr U
\]
natural in $F$ and $U$. (But beware that the summands on the right-hand
side are not functors of $U\in\Open\pr M$, because an inclusion of
open sets of $M$ may fail to be injective on $\pi_{0}$.) Moreover,
Lemma \ref{lem:homog_inv} give us equivalences
\begin{align*}
T_{k-1}\widetilde{F}\pr U & =\colim_{V\in\Disj^{\leq k-1}\pr U}\colim_{\Sing B_{k}\pr V}F\\
 & \simeq\colim_{\int^{V\in\Disj^{\leq k-1}\pr U}\Sing B_{k}\pr V}F\pr W\\
 & \simeq\colim_{\Sing\pr{B''_{k}\pr U}}F.
\end{align*}
Under these equivalences, the map $T_{k-1}\widetilde{F}\pr U\to\widetilde{F}\pr U$
can be identified with the inclusion of the summand
\[
\pr{\colim_{\Sing B''_{k}\pr U}F}\to\pr{\colim_{\Sing B'_{k}\pr U}F}\vee\pr{\colim_{\Sing B''_{k}\pr U}F}.
\]
Hence we obtain equivalences
\[
L_{k}\widetilde{F}\pr U\simeq\cofib\pr{T_{k-1}\widetilde{F}\pr U\to\widetilde{F}\pr U}\simeq\colim_{\Sing B'_{k}\pr U}F
\]
natural in $F$ and $U$. As we observed right before stating Proposition
\ref{prop:inverse}, the right-hand side computes the desired inverse,
so we are done.
\end{proof}

Proposition \ref{prop:inverse} has a refinement when there is a good
supply of open sets of $B_{k}\pr M$ of configurations where two or
more points are ``close.'' To explain this in more detail, we need
a few preliminaries. 
\begin{warning}
Following Warning \ref{warn:op}, we will focus on manifold calculus
of \textit{presheaves} for the remainder of this section. If $M$
is a manifold and $\cal C$ is a pointed $\infty$-category with small
limits, we write $\Homog^{k}_{\istp,\exh}\pr{M^{\op};\cal C}\subset\Fun\pr{\Open\pr M^{\op},\cal C}$
for the full subcategory spanned by the functors that are isotopy
invariant, exhaustive, and homogeneous of degree $k$ when viewed
as a functor $\Open\pr M^{\op}\to\cal C$. 
\end{warning}

\begin{notation}
\label{nota:gamma}Let $\cal C$ be an $\infty$-category with small
limits, let $X$ be a topological space, and let $F:\Sing X\to\cal C$
be a functor. For each subset $A\subset X$, we define $\Gamma\pr{A;F}=\lim_{\Sing A}F.$ 
\end{notation}

\begin{rem}
Notation \ref{nota:gamma} is motivated by the following observation:
In the situation of Notation \ref{nota:gamma}, suppose that $\cal C=\cal S$.
By the straightening--unstraightening equivalence, $F$ corresponds
to a Kan fibration $p:E\to\Sing X$. The naturality of the straightening--unstraightening
correspondence gives an equivalence between $\Gamma\pr{A;F}$ and
the Kan complex $\Fun_{\Sing X}\pr{\Sing A,E}$ of sections of $p$
over $\Sing A$.

We can also interpret $\lim_{\Sing X}F$ as a ``twisted cochain.''
To see this, consider the functors
\[
\Del\xleftarrow{\pi}\Del_{/\Sing X}\xrightarrow{\fin}\Sing X,
\]
where $\fin$ denotes the final vertex functor (dual of \cite[Lemma 3.3.6]{landoo-cat}).
The functor $\fin$ is initial \cite[Theorem 3.3.8]{landoo-cat},
so we have equivalences
\begin{align*}
\lim_{\Sing X}F & \simeq\lim_{\Del_{/\Sing X}}\pr{F\circ\fin}\\
 & \simeq\lim_{\Del}\Ran_{\pi}\pr{F\circ\fin}\\
 & \simeq\lim_{\Del}\pr{[n]\mapsto\prod_{\sigma:\Delta^{n}\to X\in\pr{\Sing X}_{n}}F\pr{\sigma\pr n}}.
\end{align*}
Here the last equivalence is an instance of Lemma \ref{lem:02ZM}.
When $\cal C$ is the $\infty$-category of cochain complexes of abelian
groups, the limit of a cosimplicial cochain complex is computed as
the totalization of the associated bicomplex \cite[Remark 1.20, Example 2.4]{A23c}.
If further $F$ takes values in the category of abelian groups, the
resulting complex is nothing but the singular complex of $X$ with
local coefficients $F$ \cite[VI.2]{WhiteheadAT}.
\end{rem}

\begin{defn}
\label{def:fd}Let $M$ be a manifold, and let $k\geq0$. We define
the\textbf{ $k$-fold fat diagonal} of $M$ by
\[
\blacktriangle_{k}\pr M=\SP^{k}\pr M\setminus B_{k}\pr M,
\]
where $\SP^{k}\pr M=X^{k}/\Sigma_{k}$ denotes the $k$th symmetric
product of $X$. We write $\Nbd\pr{\blacktriangle_{k}\pr M}$ for
the poset of neighborhoods of $\blacktriangle_{k}\pr M$ in $\SP^{k}\pr M$. 

If $M$ is triangulated and its vertices are well-ordered, $\SP^{k}\pr M$
has a preferred triangulation containing $\blacktriangle_{k}\pr M$
as its subcomplex (Example \ref{exa:triangulation}). In this situation,
we write $\Nbd'\pr{\blacktriangle_{k}\pr M}$ for the full subposet
of $\Nbd\pr{\blacktriangle_{k}\pr M}$ consisting of the (underlying
polyhedra of) derived neighborhoods of $\blacktriangle_{k}\pr M$
in $\SP^{k}\pr M$.
\end{defn}

\begin{defn}
\label{def:gammac}Let $\cal C$ be a pointed $\infty$-category with
small limits, let $M$ be a manifold that admits a PL structure, and
let $F:\Sing B_{k}\pr M\to\cal C$ be a functor. We define a functor
$\Gamma'_{\blacktriangle}\pr{B_{k}\pr -;F}:\Open\pr M^{\op}\to\cal C$
by
\[
\Gamma'_{\blacktriangle}\pr{B_{k}\pr U;F}=\colim_{Q\in\Nbd\pr{\blacktriangle_{k}\pr U}^{\op}}\Gamma\pr{Q\cap B_{k}\pr U;F}.
\]

Since $\Nbd\pr{\blacktriangle_{k}\pr U}^{\op}$ is weakly contractible
(as it is filtered), there is a natural transformation 
\[
\Gamma\pr{B_{k}\pr -;F}\simeq\colim_{Q\in\Nbd\pr{\blacktriangle_{k}\pr -}^{\op}}\Gamma\pr{Q\cap B_{k}\pr -;F}\to\Gamma'_{\blacktriangle}\pr{B_{k}\pr -;F}
\]
We write $\Gamma_{\blacktriangle}\pr{B_{k}\pr -;F}$ for the fiber
of this map.
\end{defn}

\begin{rem}
\label{rem:section_well-defined}In the situation of Definition \ref{def:gammac},
the colimit defining $\Gamma'_{\blacktriangle}\pr{B_{k}\pr -;F}$
exists. To see this, fix a triangulation of $U$. Since the inclusion
$\Nbd'\pr{\blacktriangle_{k}\pr U}\to\Nbd\pr{\blacktriangle_{k}\pr U}$
is initial (Lemma \ref{lem:small_reg_nbhd}), we only need to take
the colimit over $\Nbd'\pr{\blacktriangle_{k}\pr U}$. Since $\Nbd'\pr{\blacktriangle_{k}\pr U}$
is weakly contractible (being cofiltered by Lemma \ref{lem:small_reg_nbhd})
and every inclusion $Q\to Q'$ in $\Nbd'\pr{\blacktriangle_{k}\pr U}$
induces a homotopy equivalence $Q\cap B_{k}\pr U\xrightarrow{\simeq}Q'\cap B_{k}\pr U$,
the colimit over $\Nbd'\pr{\blacktriangle_{k}\pr U}$ indeed exists.
Moreover, this argument shows that the map
\[
\Gamma\pr{Q\cap B_{k}\pr U;F}\to\Gamma'_{\blacktriangle}\pr{B_{k}\pr U;F}
\]
is an equivalence for every $U\in\Open\pr M$ and $Q\in\Nbd'\pr{\blacktriangle_{k}\pr U}$.
\end{rem}

We can now state the explicit formula for the inverse equivalence.
In the case where $\cal C$ is the $\infty$-category of pointed spaces,
the formula is due to Weiss, and the proof we present below should
be compared with his original argument \cite[Section 7]{Weiss99}. 
\begin{thm}
\label{thm:inv}Let $\cal C$ be a pointed $\infty$-category with
small limits. For every manifold $M$ admitting a PL structure and
every $k\geq0$, the functor
\[
\Gamma_{\blacktriangle}\pr{B_{k}\pr -;-}:\Fun\pr{\Sing B_{k}\pr M,\cal C}\to\Homog^{k}_{\istp,\exh}\pr{M^{\op};\cal C}
\]
is well-defined and an inverse of the equivalence of Proposition \ref{prop:homog}.
\end{thm}

\begin{proof}
In light of Proposition \ref{prop:inverse}, it will suffice to show
that for each functor $F:\Sing B_{k}\pr M\to\cal C$, the map $\Gamma\pr{B_{k}\pr -;F}\to\Gamma'_{\blacktriangle}\pr{B_{k}\pr -;F}$
exhibits $\Gamma'_{\blacktriangle}\pr{B_{k}\pr -;F}$ as a $\pr{k-1}$th
polynomial approximation of $\Gamma\pr{B_{k}\pr -;F}$. We will prove
this in several steps. For notational convenience, we will write $\Phi=\Gamma\pr{B_{k}\pr -;F}$
and $\Psi=\Gamma'_{\blacktriangle}\pr{B_{k}\pr -;F}$. Note that these
functors can take closed subsets of $M$ as its inputs by extending
them by the same defining formula.

\begin{enumerate}[label=(\textit{Step\arabic*}), wide, itemsep = 1em]

\item We show that, for each $U\in\Disj\pr M$, the map
\[
\theta:\Gamma\pr{B''_{k}\pr U;F}\to\colim_{Q\in\Nbd\pr{\blacktriangle^{k}\pr U}}\Gamma\pr{B_{k}\pr U\cap Q;F}
\]
is an equivalence. Identify each component of $U$ with $\bb R^{n}$
by choosing a homeomorphism. Given a decreasing continuous map $\varepsilon:[0,\infty)\to(0,1)$,
let $\widetilde{Q}_{\varepsilon}\subset U^{k}$ denote the open set
of the points $\pr{p_{1},\dots,p_{k}}\in U^{k}$ satisfying the following
conditions:
\begin{itemize}
\item The map $\{p_{1},\dots,p_{k}\}\to\pi_{0}\pr U$ is not injective.
\item For each component $U'\subset U$ containing at least one point in
$\{p_{1},\dots,p_{k}\}$, we have $\max_{p_{i},p_{j}\in U'}\abs{p_{i}-p_{j}}<\varepsilon\pr{\max_{p_{i}\in U'}\abs{p_{i}}}$. 
\end{itemize}

Let $Q_{\varepsilon}\subset\SP^{k}\pr U$ denote the image of $\widetilde{Q}_{\varepsilon}$.
Notice that $Q_{\varepsilon}$ is a neighborhood of $\blacktriangle^{k}\pr U$
in $\SP^{k}\pr U$, and moreover that every neighborhood of $\blacktriangle^{k}\pr U$
contains a neighborhood of the form $Q_{\varepsilon}$. Thus we obtain
an equivalence
\[
\colim_{\varepsilon}\Gamma\pr{Q_{\varepsilon}\cap B_{k}\pr U;F}\xrightarrow{\simeq}\colim_{Q\in\Nbd\pr{\blacktriangle^{k}\pr U}}\Gamma\pr{Q\cap B_{k}\pr U;F}.
\]
On the other hand, the inclusion $Q_{\varepsilon}\cap B_{k}\pr U\to B''_{k}\pr U$
is a homotopy equivalence. To see this, choose a homeomorphism $f:[0,\infty)\to[0,1)$.
By using, for each component $U'\subset U$ intersecting $\{p_{1},\dots,p_{k}\}$,
the map 
\[
\{p_{i}\}_{i\in U''}\mapsto\{f\pr{\max_{p_{s}\in U'}\abs{p_{s}}}\varepsilon\pr{\max_{p_{s}\in U'}\abs{p_{s}}}p_{i}\}_{i\in U''},
\]
we obtain a map $B''_{k}\pr U\to Q_{\varepsilon}\cap B_{k}\pr U$.
Using linear homotopies, one verifies that this is an inverse homotopy
equivalence of the inclusion. (The above map and the homotopies are
well-defined because $\varepsilon$ is a decreasing function with
values in $\pr{0,1}$.) So we obtain another equivalence
\[
\Gamma\pr{B''_{k}\pr U;F}\xrightarrow{\simeq}\colim_{\varepsilon}\Gamma\pr{Q_{\varepsilon}\cap B_{k}\pr U;F}.
\]
The map $\theta$ is the composite of these maps, so it is an equivalence,
as required.

\item Let $T_{k-1}\Phi$ and $T_{k-1}\Psi$ denote the right Kan
extensions of $\Phi$ and $\Psi$ along the inclusion $\Disj^{\leq k-1}\pr M^{\op}\hookrightarrow\Open\pr M^{\op}$.
We then have the following diagram:
\[\begin{tikzcd}
	\Phi & {\Psi } \\
	{T_{k-1}\Phi} & {T_{k-1}\Psi}
	\arrow[from=1-1, to=1-2]
	\arrow["{\eta_{\Phi}}"', from=1-1, to=2-1]
	\arrow["{\eta_{\Psi}}", from=1-2, to=2-2]
	\arrow["\alpha"', from=2-1, to=2-2]
\end{tikzcd}\]The map $\alpha$ is an equivalence by Step 1 (because $B''_{k}\pr U=B_{k}\pr U$
for every $U\in\Disj^{\leq k-1}\pr M$). Consequently, it will suffice
to show that $\eta_{\Psi}$ is an equivalence. We will prove this
in the next step.

\item We show that $\eta_{\Psi}$ is an equivalence. By Remark \ref{rem:MV},
it suffices to show that every open set $U\subset M$ equipped with
a triangulation coming from a PL structure of $M$ satisfies the following
conditions:

\begin{enumerate}[label=(\Roman*)]

\item If $U\in\Disj\pr M$, then the map $\eta_{\Psi,U}:\Psi\pr U\to T_{k-1}\Psi\pr U$
is an equivalence.

\item Let $K_{0}\subset K_{1}\subset\cdots\subset U$ be a sequence
of subcomplexes whose interiors cover $U$. The map 
\[
\Psi\pr U\to\lim_{i}\Psi\pr{K_{i}}
\]
is an equivalence.

\item Let $A_{0},\dots,A_{k-1}\subset U$ be pairwise disjoint compact
subcomplexes. The map
\[
\Psi\pr U\to\lim_{S\in\cal P_{0}\pr{\{0,\dots,k-1\}}}\Psi\pr{U\setminus\bigcup_{i\in S}A_{i}}
\]
is an equivalence. 

\end{enumerate}

For (I), we consider the diagram
\[\begin{tikzcd}
	{\Phi (U)} \\
	& {\Gamma(B''_k(U);F)} & {\Psi(U)} \\
	& {\lim_{V\in \operatorname{Disj}^{\leq k-1}(U)}\Gamma(B_k(V);F)} & {T_{k-1}\Psi(U).}
	\arrow[from=1-1, to=2-2]
	\arrow[curve={height=-12pt}, from=1-1, to=2-3]
	\arrow["{\eta_{\Phi,U}}"', curve={height=12pt}, from=1-1, to=3-2]
	\arrow["\simeq", from=2-2, to=2-3]
	\arrow["\psi"', from=2-2, to=2-3]
	\arrow["\simeq"', from=2-2, to=3-2]
	\arrow["\phi", from=2-2, to=3-2]
	\arrow["{\eta_{\Psi,U}}", from=2-3, to=3-3]
	\arrow["{\alpha_U}", from=3-2, to=3-3]
	\arrow["\simeq"', from=3-2, to=3-3]
\end{tikzcd}\]We saw in the proof of Proposition \ref{prop:inverse} that $\phi$
is an equivalence, and in Step 1 that $\psi$ is an equivalence. Thus
$\eta_{\Psi,U}$ is an equivalence, as claimed.

For (II), use Example \ref{exa:triangulation} and triangulate $\SP^{k}\pr U$
in such a way that it contains $\blacktriangle_{k}\pr U$ as a subcomplex.
As we observed in Remark \ref{rem:section_well-defined}, for an arbitrary
element $Q\in\Nbd'\pr{\blacktriangle_{k}\pr U}$, the map
\[
\Gamma\pr{Q\cap B_{k}\pr U;F}\to\Psi\pr U
\]
is an equivalence. A similar argument shows that the map
\[
\Gamma\pr{Q\cap B_{k}\pr{K_{i}};F}\to\Psi\pr{K_{i}}
\]
is an equivalence for every $i$, since $Q\cap\SP^{k}\pr{K_{i}}\in\Nbd'\pr{\blacktriangle_{k}\pr{K_{i}}}$
(Remark \ref{rem:dnbhd_compat}). Consequently, we are reduced to
showing that the map
\[
\Gamma\pr{Q\cap B_{k}\pr U;F}\to\lim_{i}\Gamma\pr{Q\cap B_{k}\pr{K_{i}};F}
\]
is an equivalence. For this, it suffices to show that $\Sing\pr{Q\cap B_{k}\pr U}=\bigcup_{i\geq0}\Sing\pr{Q\cap B_{k}\pr{K_{i}}}$
(by Remark \ref{rem:dec_lim} and the fact filtered colimits of simplicial
sets already compute homotopy colimits), which follows by inspection.

For (III), subdividing $U$ if necessary, we may assume that no simplex
of $U$ intersects $A_{i}$ and $A_{j}$ for distinct indices $i\neq j$.
Let $C_{i}$ denote the simplicial complement of $A_{i}$ in $U$.
For each subset $S\subset\{0,\dots,k\}$, the inclusion $U\setminus\bigcup_{i\in S}A_{i}\to C_{i}$
is a homotopy equivalence (see Remark \ref{rem:scomplement}), so
it suffices to show that the map
\[
\Psi\pr U\to\lim_{S\in\cal P_{0}\pr{\{0,\dots,k-1\}}}\Psi\pr{C_{i}}
\]
is an equivalence. This follows from an argument similar to the one
in the previous paragraph (and \cite[Theorem A.3.1]{HA}), noting
that if $Q\in\Nbd\pr{\blacktriangle_{k}\pr U}$ is sufficiently small,
we have $B_{k}\pr U\cap Q=\bigcup^{k-1}_{i=0}B_{k}\pr{C_{i}}\cap Q$\@.
\qedhere

\end{enumerate}
\end{proof}

As a corollary of Proposition \ref{prop:inverse}, we give a connectivity
estimate of polynomial approximations. When $\cal C$ is the category
of chain complexes of abelian groups or that of spectra, this is proved
in \cite[Lemma E.4]{RW14} by a rather ad-hoc argument that reduces
the claim to that of pointed spaces. Proposition \ref{prop:inverse}
allows us to get away with this reduction while keeping the essence
of the argument of \cite{RW14} intact.
\begin{cor}
\label{cor:RW14}Let $M$ be an $n$-manifold admitting a PL structure,
let $\cal C$ be a stable $\infty$-category with small limits and
a t-structure, and let $F:\Open\pr M^{\op}\to\cal C$ be an isotopy
functor. Suppose that:

\begin{enumerate}[label=(\roman*)]

\item For each $U\in\Disj\pr M$, the object $F\pr U\in\cal C$ is
$a$-connected.

\item $0$-connected objects are stable under small products.

\end{enumerate}

Then for every $k\geq0$, the values of $T_{k}F$ are $\pr{a-kn-k}$-connected.
\end{cor}

\begin{proof}
Shifting degrees if necessary, we may assume that $a=0$. We set $\lambda_{k}=-kn-k$.
The proof proceeds by induction on $k$. 

For the base step where $k=0$, $T_{0}F$ is the constant diagram
at $F\pr{\emptyset}$, so the claim is trivial. For the inductive
step, let $k\geq1$, and suppose we have proved the claim for $T_{k-1}F$.
We have a fiber sequence
\[
L_{k}F\to T_{k}F\to T_{k-1}F
\]
in $\Fun\pr{\Open\pr M^{\op},\cal C}$. By the induction hypothesis,
the values of $T_{k-1}F$ are $\lambda_{k-1}$-connected, and hence
are $\lambda_{k}$-connected. Thus, by the long exact sequence of
homotopy groups, we are reduced to showing that $L_{k}F$ is $\lambda_{k}$-connected.
We prove this in several steps.

\begin{enumerate}[label=(\textit{Step\arabic*}), wide, itemsep = 1em]

\item We show that, for each $U\in I^{=k}_{M}$, the object $\pr{L_{k}F}\pr U\in\cal C$
is $\pr{-k}$-connected. By Remark \ref{rem:interpretation} and condition
(i), $\pr{L_{k}F}\pr U$ can be written as a total fiber of a diagram
$[1]^{k}\to\cal C_{\geq1}$. Since the fiber of maps of $a$-connected
objects is $\pr{a-1}$-connected, the claim thus follows by induction
on $k$ and the inductive formula for total fibers \cite[Corollary 2.2]{AB22}.

\item We will show that the values of $L_{k}F$ are $\lambda_{k}$-connected.
Let $H:\Sing B_{k}\pr M\to\cal C$ denote the functor corresponding
to the homogeneous functor $L_{k}F$. For each $U\in\Open\pr M$,
 Proposition \ref{thm:inv} gives us an equivalence
\begin{align*}
L_{k}F\pr U & \simeq\Gamma_{\blacktriangle}\pr{U;H}\\
 & \simeq\fib\pr{\Gamma\pr{B_{k}\pr U;H}\to\Gamma\pr{B_{k}\pr U\cap Q;H}}\\
 & \simeq\fib\pr{\lim_{\Sing B_{k}\pr U}H\to\lim_{\Sing\pr{B_{k}\pr U\cap Q}}H},
\end{align*}
where $Q$ is a derived neighborhood of $\blacktriangle_{k}\pr U$
in $\SP^{k}\pr U$ with respect to some triangulation of $U$. Since
$B_{k}\pr U$ is obtained from $B_{k}\pr U\cap Q$ by attaching cells
of dimension $\leq kn$, combining Step 1 with Lemma \ref{lem:conn_lim}
(and condition (ii)), we deduce that the map
\[
\lim_{\Sing B_{k}\pr U}H\to\lim_{\Sing\pr{B_{k}\pr U\cap Q}}H
\]
is $\pr{\lambda_{k}+1}$-connected. Thus its fiber is $\lambda_{k}$-connected,
completing the induction.\qedhere

\end{enumerate}
\end{proof}

\begin{rem}
Let $M$ be a smooth manifold and let $I^{=k}_{\sm,M}$ denote the
subcategory of $\Disj^{\leq k}_{\sm}\pr M$ spanned by the smooth
isotopy equivalences between objects with exactly $k$ components.
Everything in this section remains valid if we replace $I^{=k}_{M}$
by$I^{=k}_{\sm,M}$ and $\Disj,\Disj^{\leq k}$ by $\Disj_{\sm}$
and $\Disj^{\leq k}_{\sm}$. A similar remark applies to the PL case.
\end{rem}

\section{\label{sec:context_free}Context-Free Case}

In many cases, functors subject to the analysis of manifold calculus
are restrictions of functors defined on all of $\Mfld_{n}$ or its
variants. The study of such functors is called \textbf{context-free
manifold calculus} \cite{MR3071127}. Some of the results of Section
\ref{sec:poly} generalize to the context-free case, which we record
in this section.
\begin{defn}
\label{def:ctxfree}Let $n,k\geq0$. Let $\cal C$ be an an $\infty$-category
and let $F:\Mfld_{n}\to\cal C$ be a functor.
\begin{enumerate}
\item The \textbf{Weiss $k$-topology} on $\sf{Mfld}_{n}$ is the Grothendieck
topology on $\sf{Mfld}_{n}$ such that for each object $M\in\sf{Mfld}_{n}$,
a sieve $\cal U$ on $M$ is a covering sieve if and only if for each
finite set $S\subset M$ of cardinality $\leq k$, there is an object
$\pr{U\to M}\in\cal U$ such that $S\subset U$. The intersection
of the Weiss topologies on $M$ are called the \textbf{Weiss topology}. 
\item We say that $F$ is a \textbf{Weiss $k$-cosheaf} (resp. \textbf{Weiss
cosheaf}) if the composite $\sf{Mfld}_{n}\to\Mfld_{n}\xrightarrow{F}\cal C$
is a cosheaf with respect to the Weiss $k$-topology (resp. Weiss
topology).
\item We say that $F$ is \textbf{$k$-excisive}, or \textbf{polynomial
of degree $\leq k$}, if for each $n$-manifold $M$ and each pairwise
disjoint closed sets $A_{0},\dots,A_{k}\subset M$, the map
\[
\colim_{\emptyset\neq S\subset\{0,\dots,k\}}F\pr{M\setminus\bigcup_{i\in S}A_{i}}\to F\pr M
\]
is an equivalence.
\item We say that $F$ is \textbf{exhaustive} if for each $n$-manifold
$M$ and each increasing sequence $U_{0}\subset U_{1}\subset\cdots$
of open sets of $M$ which covers $M$, the map
\[
\colim_{i}F\pr{U_{i}}\to F\pr M
\]
is an equivalence.
\item We will write $\Exc^{k}_{\exh}\pr{\Mfld_{n};\cal C}\subset\Fun\pr{\Mfld_{n},\cal C}$
for the full subcategory spanned by the $k$-excisive, exhaustive
functors $\Mfld_{n}\to\cal C$.
\item The \textbf{Taylor cotower} of $F$ is the Taylor cotower of $F$
with respect to the sequence of full subcategories $\Disk^{\leq0}_{n}\subset\Disk^{\leq1}_{n}\subset\cdots$
(Definition \ref{def:Taylor}).
\end{enumerate}
\end{defn}

\begin{rem}
In light of the constructions in non-context-free manifold calculus,
some readers find it more natural to work with functors on $\sf{Mfld}_{n}$
which carry isotopy equivalences to equivalences, rather than functors
defined on $\Mfld_{n}$. There is a reason for this: The localization
theorem for isotopy equivalences (Theorem \ref{thm:localizing_wrt_istpy}),
which formed a basis of our argument in non-context-free manifold
calculus, is significantly more subtle in the context-free case. We
will come back to this point in Appendix \ref{App:(Non)-Loc}.
\end{rem}

\begin{rem}
The Taylor cotower of $F$ coincides with the tower of Kurannich--Kupers
\cite[5.3.3]{KK24}, both in the smooth case and the topological case.
See \cite[Remark 5.8]{KK24}.
\end{rem}

The following theorems, which are the main result of this section,
are analogs of Theorems \ref{thm:best_approx_exists}, \ref{thm:Weiss_k-cosheaf}
and \ref{thm:Weiss_cosheaf} in the context-free case. 
\begin{thm}
\label{thm:context_free1}Let $n,k\geq0$, let $\cal C$ be an $\infty$-category
with small colimits, and let $F:\Mfld_{n}\to\cal C$ be a functor.
The following conditions are equivalent:
\begin{enumerate}
\item $F$ is a left Kan extension of $F\vert\Disk^{\leq k}_{n}$.
\item For each $n$-manifold $M$, each small $\infty$-category $\cal I$,
and each functor $\chi:\cal I\to N\pr{\Open\pr M}$ satisfying the
Weiss $k$-condition (Definition \ref{def:Weiss_condition}), the
map
\[
\colim_{I\in\cal I}F\pr{\chi\pr I}\to F\pr M
\]
is an equivalence.
\item $F$ is a Weiss $k$-cosheaf.
\item $F$ is $k$-excisive and exhaustive.
\end{enumerate}
In particular, the inclusion $\Exc^{k}_{\exh}\pr{\Mfld_{n};\cal C}\hookrightarrow\Fun\pr{\Mfld_{n},\cal C}$
admits a right adjoint, which carries a functor $F\in\Fun\pr{\Mfld_{n},\cal C}$
to a left Kan extension of $F\vert\Disk^{\leq k}_{n}$.
\end{thm}

\begin{thm}
\label{thm:context_free2}Let $n\geq0$, let $\cal C$ be an $\infty$-category
with small colimits, and let $F:\Mfld_{n}\to\cal C$ be a functor.
The following conditions are equivalent:
\begin{enumerate}
\item $F$ is a left Kan extension of $F\vert\Disk_{n}$.
\item For each $n$-manifold $M$, each small $\infty$-category $\cal I$,
and each functor $\chi:\cal I\to\Open\pr M$ satisfying the Weiss
condition (Definition \ref{def:Weiss_condition}), the map
\[
\colim_{I\in N\pr{\cal I}}F\pr{\chi\pr I}\to F\pr M
\]
is an equivalence.
\item $F$ is a Weiss cosheaf.
\end{enumerate}
\end{thm}

\begin{rem}
Theorems \ref{thm:context_free1} and \ref{thm:context_free2} remain
valid if we replace $\Mfld_{n}$, $\Disk_{n}$, and $\Disk^{\leq k}_{n}$
by $\Mfld_{\sm,n}$, $\Disk_{\sm,n}$, and $\Disk^{\leq k}_{\sm,n}$,
with essentially the same proof.
\end{rem}

The proofs of Theorems \ref{thm:context_free1} and \ref{thm:context_free2}
are very similar, so we will focus on Theorem \ref{thm:context_free1}.
\begin{proof}
[Proof of Theorem \ref{thm:context_free1}]Clearly (2)$\implies$(3)$\implies$(4),
so it suffices to show that (1)$\implies$(2) and that (4)$\implies$(1). 

First we show that (1)$\implies$(2). Suppose that $F$ is a left
Kan extension of $F\vert\Disk^{\leq k}_{n}$. Let $M$ be an $n$-manifold
and let $\chi:\cal I\to\Open\pr M$ be a functor of small categories
satisfying the Weiss $k$-condition. We wish to show that the map
\[
\colim_{I\in\cal I}F\pr{\chi\pr I}\to F\pr M
\]
is an equivalence. In light of Theorem \ref{thm:Weiss_k-cosheaf},
it suffices to show that the composite
\[
\Open\pr M\to\Mfld_{n}\xrightarrow{F}\cal C
\]
is a left Kan extension of $\Disj^{\leq k}\pr M$. Since $F$ is a
left Kan extension of $F\vert\Disk^{\leq k}_{n}$, it suffices to
show that, for each open set $U\subset M$, the map
\[
\Disj^{\leq k}\pr M_{/U}\to\Disk^{\leq k}_{n/U}
\]
is final. This follows from Lemma \ref{lem:final_slice_mfd} (or from
Theorem \ref{thm:localizing_wrt_istpy}, for localization functors
are final \cite[\href{https://kerodon.net/tag/02N9}{Tag 02N9}]{kerodon}). 

Next we show that (4)$\implies$(1). Let $G$ denote a left Kan extension
of $F\vert\Disk^{\leq k}_{n}$. We wish to show that the induced natural
transformation $\alpha:G\to F$ is a natural equivalence. It suffices
to show that, for each $n$-manifold $M$, the map $\alpha\iota:G\iota\to F\iota$
is a natural equivalence, where $\iota:\Open\pr M\to\Mfld_{n}$ denotes
the inclusion. By Theorem \ref{thm:Weiss_k-cosheaf}, the functor
$F\iota$ is a left Kan extension of its restriction to $\Disj^{\leq k}\pr M$.
As we saw in the previous paragraph, the functor $G\iota$ is also
left Kan extensions of its restriction to $\Disj^{\leq k}\pr M$.
Since the components of $\alpha\iota$ at each object in $\Disj^{\leq k}\pr M$
are equivalences, we deduce that $\alpha\iota$ must be an equivalence.
The proof is now complete.
\end{proof}

As a corollary of the theorems, we obtain the following structural
result on Taylor cotowers.
\begin{cor}
Let $n\geq0$, let $\cal C$ be an $\infty$-category with small colimits,
and let $F:\Mfld_{n}\to\cal C$ be a functor. Then:
\begin{enumerate}
\item For each $k\geq0$, the $k$th level $T_{k}F$ of the Taylor cotower
of $F$ is polynomial of degree $\leq k$ and exhaustive.
\item The Taylor cotower of $F$ is convergent if and only if it satisfies
the equivalent conditions of Theorem \ref{thm:context_free2}.
\end{enumerate}
\end{cor}

\begin{proof}
Part (1) follows from Theorem \ref{thm:context_free1}. Part (2) is
a consequence of Proposition \ref{prop:taylor}.
\end{proof}

\begin{rem}
Theorem \ref{thm:context_free1} admits a following variant: Suppose
we are given a right fibration $\pi:\cal M\to\Mfld_{n}$. Set $\cal D=\Disk_{n}\times_{\Mfld_{n}}\cal M$
and $\cal D^{\leq k}=\Disk^{\le k}_{n}\times_{\Mfld_{n}}\cal M$.
Given an $\infty$-category $\cal C$ with small colimits, we say
that a functor $F:\cal M\to\cal C$ is a Weiss $k$-cosheaf if for
each $n$-manifold $M$ and each functor $\Open\pr M\to\cal M$ rendering
the diagram 
\[\begin{tikzcd}
	& {\mathcal{M}} \\
	{\operatorname{Open}(M)} & {\mathcal{M}\mathsf{fld}_n}
	\arrow[from=1-2, to=2-2]
	\arrow[from=2-1, to=1-2]
	\arrow[from=2-1, to=2-2]
\end{tikzcd}\]commutative, the composite $\Open\pr M\to\cal M\xrightarrow{F}\cal C$
is a Weiss $k$-cosheaf. We define what it means for a functor $\cal M\to\cal C$
to be a Weiss cosheaf, $k$-excisive, or exhaustive similarly. Then
the following conditions for a functor $F:\cal M\to\cal C$ are equivalent: 
\begin{enumerate}
\item $F$ is a left Kan extension of $F\vert\cal D^{\leq k}$.
\item For each $n$-manifold $M$, each small $\infty$-category $\cal I$,
and each functor $\chi:\cal I\to N\pr{\Open\pr M}$ satisfying the
Weiss $k$-condition (Definition \ref{def:Weiss_condition}), the
map
\[
\colim_{I\in\cal I}F\pr{f\pr{\chi\pr I}}\to F\pr{f\pr M}
\]
is an equivalence, where $f:\Open\pr M\to\cal M$ is any functor lifting
$\Open\pr M\to\Mfld_{n}$.
\item $F$ is a Weiss $k$-cosheaf.
\item $F$ is $k$-excisive and exhaustive.
\end{enumerate}
The proof is essentially the same as that of Theorem \ref{thm:context_free1},
using the following observations: 
\begin{itemize}
\item For each object $\overline{M}\in\cal M$ with image $M\in\Mfld_{n}$,
the functor $\cal D_{/\overline{M}}\to\Disk_{n/M}$ is a trivial fibration
(because $\pi$ is a right fibration).
\item In point (2), the composite $\Open\pr M\to\cal M\xrightarrow{F}\cal C$
is an isotopy functor, because $\pi$ is conservative (being a right
fibration).
\end{itemize}
We can similarly prove a variant of Theorem \ref{thm:context_free2}
for $\cal M$ and $\cal D$; we can also consider right fibrations
over $\Mfld_{\sm,n}$. 

One prominent source of right fibrations over $\Mfld_{n}$ comes from
framed manifolds. Let $B\Top\pr n\subset\Mfld_{n}$ denote the full
subcategory spanned by $\bb R^{n}$. Given a map $p:B\to B\Top\pr n$
of Kan complexes, the $\infty$-category of \textbf{$B$-framed $n$-manifolds}
\cite[Definition 2.7]{FHTM} is modeled by the fiber product
\[
\Mfld^{B}_{n}=\Mfld_{n}\times_{\cal S_{/B\Top\pr n}}\cal S_{/p}.
\]
Since the functor $\cal S_{/p}\to\cal S_{/B\Top\pr n}$ is a right
fibration, so is its pullback $\Mfld^{B}_{n}\to\Mfld_{n}$. Therefore,
we have analogs of Theorem \ref{thm:context_free1} and \ref{thm:context_free2}
for $B$-framed manifolds and $B$-framed disks.
\end{rem}

\section{\label{sec:bd}Boundary Case}

The results in Sections \ref{sec:poly} through \ref{sec:context_free}
extend to the case of manifolds with boundaries, with very little
modification to the arguments. We list the key changes below, leaving
the details to the reader. Throughout this section, we fix an integer
$n\geq1$ and an $\pr{n-1}$-manifold $Z$.
\begin{itemize}
\item We write $\Mfld_{Z}$ for the homotopy coherent nerve of the topological
category whose objects are $n$-manifolds $M$ with boundary equipped
with a homeomorphism $\partial M\cong Z$, and whose mapping spaces
are given by the subspace $\Emb_{\partial}\pr{M,N}\subset\Emb\pr{M,N}$
for the subspace consisting of the open embeddings $f:M\to N$ extending
the boundary identifications. We also define (the nerve of) an ordinary
category $\sf{Mfld}_{Z}$ similarly. From section 3 onward, $\Mfld_{n}$
must be replaced by $\Mfld_{Z}$, and whenever a term ``manifold''
appears in a definition, theorem, proposition, etc, we must replace
it by a ``manifold with boundary $Z$''.
\item We write $\sf{Disk}_{Z}\subset\sf{Mfld}_{Z}$ for the full subcategory
spanned by the objects isomorphic to $(Z\times\bb R_{\geq0})\amalg(\bb R^{n}\times S)$
for some finite set $S$. We then write $\Disk_{Z}\subset\Mfld_{Z}$
for the full subcategory spanned by the objects of $\sf{Disk}_{Z}$.
We define $\infty$-categories $\Disk^{\leq k}_{Z}$ and $\sf{Disk}^{\leq k}_{Z}$
in the obvious manner. From section 3 onward, $\Disk_{n}$ and $\Disk^{\leq k}_{n}$
must be replaced by $\Disk_{Z}$ and $\sf{Disk}^{\leq k}_{Z}$.
\item From section 3 onward, $\Open\pr M$ must be replaced by the poset
$\Open_{\partial}\pr M$ of open sets of $M$ containing $\partial M$.
For each $U\in\Open_{\partial}\pr M$, we write $U_{\partial}\subset U$
for the union of the components that intersect $\partial M$, and
set $U_{\mathrm{in}}=U\setminus U_{\partial}$.
\item From section 3 onward, we must replace $\Disj\pr M$ by the full subposet
$\Disj_{\partial}\pr M\subset\Open_{\partial}\pr M$ of those $U$
of the form $U_{\partial}\amalg U_{\mathrm{in}}$, where $U_{\partial}$
is a neighborhood of $M$ and $U_{\mathrm{in}}$ is an open set of
$M^{\circ}$. A similar remark applies to $\Disj^{\leq k}$. We then
define $I^{=k}_{M}$ as the non-full subposet of $\Disj^{\leq k}_{\partial}\pr M$
spanned by the inclusions $U\subset V$ such that $\pi_{0}\pr U\to\pi_{0}\pr V$
is bijective and $\abs{\pi_{0}\pr{U_{\mathrm{in}}}}=k$. 
\item One of the key ingredient of Section 3 is Ayala--Francis's theorem
(Theorem \ref{thm:localizing_wrt_istpy}), and we need this for for
$M\in\Mfld_{Z}$ (with various disk categories replaced by the $Z$-counterparts).
An inspection reveals that we only need the following generalization
of Proposition \ref{prop:emb_conf_cart}:
\begin{prop}
\label{prop:emb_conf_cart_bd}For every finite set $S$ and every
morphism $\phi:M\to N$ in $\sf{Mfld}_{Z}$, the square 
\[\begin{tikzcd}
	{\operatorname{Emb}_\partial(Z\times[0,\infty))\amalg(\mathbb{R}^n\times S),M)} & {\operatorname{Emb}_\partial((Z\times[0,\infty))\amalg(\mathbb{R}^n\times S),N)} \\
	{\operatorname{Conf}(S,M^\circ)} & {\operatorname{Conf}(S,N^\circ)}
	\arrow[from=1-1, to=1-2]
	\arrow[from=1-1, to=2-1]
	\arrow[from=1-2, to=2-2]
	\arrow[from=2-1, to=2-2]
\end{tikzcd}\]is homotopy cartesian.
\end{prop}

\begin{proof}
For each continuous map $\varepsilon:\partial M\to(0,\infty)$, we
define 
\[
Z\pr{\varepsilon}=\{\pr{p,t}\in\partial M\times\bb R_{\geq0}\mid0\leq t<\varepsilon\pr p\}.
\]
We then define a simplicial set $\Germ_{\partial}\pr{S,M}$ as the
colimit
\[
\Germ_{\partial}\pr{S,M}=\colim_{r,\varepsilon}\Sing\Emb_{\partial}\pr{Z\pr{\varepsilon}\amalg\pr{S\times B^{n}\pr r},M},
\]
where the colimit is indexed over all positive reals $r>0$ and all
continuous maps $\varepsilon:Z\to(0,\infty)$. The map $\Emb_{\partial}\pr{(Z\times[0,\infty))\amalg(\bb R^{n}\times S),M}\to\Germ_{\partial}\pr{S,M}$
is a homotopy equivalence, so it suffices to show that the square
\[\begin{tikzcd}
	{\operatorname{Germ}_\partial(S,M)} & {\operatorname{Germ}_\partial(S,N)} \\
	{\operatorname{Sing}\operatorname{Conf}(S,M^\circ)} & {\operatorname{Sing}\operatorname{Conf}(S,N^\circ)}
	\arrow[from=1-1, to=1-2]
	\arrow[from=1-1, to=2-1]
	\arrow[from=1-2, to=2-2]
	\arrow[from=2-1, to=2-2]
\end{tikzcd}\]is homotopy cartesian. A choice of a collar of $M$ gives us an isomorphism
of simplicial sets $\Germ_{\partial}\pr{S,M}\cong\Germ_{\partial}\pr{\emptyset,Z\times[0,1)}\times\Germ\pr{S,M^{\circ}}$,
so the claim follows from Proposition \ref{prop:emb_conf_cart}.
\end{proof}

\item In the definition of $k$-excisive functors (Definitions \ref{def:istp_exh_poly}
and \ref{def:ctxfree}), we must assume that the closed sets $A_{i}$
lie in $M^{\circ}$. 
\item In Proposition \ref{prop:I_M=00003Dk_htpytype}, Corollary \ref{cor:homog_class},
Proposition \ref{prop:inverse}, and Theorem \ref{thm:inv}, we must
replace $B_{k}\pr -$ by $B_{k}\pr{\pr -^{\circ}}$.
\item Lemma \ref{lem:homog_inv} will be a claim about $U\in\Disk_{Z}$
and the map
\[
\int^{V\in\Disj^{\leq k-1}_{\partial}\pr U}\Sing B_{k}\pr{V^{\circ}}\to\Sing\pr{B_{k}\pr{U^{\circ}}\setminus B'_{k}\pr{U_{\mathrm{in}}}}.
\]
\item In Proposition \ref{prop:inverse}, $B_{k}\pr U$ in the definition
of $\widetilde{F}$ should be replaced by $B_{k}\pr{U_{\mathrm{in}}}$.
\item In the definition of fat diagonal (Definition \ref{def:fd}), we must
adjoin to $\blacktriangle^{k}M$ the images of the points $\pr{p_{1},\dots,p_{k}}\in M^{k}$
such that some $p_{i}$ lies in $\partial M$.
\item In Definition \ref{def:gammac}, we must consider the functor $\Gamma'_{\blacktriangle}\pr{B_{k}\pr{\pr -^{\circ}};F}:\Open_{\partial}\pr M^{\op}\to\cal C$
defined by
\[
\Gamma'_{\blacktriangle}\pr{B_{k}\pr{U^{\circ}};F}=\colim_{Q\in\Nbd\pr{\blacktriangle_{k}\pr U}^{\op}}\Gamma\pr{Q\cap B_{k}\pr{U^{\circ}};F}.
\]
We must also consider the functor $\Gamma\pr{B_{k}\pr{\pr -^{\circ}};F}$.
We then define $\Gamma_{\blacktriangle}\pr{B_{k}\pr{\pr -^{\circ}};F}$
as the fiber of the map $\Gamma\pr{B_{k}\pr{\pr -^{\circ}};F}\to\Gamma'_{\blacktriangle}\pr{B_{k}\pr{\pr -^{\circ}};F}$.
\end{itemize}
\appendix

\section{\label{App:(Non)-Loc}(Non)-Localization Theorems}

Let $L:\cal C\to\cal D$ be a functor of $\infty$-categories, and
let $S$ be a set of morphisms of $\cal C$. Recall (\cite[Definition 2.4.2]{landoo-cat})
that $L$ is said to \textbf{exhibit $\cal D$ as a }(\textbf{Dwyer--Kan})\textbf{
localization of $\cal C$ with respect to $S$} it satisfies the following
conditions:
\begin{itemize}
\item The functor $L$ carries every morphism in $S$ to an equivalence.
\item For every $\infty$-category $\cal E$, the functor
\[
\Fun\pr{\cal D,\cal E}\to\Fun'\pr{\cal C,\cal E}
\]
is a categorical equivalence, where $\Fun'\pr{\cal C,\cal E}$ denotes
the full subcategory of $\Fun\pr{\cal C,\cal E}$ spanned by the functors
$\cal C\to\cal E$ carrying every morphism in $S$ to an equivalence.
\end{itemize}
A recurring theme in homotopy theory is that many nontrivial $\infty$-categories
arise as localizations of \textit{ordinary} categories. A famous theorem
of Ayala and Francis \cite[Proposition 2.19]{FHTM} is one manifestation
of this: It asserts that for each $n$-manifold $M$, the functor
\[
\D_{n/M}\to\Disk_{n/M}
\]
exhibits $\Disk_{n/M}$ as a localization of $\D_{n/M}$ with respect
to isotopy equivalences. The goal of this subsection is to present
a detailed proof of this theorem and some of its variants.\footnote{The original proof of Ayala and Francis is missing some important
justifications. The author learned the current proof from David Ayala
and is grateful to him for his explanation.}

\subsection{\label{subsec:loc}Ayala--Francis's Theorem on Localizations at
Isotopy Equivalences}

In this subsection, we give a proof of the following theorem. Point
(1) is due to Ayala and Francis:
\begin{thm}
\label{thm:localizing_wrt_istpy}Let $n\geq0$, and let $M$ be an
$n$-manifold. 
\begin{enumerate}
\item The functor $\D_{n/M}\to\Disk_{n/M}$ exhibits $\Disk_{n/M}$ as a
localization of $\D_{n/M}$ with respect to isotopy equivalences.
\item For every $k\geq0$, the functor $\D^{\leq k}_{n/M}\to\Disk^{\leq k}_{n/M}$
exhibits $\Disk^{\leq k}_{n/M}$ as a localization of $\D^{\leq k}_{n/M}$
with respect to isotopy equivalences.
\end{enumerate}
\end{thm}

We will establish Theorem \ref{thm:localizing_wrt_istpy} after a
few preliminaries. We begin with a variation of Quillen's Theorem
B. 
\begin{defn}
\cite[Definition 4.6.3]{CisinskiHCHA}\label{def:loc_const} A map
$E\to B$ of simplicial sets is said to be\textbf{ locally constant}
if for each morphism $X\to B$ of simplicial sets, the square 
\[\begin{tikzcd}
	{X\times _BE} & E \\
	X & B
	\arrow[from=1-2, to=2-2]
	\arrow[from=2-1, to=2-2]
	\arrow[from=1-1, to=1-2]
	\arrow[from=1-1, to=2-1]
\end{tikzcd}\]is homotopy cartesian in the Kan--Quillen model structure.
\end{defn}

\begin{prop}
[A Variation of Quillen's Theorem B]\label{prop:Theorem_B}Let $\pi:E\to B$
be a cocartesian fibration of simplicial sets. The following conditions
are equivalent:
\begin{enumerate}
\item The map $\pi$ is locally constant.
\item For each edge $\alpha:b\to b'$ in $B$, the induced functor
\[
\alpha_{!}:E_{b}=E\times_{B}\{b\}\to E_{b'}
\]
is a weak homotopy equivalence. 
\end{enumerate}
\end{prop}

\begin{proof}
We will prove that (2)$\implies$(1); the reverse implication can
be proved similarly (and is easier). Factor the map $X\to B$ as $X\xrightarrow{i}X'\xrightarrow{p}B$,
where $i$ is anodyne and $p$ is a Kan fibration. Since the Kan--Quillen
model structure is right proper, it suffices to show that the map
$X\times_{B}E\to X'\times_{B}E$ is a weak homotopy equivalence. We
will prove something more general: We claim that, for every pair of
morphisms $K\xrightarrow{f}L\to B$ of simplicial sets with $f$ a
weak homotopy equivalence, the map
\[
K\times_{B}E\to L\times_{B}E
\]
is a weak homotopy equivalence. Let $\scr M$ denote the class of
morphisms $S\to T$ of simplicial sets such that, for any map $T\to B$,
the induced map $S\times_{B}E\to T\times_{B}E$ is a weak homotopy
equivalence. We must show that $\scr M$ contains all weak homotopy
equivalences. By \cite[Proposition 4.6.1]{CisinskiHCHA}, it suffices
to show that $\scr M$ contains all morphisms of the form $\Delta^{0}\to\Delta^{n}$,
for any $n\geq0$. 

Let $0\le i\leq n$ be integers and let $\sigma:\Delta^{n}\to B$
be an $n$-simplex, which we depict as $b_{0}\to\cdots\to b_{n}$.
We wish to show that the map $E_{b_{i}}\to\Delta^{n}\times_{B}E$
is a weak homotopy equivalence. Choose a cocartesian natural transformation
$h:\Delta^{\{i,n\}}\times E_{b_{i}}\to\Delta^{n}\times_{B}E$ rendering
the diagram 
\[\begin{tikzcd}
	{\{i\}\times E_{b_i}} & {\Delta^n\times _B E} \\
	{\Delta^{\{i,n\}}\times E_{b_i}} & {\Delta^n}
	\arrow[from=1-1, to=2-1]
	\arrow[from=1-1, to=1-2]
	\arrow[from=1-2, to=2-2]
	\arrow[from=2-1, to=2-2]
	\arrow[dashed, from=2-1, to=1-2]
\end{tikzcd}\]commutative. It will suffice to show that the restriction $h\vert\{n\}\times E_{b_{i}}:\{n\}\times E_{b_{i}}\to\Delta^{n}\times_{B}E$
is a weak homotopy equivalence. We can factor this map as
\[
\{n\}\times E_{b_{i}}\xrightarrow{h'}E_{b_{n}}\xrightarrow{h''}\Delta^{n}\times_{B}E,
\]
where $h''$ denotes the inclusion. The map $h'$ is a weak homotopy
equivalence by hypothesis. Also, since $p$ is a cocartesian fibration
and the inclusion $\{n\}\subset\Delta^{n}$ is final, \cite[Proposition 4.1.2.15]{HTT}
shows that $h''$ is final. In particular, $h''$ is a weak homotopy
equivalence. It follows that the map $h\vert\{n\}\times E_{b_{i}}$
is a composite of weak homotopy equivalences, for it is the composite
of two weak homotopy equivalences. The claim follows.
\end{proof}

We use Proposition \ref{prop:Theorem_B} to prove the following recognition
result for localizations.
\begin{prop}
\label{prop:localization_criterion}Let $f:\cal C\to\cal D$ be a
functor of $\infty$-categories. Set $\cal W=\cal C\times_{\cal D}\cal D^{\simeq}$.
Suppose that, for each object $C\in\cal C$, the map $\cal C_{/C}\times_{\cal C}\cal W\to\pr{\cal D_{/f\pr C}}^{\simeq}$
is a weak homotopy equivalence. The following conditions are equivalent:
\begin{enumerate}
\item The functor $f$ exhibits $\cal D$ as a localization of $\cal C$
with respect to morphisms in $\cal W$.
\item The map $\cal W\to\cal D^{\simeq}$ is a weak homotopy equivalence.
\end{enumerate}
\end{prop}

\begin{proof}
We will use the localization criterion using Rezk's \textbf{classification
diagram}, due to Mazel-Gee. For each simplicial set $K$, let $\Fun^{\cal W}\pr{K,\cal C}\subset\Fun\pr{K,\cal C}$
denote the subcategory spanned by the natural transformations whose
components are morphisms of $\cal W$. Recall that the classification
diagram $N\pr{\cal C,\cal W}$ is the simplicial object in $\SS$
whose $n$th simplicial set is given by
\[
N\pr{\cal C,\cal W}_{n}=\Fun^{\cal W}\pr{\Delta^{n},\cal C}=\Fun\pr{\Delta^{n},\cal C}\times_{\cal C^{n+1}}\cal W^{n+1}.
\]
We will write $N\pr{\cal D}=N\pr{\cal D,\cal D^{\simeq}}$; it is
a complete Segal space \cite{Rezk01}, called the \textbf{classifying
diagram} of $\cal D$. According to \cite[Theorem 3.8]{MR4045352}
(see also \cite[Corollary 4.6]{A23b} and \cite{AC25}), condition
(1) is equivalent to the following condition:
\begin{itemize}
\item [(1$'$)]The map $N\pr{\cal C,\cal W}\to N\pr{\cal D}$ is a weak
equivalence of the complete Segal space model structure.
\end{itemize}
We will show that condition (1$'$) is equivalent to condition (2).

First we show that the map $d_{0}:N\pr{\cal C,\cal W}_{1}\to N\pr{\cal C,\cal W}_{0}$
is locally constant. Consider the commutative diagram
\[\begin{tikzcd}
	{\operatorname{Fun}^{\mathcal{W}}(\Delta^{1},\mathcal{C})} & {\operatorname{Fun}(\Delta^{1},\mathcal{D})^\simeq} \\
	{\operatorname{Fun}^{\mathcal{W}}(\{1\},\mathcal{C})} & {\operatorname{Fun}(\{1\},\mathcal{D})^\simeq.}
	\arrow["\phi", from=1-1, to=1-2]
	\arrow["\pi"', from=1-1, to=2-1]
	\arrow["{\pi'}", from=1-2, to=2-2]
	\arrow[from=2-1, to=2-2]
\end{tikzcd}\]We wish to show that $\pi$ is locally constant. According to \cite[\href{https://kerodon.net/tag/0478}{Tag 0478}]{kerodon},
the maps $\pi$ and $\pi'$ are cocartesian fibrations and $\phi$
carries $\pi$-cocartesian morphisms to $\pi'$-cocartesian morphisms.
Therefore, for each morphism $\alpha:C\to C'$, the square 
\[\begin{tikzcd}
	{\mathcal{C}^{/C}\times _{\mathcal{C}}\mathcal{W}} & {(\mathcal{D}^{/f(C)})^\simeq} \\
	{\mathcal{C}^{/C'}\times _{\mathcal{C}}\mathcal{W}} & {(\mathcal{D}^{/f(C')})^\simeq}
	\arrow["{\alpha_!}"', from=1-1, to=2-1]
	\arrow[from=2-1, to=2-2]
	\arrow[from=1-1, to=1-2]
	\arrow["{f(\alpha)_!}", from=1-2, to=2-2]
\end{tikzcd}\]consisting of fibers of $\pi$ and $\pi'$ commutes up to natural
equivalence. By hypothesis (and the equivalence of the upper slices
and lower slices \cite[Proposition 4.2.1.5]{HTT}), the horizontal
arrows are weak homotopy equivalences. Also, since $f\pr{\alpha}$
is an equivalence, the functor $f\pr{\alpha}_{!}$ is a homotopy equivalence.
Hence $\alpha_{!}$ is a weak homotopy equivalence. It follows from
Proposition \ref{prop:Theorem_B} that $\pi$ is locally constant,
as desired.

Next, we show that the Reedy fibrant replacement of $N\pr{\cal C,\cal W}$
is a Segal space. Recall that a bisimplicial set $X$, regarded as
a simplicial object in simplicial set, is called a Segal space \cite[$\S$ 4.1]{Rezk01}
if it is Reedy fibrant and the map $X_{k}\to X_{1}\times_{X_{0}}\cdots\times_{X_{0}}X_{1}$
is a homotopy equivalence for every $k\geq2$. Equivalently, $X$
is a Segal space if and only if it is Reedy fibrant and the square
\[\begin{tikzcd}
	{X_k} & {X_1} \\
	{X_{k-1}} & {X_0}
	\arrow[from=1-1, to=1-2]
	\arrow[from=1-1, to=2-1]
	\arrow[from=1-2, to=2-2]
	\arrow[from=2-1, to=2-2]
\end{tikzcd}\]is homotopy cartesian for all $k\geq2$, where the maps in the diagrams
are induced by the inclusions $\{1\}\to\{0,1\}\hookrightarrow[k]$
and $\{1\}\to\{1,\dots,k\}\hookrightarrow[k]$. Consequently, it suffices
to show that for each $k\geq2$, the square 
\[\begin{tikzcd}
	{\operatorname{Fun}^{\mathcal{W}}(\Delta^{k},\mathcal{C})} & {\operatorname{Fun}^{\mathcal{W}}(\Delta^{1},\mathcal{C})} \\
	{\operatorname{Fun}^{\mathcal{W}}(\Delta^{\{1,\dotsm,k\}},\mathcal{C})} & {\operatorname{Fun}^{\mathcal{W}}(\{1\},\mathcal{C})}
	\arrow["\pi", from=1-2, to=2-2]
	\arrow[from=2-1, to=2-2]
	\arrow[from=1-1, to=2-1]
	\arrow[from=1-1, to=1-2]
\end{tikzcd}\]is homotopy cartesian in the Kan--Quillen model structure. Since
$\pi$ is locally constant, this is equivalent to the assertion that
the map
\[
\Fun^{\cal W}\pr{\Delta^{k},\cal C}\to\Fun^{\cal W}\pr{\Delta^{1}\amalg_{\{1\}}\Delta^{\{1,\dots,k\}},\cal C}
\]
be a weak homotopy equivalence. But this map is a trivial fibration,
for the inclusion $\Delta^{1}\amalg_{\{1\}}\Delta^{\{1,\dots,k\}}\subset\Delta^{k}$
is a weak categorical equivalence \cite[Lemma 5.4.5.10]{HTT}.

Now we show that (1$'$)$\implies$(2). Suppose that condition (1$'$)
holds. We wish to show that the map $N\pr{\cal C,\cal W}_{0}\to N\pr{\cal D}_{0}$
is a weak homotopy equivalence. Choose a trivial Reedy cofibration
$i:N\pr{\cal C,\cal W}\to\widetilde{N}\pr{\cal C,\cal W}$ with $\widetilde{N}\pr{\cal C,\cal W}$
Reedy fibrant (with respect to the Kan--Quillen model structure on
$\SS$). Since $N\pr{\cal D}$ is a complete Segal space, it is Reedy
fibrant, so the map $N\pr{\cal C,\cal W}\to N\pr{\cal D}$ factors
as
\[
N\pr{\cal C,\cal W}\xrightarrow{i}\widetilde{N}\pr{\cal C,\cal W}\xrightarrow{F}N\pr{\cal D}.
\]
Since $i$ is a levelwise weak homotopy equivalence, it suffices to
show that the map $\widetilde{N}\pr{\cal C,\cal W}_{0}\to N\pr{\cal D}_{0}$
is a weak homotopy equivalence. By hypothesis, the map $F$ is a weak
equivalence of the complete Segal space model structure. Therefore,
it suffices to show that the Segal space $\widetilde{N}\pr{\cal C,\cal W}$
is a \textit{complete} Segal space (for weak equivalences of the complete
Segal space model structure between complete Segal spaces are nothing
but levelwise homotopy equivalences). 

Let $\widetilde{N}\pr{\cal C,\cal W}_{{\rm hoeq}}\subset\widetilde{N}\pr{\cal C,\cal W}_{1}$
denote the union of the components whose vertices are homotopy equivalences
of the Segal space $\widetilde{N}\pr{\cal C,\cal W}$. We must show
that the map
\[
\theta:\widetilde{N}\pr{\cal C,\cal W}_{0}\to\widetilde{N}\pr{\cal C,\cal W}_{{\rm hoeq}}
\]
is a homotopy equivalence. Since $\widetilde{N}\pr{\cal C,\cal W}$
is a Segal space, the map $F$ is a Dwyer--Kan equivalence of Segal
spaces \cite[Theorem 7.1]{Rezk01}. Therefore, given a morphism $\alpha$
of $\cal C$, the morphism $i\pr{\alpha}$ is a homotopy equivalence
of $\widetilde{N}\pr{\cal C,\cal W}$ if and only if $Fi\pr{\alpha}$
is a homotopy equivalence of $N\pr{\cal D}$. By the definition of
$\cal W$, the latter condition holds if and only if $\alpha$ belongs
to $\cal W$. Therefore, the inverse image of $\widetilde{N}\pr{\cal C,\cal W}_{{\rm hoeq}}\subset\widetilde{N}\pr{\cal C,\cal W}_{1}$
under the weak homotopy equivalence
\[
i_{1}:\Fun^{\cal W}\pr{\Delta^{1},\cal C}\xrightarrow{\simeq}\widetilde{N}\pr{\cal C,\cal W}_{1}
\]
is the simplicial subset $\Fun\pr{\Delta^{1},\cal W}\subset\Fun^{\cal W}\pr{\Delta^{1},\cal C}$.
In particular (since $\widetilde{N}\pr{\cal C,\cal W}_{{\rm hoeq}}$
is a union of components) the map $i_{1}$ restricts to a weak homotopy
equivalence
\[
i_{1}':\Fun\pr{\Delta^{1},\cal W}\xrightarrow{\simeq}\widetilde{N}\pr{\cal C,\cal W}_{{\rm hoeq}}.
\]
Now consider the commutative diagram 
\[\begin{tikzcd}
	{\mathcal{W}} & {\operatorname{Fun}(\Delta^1,\mathcal{W})} & {\operatorname{Fun}^{\mathcal{W}}(\Delta^1,\mathcal{C})} \\
	{\widetilde{N}(\mathcal{C},\mathcal{W})_{0}} & {\widetilde{N}(\mathcal{C},\mathcal{W})_{\mathrm{hoeq}}} & {\widetilde{N}(\mathcal{C},\mathcal{W})_{1}.}
	\arrow[from=1-2, to=1-3]
	\arrow["{i'_1}", from=1-2, to=2-2]
	\arrow["{i_1}", from=1-3, to=2-3]
	\arrow[from=2-2, to=2-3]
	\arrow["\theta"', from=2-1, to=2-2]
	\arrow["{\theta'}", from=1-1, to=1-2]
	\arrow["{i_0}", from=1-1, to=2-1]
	\arrow["\simeq"', from=1-1, to=2-1]
	\arrow["\simeq"', from=1-2, to=2-2]
	\arrow["\simeq"', from=1-3, to=2-3]
\end{tikzcd}\]The maps $i_{0},i_{1}',i_{1}$ are weak homotopy equivalences, and
so is the map $\theta'$ (for it is a left (and a right) adjoint).
Therefore, the map $\theta$ is a weak homotopy equivalence, as required.

Next we show (2)$\implies$(1$'$). It will suffice to show that,
for each $n\geq0$, the map $N\pr f_{n}:N\pr{\cal C,\cal W}_{n}\to N\pr{\cal D}_{n}$
is a weak homotopy equivalence. Since $\widetilde{N}\pr{\cal C,\cal W}$
is a Segal space, we only need to prove this in the case where $n\in\{0,1\}$.
If $n=0$, the claim follows from our hypothesis (2). If $n=1$, we
consider the commutative diagram
\[\begin{tikzcd}
	{\operatorname{Fun}^{\mathcal{W}}(\Delta^{1},\mathcal{C})} & {\operatorname{Fun}(\Delta^{1},\mathcal{D})^\simeq} \\
	{\operatorname{Fun}^{\mathcal{W}}(\{1\},\mathcal{C})} & {\operatorname{Fun}(\{1\},\mathcal{D})^\simeq.}
	\arrow["\pi"', from=1-1, to=2-1]
	\arrow["{\pi'}", from=1-2, to=2-2]
	\arrow["{N(f)_0}"', from=2-1, to=2-2]
	\arrow["{N(f)_1}", from=1-1, to=1-2]
\end{tikzcd}\]As we already saw in the second paragraph of the proof, the vertical
maps are locally constant cocartesian fibrations. (The map $\pi'$
is even a Kan fibration). Also, our hypothesis ensures that, for each
$C\in\Fun^{\cal W}\pr{\{1\},\cal C}$, the induced map $\pi^{-1}\pr C\to\pr{\pi'}^{-1}\pr{f\pr C}$
between the fibers is a weak homotopy equivalence. Hence the square
is homotopy cartesian in the Kan--Quillen model structure. Since
$N\pr f_{0}$ is a weak homotopy equivalence, so must be $N\pr f_{1}$,
and the proof is complete.
\end{proof}

\begin{cor}
\label{cor:localization_criterion}Let $\cal C$ and $\cal D$ be
$\infty$-categories, let $\cal C'\subset\cal C$ and $\cal D'\subset\cal D$
be full subcategories, and let $f:\cal C\to\cal D$ be a functor which
carries $\cal C'$ into $\cal D'$. Set $\cal W=\cal C\times_{\cal D}\cal D^{\simeq}$.
Suppose that, for each $C\in\cal C$, the functor
\[
\cal C'_{/C}\times_{\cal C}\cal W\to\pr{\cal D'_{/f\pr C}}^{\simeq}
\]
is a weak homotopy equivalence. Then for each $C\in\cal C$, the functor
\[
\cal C'_{/C}\to\cal D'_{/f\pr C}
\]
is a localization with respect to the morphisms whose images in $\cal D'_{f/\pr C}$
are equivalences.
\end{cor}

\begin{proof}
We apply Proposition \ref{prop:localization_criterion} to the functor
$\cal C'_{/C}\to\cal D'_{/f\pr C}$. It suffices to show that, for
each object $\pr{\alpha:C'\to C}\in\cal C'_{/C}$, the map
\[
\theta:\cal C'_{/\alpha}\times_{\cal C}\cal W\to\pr{\cal D'_{/f\pr{\alpha}}}^{\simeq}
\]
is a weak homotopy equivalence. Consider the commutative diagram
\[\begin{tikzcd}
	{\mathcal{C}'_{/\alpha}\times _{\mathcal{C}}\mathcal{W}} & {(\mathcal{D}'_{/f(\alpha)})^\simeq} \\
	{\mathcal{C}'_{/C'}\times _{\mathcal{C}}\mathcal{W}} & {(\mathcal{D}'_{/f(C')})^\simeq.}
	\arrow["\theta", from=1-1, to=1-2]
	\arrow[from=1-2, to=2-2]
	\arrow[from=1-1, to=2-1]
	\arrow["{\theta'}"', from=2-1, to=2-2]
\end{tikzcd}\]The vertical maps are trivial fibrations, and $\theta'$ is a weak
homotopy equivalence by hypothesis. Hence $\theta$ is a weak homotopy
equivalence, as desired.
\end{proof}

To apply Corollary \ref{cor:localization_criterion} to our context,
we need a lemma. For each $n,k\geq0$, write $\D^{=k}_{n}$ and $\Disk^{=k}_{n}$
for the full subcategories of $\D_{n}$ and $\D_{n}$ spanned by the
objects homeomorphic to $\bb R^{n}\times\{1,\dots,k\}$. Note that,
by Kister's theorem \cite{Kister_MFB}, the subcategory $\D^{=k}_{n/M}\times_{\Mfld_{n}}\Mfld^{\simeq}_{n}\subset\sf{Disk}^{=k}_{n/M}$
is spanned by the morphisms which induce bijections between the sets
of components.
\begin{lem}
\label{lem:loc_lem}Let $n,k\geq0$, and let $M$ be an $n$-manifold.
The map
\[
\D^{=k}_{n/M}\times_{\Mfld_{n}}\Mfld^{\simeq}_{n}\to\pr{\Disk^{=k}_{n/M}}^{\simeq}
\]
is a weak homotopy equivalence.
\end{lem}

For the proof of Lemma \ref{lem:loc_lem}, we introduce a bit of notation.
\begin{notation}
\label{nota:B'_k}Let $k\geq0$, and let $M$ be a manifold. 
\begin{itemize}
\item We will write $B'_{k}\pr M\subset B_{k}\pr M$ for the open set consisting
of those subsets $S\subset M$ of cardinality $k$ such that the map
$\pi_{0}\pr S\to\pi_{0}\pr M$ is injective.
\item Let $k\geq0$, and let $M$ be a manifold. We let $I^{=k}_{M}$ denote
the (non-full) subcategory of $\Disj\pr M$ spanned by the isotopy
equivalences between the elements of $\Disj\pr M$ with exactly $k$
components.
\end{itemize}
\end{notation}

\begin{proof}
[Proof of Lemma \ref{lem:loc_lem}]The assignment $U\mapsto\pr{U\hookrightarrow M}$
induces a categorical equivalence $I^{=k}_{M}\xrightarrow{\simeq}\D^{=k}_{n/M}\times_{\Mfld_{n}}\Mfld^{\simeq}_{n}$,
so it suffices to show that the composite
\[
\theta:I^{=k}_{M}\to\D^{=k}_{n/M}\times_{\Mfld_{n}}\Mfld^{\simeq}_{n}\to\pr{\Disk^{=k}_{n/M}}^{\simeq}
\]
is a weak homotopy equivalence.

Given $n$-manifolds $X,Y$, let $\Emb^{\simeq}\pr{X,Y}\subset\Emb\pr{X,Y}$
denote the subspace consisting of the isotopy equivalences from $X$
to $Y$. According to Proposition \ref{prop:final_map_into_slice_2}
and \cite[Proposition 4.2.4.1]{HTT}, the map $\theta$ is a weak
homotopy equivalence if and only if the following condition holds:
\begin{itemize}
\item [($\ast$)]For each object $U\in\Disk^{=k}_{n}$, the map
\[
\hocolim_{V\in I^{=k}_{M}}\Sing\Emb^{\simeq}\pr{U,V}\to\Sing\Emb\pr{U,M}
\]
is a weak homotopy equivalence.
\end{itemize}
We will prove ($\ast$). The claim is obvious if $k=0$, so assume
that $k\geq1$. Fix a homeomorphism $U\cong\bb R^{n}\times\{1,\dots,k\}$
and let $V_{1},\dots,V_{k}\subset V$ denote the components of $V$.
Evaluation at the origin gives us a commutative diagram 
\[\begin{tikzcd}
	{\operatorname{Emb}^{\simeq}(U,V)} & {\operatorname{Emb}(U,V)} & {\operatorname{Emb}(U,M)} \\
	{\coprod_{\sigma\in \Sigma_k}\prod_{i=1}^kV_{\sigma(i)}} & {\operatorname{Conf}(k,V)} & {\operatorname{Conf}(k,M)} \\
	{B'_{k}(V)} & {B_k(V)} & {B_k(M)}
	\arrow[from=1-1, to=2-1]
	\arrow[from=1-1, to=1-2]
	\arrow[from=1-2, to=2-2]
	\arrow[from=1-2, to=1-3]
	\arrow[from=1-3, to=2-3]
	\arrow[from=2-2, to=2-3]
	\arrow[from=2-1, to=2-2]
	\arrow[from=2-3, to=3-3]
	\arrow[from=2-2, to=3-2]
	\arrow[from=3-2, to=3-3]
	\arrow[from=2-1, to=3-1]
	\arrow[from=3-1, to=3-2]
	\arrow["{(1)}"{description}, draw=none, from=1-2, to=2-3]
	\arrow["{(3)}"{description}, draw=none, from=1-1, to=2-2]
	\arrow["{(4)}"{description}, draw=none, from=2-1, to=3-2]
	\arrow["{(2)}"{description}, draw=none, from=2-2, to=3-3]
\end{tikzcd}\]of topological spaces. Proposition \ref{prop:emb_conf_cart} shows
that square (1) is homotopy cartesian. Squares (2), (3), (4) are homotopy
cartesian because they are strictly cartesian and their vertical arrows
are Serre fibrations. It follows that the outer square is homotopy
cartesian. Since colimits in $\cal S$ are universal \cite[Lemma 6.1.3.14]{HTT},
we are reduced to showing that the map
\[
\hocolim_{V\in I^{=k}_{M}}\Sing B'_{k}\pr V\to\Sing B_{k}\pr M
\]
is a weak homotopy equivalence. This follows from \cite[Theoerm A.3.1]{HA}.
\end{proof}

We now arrive at the proof of Theorem \ref{thm:localizing_wrt_istpy}.

\begin{proof}
[Proof of Theorem \ref{thm:localizing_wrt_istpy}]We will apply Corollary
\ref{cor:localization_criterion} to the functor $\sf{Mfld}_{n}\to\Mfld_{n}$
and the full subcategories $\D_{n}\subset\sf{Mfld}_{n}$ and $\Disk_{n}\subset\Mfld_{n}$
for part (1), and the full subcategories $\D^{\leq k}_{n}\subset\sf{Mfld}_{n}$
and $\Disk^{\leq k}_{n}\subset\Mfld_{n}$ for part (2). For part (1),
we must show that for each $n$-manifold $M$, the functor
\[
\D_{n/M}\times_{\Mfld_{n}}\Mfld^{\simeq}_{n}\to\pr{\Disk_{n/M}}^{\simeq}
\]
is a weak homotopy equivalence, and for part (2), we must show that
the functor
\[
\D^{\leq k}_{n/M}\times_{\Mfld_{n}}\Mfld^{\simeq}_{n}\to\pr{\Disk^{\leq k}_{n/M}}^{\simeq}
\]
is a weak homotopy equivalence. Both of these assertions follow from
Lemma \ref{lem:loc_lem}.
\end{proof}

\begin{rem}
\label{rem:loc_sm}The proof of Theorem \ref{thm:localizing_wrt_istpy}
carries over to the smooth case (using Remark \ref{rem:emb_conf_cart}).
In other words, for each  $n,k\geq0$ and each smooth $n$-manifold
$M$, the functors $\sf{Disk}^{\leq k}_{\sm,n/M}\to\Disk^{\leq k}_{\sm,n/M}$
and $\sf{Disk}_{\sm,n/M}\to\Disk_{\sm,n/M}$ are localization with
respect to smooth isotopy equivalences.
\end{rem}

\begin{rem}
In \cite[Lemma A.1, A.2]{KSW24}, Karlsson, Scheimbauer, and Walde
independently proved Proposition \ref{prop:Theorem_B} and the implication
(2)$\implies$(1) of Proposition \ref{prop:localization_criterion}.
\end{rem}

\begin{variant}
\label{var:o-small}Let $n\geq0$, and let $M$ be an $n$-manifold.
Let $\cal O$ be a basis of the topology of $M$ whose elements are
homeomorphic to $\bb R^{n}$, and let $\sf{Disk}^{\cal O}_{n/M}\subset\sf{Disk}_{n/M}$
denote the full subcategory spanned by the objects whose images are
finite disjoint union of elements in $\cal O$. The functor $\sf{Disk}^{\cal O}_{n/M}\to\Disk_{n/M}$
is a localization at isotopy equivalences.

Indeed, by applying Corollary \ref{cor:localization_criterion} to
the full subcategories $\sf{Disk}^{\cal O}_{n/M}\subset\pr{\sf{Disk}^{\cal O}_{n/M}}^{\rcone}$
and $\Disk_{n/M}\subset\pr{\Disk_{n/M}}^{\rcone}$, we are reduced
to showing that the maps
\[
\sf{Disk}^{\cal O}_{n/N}\times_{\Mfld_{n}}\Mfld^{\simeq}_{n}\to\pr{\Disk_{n/N}}^{\simeq}
\]
is a weak homotopy equivalence for all $N\in\Mfld_{n}$, which can
be proved exactly as in Lemma \ref{lem:loc_lem}. 

By a similar argument, we find that the functors 
\[
\sf{Disk}^{\cal O}_{n/M}\cap\sf{Disk}^{\leq k}_{n/M}\to\Disk^{\leq k}_{n/M}
\]
is a localizations at isotopy equivalences for every $k\geq0$.
\end{variant}

\subsection{Global case}

Ayala--Francis's theorem (Theorem \ref{thm:localizing_wrt_istpy}),
the main result of the previous subsection, is crucial in the development
of non-context-free manifold calculus. In this section, we consider
a context-free analog of the localization theorem. The results in
this section will not be used elsewhere in this paper.

It turns out that the localization theorem is a bit nuanced in the
context-free case. (see also {[}AF20, 2.2.13{]} for a relevant observation,
and \cite[Corollary 2.7]{DWW03} for a closely related result.) The
validity of the theorem depends on which category (topological or
smooth) we work in, as the following theorem shows.
\begin{thm}
\label{thm:nonloc}Let $n,k\geq1$.
\begin{enumerate}
\item Both of the functors $\sf{Disk}^{\leq k}_{n}\to\Disk^{\leq k}_{n}$
and $\sf{Disk}_{n}\to\Disk_{n}$ are localizations with respect to
isotopy equivalences.
\item Neither of the functors $\sf{Disk}^{\leq k}_{\sm,n}\to\Disk^{\leq k}_{\sm,n}$
and $\sf{Disk}_{\sm,n}\to\Disk_{\sm,n}$ is a localization with respect
to smooth isotopy equivalences.
\end{enumerate}
\end{thm}

\begin{proof}
We begin with (1). We will focus on the functor $\sf{Disk}_{n}\to\Disk_{n}$;
the proof that the functor $\sf{Disk}^{\leq k}_{n}\to\Disk^{\leq k}_{n}$
is a localization is similar. By Proposition \ref{prop:localization_criterion}
and Lemma \ref{lem:loc_lem}, it suffices to show that the functor
\[
\D_{n}\times_{\Disk_{n}}\Disk^{\simeq}_{n}\to\Disk^{\simeq}_{n}
\]
is a weak homotopy equivalence. For convenience, in this proof we
will replace $\D_{n}$ and $\Disk_{n}$ by their full subcategories
spanned by the objects $\{\bb R^{n}\times\{1,\dots,k\}\}_{k=0,1,\dots}$
and still denote them by $\D_{n}$ and $\Disk_{n}$.

Regard the $k$th symmetric group $\Sigma_{k}$ as a category with
one object $\{1,\dots,k\}$ with morphism given by bijections, and
let $N\Sigma_{k}$ denote its nerve. We consider the commutative diagram
\[\begin{tikzcd}
	{\mathsf{Disk}_n\times_{\mathcal{D}\mathsf{isk}_n}\mathcal{D}\mathsf{isk}_n^\simeq} && {\mathcal{D}\mathsf{isk}_n^\simeq} \\
	& {\coprod_{k\geq 0}N\Sigma_k}
	\arrow[from=1-1, to=1-3]
	\arrow["q", from=1-3, to=2-2]
	\arrow["p"', from=1-1, to=2-2]
\end{tikzcd}\]where $p$ and $q$ are given by $f\mapsto\pi_{0}\pr f$. The functors
$p$ and $q$ are cartesian fibrations. Since every morphism of $\coprod_{k\geq0}N\Sigma_{k}$
is an equivalence, both $p$ and $q$ are locally constant (Definition
\ref{def:loc_const}). Therefore, it suffices to show that for each
$k\geq0$, the map
\[
\theta_{k}:p^{-1}\pr{\{1,\dots,k\}}\to q^{-1}\pr{\{1,\dots,k\}}
\]
is a weak homotopy equivalence. 

We can identify the map $\theta_{k}$ with the map
\[
\pr{N\Emb\pr{\bb R^{n},\bb R^{n}}_{\delta}}^{k}\to\pr{N\Emb\pr{\bb R^{n},\bb R^{n}}}^{k},
\]
where $N\Emb\pr{\bb R^{n},\bb R^{n}}$ denotes the homotopy coherent
nerve of the topological monoid $\Emb\pr{\bb R^{n},\bb R^{n}}$ (regarded
as a topological category with a single object) and $N\Emb\pr{\bb R^{n},\bb R^{n}}_{\delta}$
denotes the nerve of the same category with the discrete topology.
It thus suffices to show that the map
\[
N\Emb\pr{\bb R^{n},\bb R^{n}}_{\delta}\to N\Emb\pr{\bb R^{n},\bb R^{n}}
\]
is a weak homotopy equivalence. This is a consequence of McDuff's
theorem \cite[Corollary 2.15]{McDuff80}. (For a comparison between
the homotopy coherent nerve and McDuff's model of classifying spaces,
see \cite[Corollary 4.2]{A24a}.)

Next, we prove (2). As in the previous paragraph, it suffices to show
that the map
\[
N\Emb_{\sm}\pr{\bb R^{n},\bb R^{n}}_{\delta}\to N\Emb_{\sm}\pr{\bb R^{n},\bb R^{n}}
\]
is \textit{not} a weak homotopy equivalence. Arguing as in \cite[Proof of Theorem 1.1]{McDuff80},
we can reduce this to showing that the map $B\Gamma^{\infty}_{n}\to BGL_{n}\pr{\bb R}$
is not a weak homotopy equivalence of topological spaces, where $\Gamma^{\infty}_{n}$
denotes Haefliger's groupoid \cite[p.143]{Haefliger71} for codimension
$n$ smooth foliations. (See also Subsection \ref{subsec:final_remark}.)
This is a consequence of Bott's theorem \cite{Bott70} (see \cite[p. 143, I.8 (a)]{Haefliger71}).
\end{proof}

\subsection{\label{subsec:final_remark}A Remark on Classifying Spaces of Groupoids
Internal to \texorpdfstring{$\mathsf{Top}$}{Top}}

A little care is necessary in the final step of Theorem \ref{thm:nonloc},
because in \cite{McDuff80} and \cite{Haefliger71}, McDuff and Haefliger
use different models of classifying spaces of groupoids internal to
the category $\sf{Top}$ of topological spaces. The equivalence of
the two models is probably folklore, but the author is not aware of
a convenient reference for this. Thus we record a proof below.

We start by recalling McDuff and Haefliger's models.
\begin{construction}
Let $\cal G$ be a groupoid internal to $\sf{Top}$. Abusing notation,
let $\cal G$ denote the simplicial topological space corresponding
to $\cal G$. (Thus $\cal G_{1}=\mor\cal G$, $\cal G_{2}=\mor\cal G\times_{\ob\cal G}\mor\cal G$,
etc.) McDuff and Haefliger define the classifying space of $\cal G$
as follows:
\begin{itemize}
\item In \cite{McDuff80}, McDuff defines the classifying space of $\cal G$
to be the fat realization of the simplicial topological space $\cal G$.
We denote this topological space by $B_{{\rm MD}}\pr{\cal G}$.\footnote{This model was introduced by Segal {\cite{Segal1968}}.}
\item Haefliger's model is slightly more complicated. First, consider the
space $E_{{\rm Hae}}\cal G$ whose points are the equivalence classes
of formal expressions of the form
\[
\pr{t_{0},g_{0},t_{1},g_{1},\dots}
\]
where $g_{0},g_{1},\dots$ are morphisms of $\cal G$ with a common
codomain and $t_{0},t_{1},\dots$ are non-negative real numbers, all
but finitely many of which are zero, and $\sum^{\infty}_{i=0}t_{i}=1$.
Two expressions $\pr{t_{0},g_{0},t_{1},g_{1},\dots}$ and $\pr{t'_{0},g'_{0},t'_{1},g'_{1},\dots}$
are equivalent if $t_{i}=t_{i}'$ for all $i$ and $g_{i}=g'_{i}$
whenever $t_{i}>0$. The equivalence class of $\pr{t_{0},g_{0},t_{1},g_{1},\dots}$
will be denoted by $\bigoplus_{i\geq0}t_{i}g_{i}=t_{0}g_{0}\oplus t_{1}g_{1}\oplus\cdots$.
We will write $\tau_{i}:E_{{\rm Hae}}\cal G\to[0,1]$ for the set
map $\tau_{i}\pr{\bigoplus_{j\geq0}t_{j}g_{j}}=t_{i}$, and define
$p_{i}:\tau^{-1}_{i}\pr{(0,1]}\to\mor\cal G$ by $p_{i}\pr{\bigoplus_{j\geq0}t_{j}g_{j}}=g_{i}$.
We topologize $E_{{\rm Hae}}\cal G$ so that it has the following
universal property: If $X$ is a topological space, then a set map
$X\to E_{{\rm Hae}}\cal G$ is continuous if and only if the composites
$\tau_{i}f:X\to[0,1]$ and $p_{i}f:\pr{\tau_{i}f}^{-1}\pr{(0,1]}\to\mor\cal G$
are continuous. 

Now introduce an equivalence relation on $E_{{\rm Hae}}\cal G$ as
follows: Two points $\bigoplus_{j\geq0}t_{j}g_{j}$ and $\bigoplus_{j\geq0}t'_{j}g'_{j}$
are equivalent if there is some morphism $g$ of $\cal G$ such that
$\bigoplus_{j\geq0}t_{j}gg_{j}=\bigoplus_{j\geq0}t_{j}gg'_{j}$. The
resulting quotient space will be denoted by $B_{{\rm Hae}}\cal G$.
We still denote by $\tau_{i}:B_{{\rm Hae}}\cal G\to[0,1]$ the map
induced by $\tau_{i}:E_{{\rm Hae}}\cal G\to[0,1]$. The space $B_{{\rm Hae}}\cal G$
is the model of the classifying space of $\cal G$ used by Haefliger
in \cite{Haefliger71}.\footnote{This model was introduced by Milnor in {\cite{Milnor56}}.}
\end{itemize}
\end{construction}

These models are equivalent in the following sense:
\begin{prop}
\label{prop:comparison}Let $\cal G$ be a groupoid internal to $\sf{Top}$.
There is a zig-zag of weak homotopy equivalences between $B_{{\rm MD}}\cal G$
and $B_{{\rm Hae}}\cal G$, which is natural in $\cal G$.
\end{prop}

Our proof of Proposition \ref{prop:comparison} closely relies on
a few lemmas, whose ideas can be traced back to Segal's original writing
\cite[$\S$3]{Segal1968}. Write $N^{{\rm nd}}\pr{\bb Z_{\geq0}}$
for the semisimplicial set of nondegenerate simplices of the nerve
of the poset $\bb Z_{\geq0}$.
\begin{lem}
\label{lem:comparison}Let $\cal G$ be a groupoid internal to $\sf{Top}$.
There is a bijective continuous map
\[
\phi:\norm{\cal G\times N^{{\rm nd}}\pr{\bb Z_{\geq0}}}\to B_{{\rm Hae}}\cal G,
\]
which is a homotopy equivalence.
\end{lem}

\begin{proof}
We will write $X=\norm{\cal G\times N^{{\rm nd}}\pr{\bb Z_{\geq0}}}$.
By definition, a point of $X$ can be represented by a sequence
\[
\pr{\pr{t_{0},\dots,t_{n}},x_{0}\xrightarrow{f_{1}}\cdots\xrightarrow{f_{n}}x_{n},k_{0}<\cdots<k_{n}},
\]
where $\pr{t_{0},\dots,t_{n}}$ is a point of $\abs{\Delta^{n}}$,
$x_{0}\xrightarrow{f_{1}}\cdots\xrightarrow{f_{n}}x_{n}$ is a point
of $\cal G_{n}$, and $k_{0}<\cdots<k_{n}$ is an element of $N^{{\rm nd}}\pr{\bb Z_{\geq0}}_{n}$.
We declare that $\phi$ carries such a point to the point 
\[
\pr{\bigoplus_{0\leq i<k_{0}}0}\oplus t_{0}f_{n}\dots f_{1}\oplus\pr{\bigoplus_{k_{0}<i<k_{1}}0}\oplus t_{1}f_{n}\dots f_{2}\oplus\cdots\oplus t_{n}\id_{x_{n}}\oplus\pr{\bigoplus_{i<k_{n}}0}.
\]
This defines a bijection $\phi:X\to B_{{\rm Hae}}\cal G$, which is
continuous by the definition of the topology on $B\cal G$. 

To show that $\phi$ is a homotopy equivalence, we first construct
a homotopy inverse of $\phi$. Use \cite[13.1.7]{DieckAT} to find
a partition of unity $\pr{\psi_{j}}^{\infty}_{j=0}$ on $B\cal G$
subordinate to the cover $\pr{\tau^{-1}_{j}(0,1]}^{\infty}_{j=0}$.
We define a set map $\psi:B_{{\rm Hae}}\cal G\to X$ as follows: Let
$b=\bigoplus^{\infty}_{i=0}t_{i}g_{i}\in B_{{\rm Hae}}\cal G$ be
a point. Let $k_{0}<\dots<k_{n}$ be the enumeration of the integers
$i\geq0$ such that $t_{i}>0$. Then $\psi\pr b$ is represented by
the point
\[
\pr{\pr{\psi_{k_{0}}\pr bt_{k_{0}},\dots,\psi_{k_{n}}\pr bt_{k_{n}}},x_{0}\xrightarrow{g^{-1}_{k_{1}}g_{k_{0}}}\cdots\xrightarrow{g^{-1}_{k_{n}}g_{k_{n-1}}}x_{n},k_{0}<\dots<k_{n}},
\]
where $x_{i}$ denotes the domain of $g_{k_{i}}$. To check the continuity
of $\psi$, choose a neighborhood $U$ of $b$ such that $\psi_{j}\vert U=0$
for all $j\in\bb Z_{\geq0}\setminus\{k_{0},\dots,k_{n}\}$. (This
is possible because $\opn{supp}\psi_{j}\subset\tau^{-1}_{j}\pr{(0,1]}$
and the supports of $\{\psi_{j}\}_{j}$ are locally finite.) Without
loss of generality, we may assume that $\tau_{k_{0}},\dots,\tau_{k_{n}}$
are positive on $U$. Then on $U$, the map $\psi$ can be written
as a composite
\[
U\to\abs{\Delta^{k}}\times\cal G_{n}\times\{k_{0}<\cdots<k_{n}\}\to X,
\]
each of which is continuous. This proves that $\psi$ is continuous.

We claim that $\psi$ is a homotopy inverse of $\phi$. Define a map
$h:X\times[0,1]\to X$ as follows: Let $x=[\pr{t_{0},\dots,t_{n}},\sigma,k_{0}<\cdots<k_{n}]$
be a point of $X$ and let $0\leq s\leq1$. Then $h\pr{x,s}$ is represented
by
\[
\pr{\pr{\pr{\pr{1-s}+s\psi_{k_{0}}\pr{\phi\pr x}}t_{0},\dots,\pr{\pr{1-s}+s\psi_{k_{0}}\pr{\phi\pr x}}t_{n}},\sigma,k_{0}<\dots<k_{n}}.
\]
This defines a homotopy from $\id_{X}$ to $\psi\phi$. Likewise,
there is a map $B_{{\rm Hae}}\cal G\times[0,1]\to B_{{\rm Hae}}\cal G$
given by
\[
\pr{b,s}=\pr{\bigoplus_{i\ge0}t_{i}g_{i},s}\mapsto\bigoplus_{i\geq0}\pr{\pr{1-s}+s\psi_{i}\pr b}t_{i}g_{i},
\]
which is a homotopy from $\id_{B\cal G}$ to $\phi\psi$. The claim
follows.
\end{proof}

For the next lemma, we write $\Del_{\inj}\subset\Del$ for the subcategory
spanned by the injective poset maps. We also identify sets with discrete
simplicial sets. We remark that, while the nerve of $\Del^{\op}$
is sifted \cite[Lemma 5.5.8.4]{HTT}, the nerve of $\Del^{\op}_{\inj}$
is not. Therefore, homotopy colimits over $\Del^{\op}_{\inj}$ do
not generally commute with products, even up to weak equivalence.
The following lemma provides an instance in which such commutation
does occur.
\begin{lem}
\label{lem:comparison2}Let $\cal C$ be a category internal to $\SS$.
The map
\[
\hocolim_{\Del^{\op}_{\inj}}\pr{\cal C\times N^{\mathrm{nd}}\pr{\bb Z_{\geq0}}}\to\hocolim_{\Del^{\op}_{\inj}}\cal C
\]
is a weak homotopy equivalence.
\end{lem}

\begin{proof}
We will prove the lemma by using specific models of homotopy colimits,
which we now introduce. For a semisimplicial simplicial set $A\in\SS^{\Del^{\op}_{\inj}}$,
we model its homotopy colimit by the coend
\[
\Phi\pr A=\int^{[n]\in\Del_{\inj}}A_{n}\times\Delta^{n}.
\]
To see that $\Phi:\SS^{\Del^{\op}_{\inj}}\to\SS$ is actually a model
of homotopy colimits, observe that its right adjoint $S\mapsto S^{\Delta^{\bullet}}$
is right Quillen for the Reedy model structure by inspection, and
that it is naturally weakly equivalent to the diagonal functor. Since
every semisimplicial simplicial set is Reedy cofibrant, this means
that $\Phi$ is a model of homotopy colimits.

Likewise, since the diagonal of bisimplicial sets carries levelwise
weak homotopy equivalences to equivalences \cite[Chapter IV, Proposition 1.7]{GJ99},
we find that the homotopy colimit of a simplicial simplicial set (i.e.,
bisimplicial set) $B\in\SS^{\Del^{\op}}$ can be computed by its diagonal
(which coincides with the coend $\int^{[n]\in\Del}B_{n}\times\Delta^{n}$).
We note that with this model, every simplicial set is the homotopy
colimit of its simplices, i.e., $S=\hocolim_{[m]\in\Del^{\op}}S_{m}$
for all simplicial set $S$.

Now let $X\in\SS^{\Del^{\op}_{\inj}}$ be a semisimplicial simplicial
set. Using the remark at the end of the previous paragraph, we obtain
a chain of weak homotopy equivalences of simplicial sets
\begin{align*}
\hocolim_{[n]\in\Del^{\op}_{\inj}}\pr{X_{n}} & \cong\hocolim_{[n]\in\Del^{\op}_{\inj}}\hocolim_{[m]\in\Del^{\op}}\pr{X_{n,m}}\\
 & \simeq\hocolim_{[m]\in\Del^{\op}}\hocolim_{[n]\in\Del^{\op}_{\inj}}\pr{X_{n,m}},
\end{align*}
which is natural in $X$. This implies that, given a map $f:X\to Y$
of semisimplicial simplicial sets, if the map
\[
\hocolim_{[n]\in\Del^{\op}_{\inj}}X_{n,m}\to\hocolim_{[n]\in\Del^{\op}_{\inj}}Y_{n,m}
\]
is a weak homotopy equivalence for every $m\geq0$, then $\hocolim_{[n]\in\Del^{\op}_{\inj}}f$
is a weak homotopy equivalence. Substituting $\cal C\times N^{\mathrm{nd}}\pr{\bb Z_{\geq0}}$
for $X$ and $\cal C$ for $Y$, we are reduced to the case where
each $\cal C_{n}$ is a (discrete simplicial) set, i.e., $\cal C=N\pr{\sf C}$
for some ordinary category $\sf C$.

Let $\sf C'\subset\sf C\times\bb Z_{\geq0}$ denote the subcategory
generated by the morphisms $\pr{X,n}\to\pr{Y,m}$ such that $n<m$.
We can identify the semisimplicial set $N\pr{\sf C}\times N^{{\rm nd}}\pr{\bb Z_{\geq0}}$
with the restriction of the nerve $N\pr{\sf C'}$ to $\Del^{\op}_{\inj}$.
Under this identification, our goal is to show that the projection
$\pi:\sf C'\to\sf C$ gives a weak homotopy equivalence
\[
\hocolim_{\Del^{\op}_{\inj}}N\pr{\sf C'}\to\hocolim_{\Del^{\op}_{\inj}}N\pr{\sf C}.
\]
Since the inclusion $\Del^{\op}_{\inj}\hookrightarrow\Del^{\op}$
is homotopy final \cite[Lemma 6.5.3.7]{HTT}, we may replace the indexing
category of homotopy colimits by $\Del^{\op}$. With our model of
$\hocolim_{\Del^{\op}}$, the relevant map will be just the nerve
of $\pi$. Thus, we are reduced to showing that $N\pr{\pi}$ is a
weak homotopy equivalence.

To prove that $N\pr{\pi}$ is a weak homotopy equivalence, we factor
$\pi$ as $\sf C'\xrightarrow{\iota}\sf C\times\bb Z_{\geq0}\xrightarrow{\pi'}\sf C$,
where $\iota$ is the inclusion and $\pi'$ is the projection. Since
$N\pr{\pi'}$ is a weak homotopy equivalence (as $N\pr{\bb Z_{\geq0}}$
is weakly contractible), it will suffice to show that $N\pr{\iota}$
is a weak homotopy equivalence. We prove this by using Quillen's Theorem
A: We must show that, for each $\pr{X,n}\in\sf C\times\bb Z_{\geq0}$,
the fiber product
\[
\sf P=\sf C'\times_{\sf C\times\bb Z_{\geq0}}\pr{\sf C\times\bb Z_{\geq0}}_{\pr{X,n}/}
\]
has weakly contractible nerve. The inclusion $\pr{\sf C'}_{\pr{X,n}/}\hookrightarrow\sf P$
is a left adjoint (with right adjoint given by $\pr{Y,m}\mapsto\pr{X,n}$
if $m=n$ and $\pr{Y,m}\mapsto\pr{Y,m}$ if $m>n$), so it suffices
to show that $N\pr{\pr{\sf C'}_{\pr{X,n}/}}$ is weakly contractible.
This is clear, as $\pr{\sf C'}_{\pr{X,n}}$ has an initial object.
The proof is now complete.
\end{proof}

\begin{proof}
[Proof of Proposition \ref{prop:comparison}]Consider the maps
\[
B_{{\rm MD}}\cal G=\norm{\cal G}\xleftarrow{\phi'}\norm{\cal G\times N^{{\rm nd}}\pr{\bb Z_{\geq0}}}\xrightarrow{\phi}B_{{\rm Hae}}\cal G,
\]
where $\phi$ is the map of Lemma \ref{lem:comparison} and $\phi'$
is induced by the projection. The map $\phi$ is a weak homotopy equivalence
by Lemma \ref{lem:comparison}. We will complete the proof by showing
that $\phi'$ is a weak homotopy equivalence.

Recall that the fat realization of semisimplicial spaces is a model
of homotopy colimits. Indeed, the functor $\norm -:\Top^{\Del^{\op}_{\inj}}\to\Top$
is left Quillen for the Reedy model structure (and $\Top$ carrying
the Quillen model structure), and its right adjoint is weakly equivalent
to the diagonal functor. Since $\norm -$ preserves weak equivalences
on the nose \cite[Theorem 2.2]{semisimplicial},\footnote{In \cite{semisimplicial}, topological spaces are assumed to be compactly
generated, but the proof of the cited theorem goes through for topological
spaces. Alternatively, one can also adopt the argument in \cite[Proposition 14.5.7]{cathtpy}.} this shows that fat realization computes homotopy colimits indexed
by $\Del^{\op}_{\inj}$. 

In light of the discussion in the previous paragraph, we can identify
$\phi'$ with the map
\[
\hocolim_{\Del^{\op}_{\inj}}\pr{\cal G\times N^{{\rm nd}}\pr{\bb Z_{\geq0}}}\to\hocolim_{\Del^{\op}_{\inj}}\cal G.
\]
This is a weak homotopy equivalence by the Quillen equivalence between
$\Top$ and $\SS$ and Lemma \ref{lem:comparison2}, and we are done.
\end{proof}

\section{Some Results on $\infty$-Categories}

In this section, we record some general results on $\infty$-categories.
Results in the first three subsections (\ref{subsec:path_fibrations},
\ref{subsec:Kan}, and \ref{subsec:t-st}) are well-known, but their
references are hard to find. The content of the final subsection (\ref{subsec:zero_ext})
is new, at least to the author's knowledge.

\subsection{\label{subsec:path_fibrations}Mapping Spaces of Arrow Categories}

In this section, we give a proof of the following (folklore) result,
and then use it to prove a few of its corollaries.
\begin{prop}
\label{prop:map_arrow}Let $\cal C$ be an $\infty$-category. There
is a pullback square 
\[\begin{tikzcd}
	{\operatorname{Fun}(\Delta^1,\mathcal{C})(f,g)} & {\mathcal{C}(X_0,Y_0)} \\
	{\mathcal{C}(X_1,Y_1)} & {\mathcal{C}(X_0,Y_1)}
	\arrow[from=1-1, to=1-2]
	\arrow[from=1-1, to=2-1]
	\arrow["{g_\ast}", from=1-2, to=2-2]
	\arrow["{f^*}"', from=2-1, to=2-2]
\end{tikzcd}\]in $\cal S$, natural in $\pr{f:X_{0}\to X_{1},g:Y_{0}\to Y_{1}}\in\Fun\pr{\Delta^{1},\cal C}^{\op}\times\Fun\pr{\Delta^{1},\cal C}$.
\end{prop}

\begin{proof}
Recall that for an $\infty$-category $\cal X$, the \textit{twisted
arrow construction} produces a left fibration $\Tw\pr{\cal X}\to\cal X$
classifying the hom-functor $\cal X\pr{-,-}:\cal X^{\op}\times\cal X\to\cal S$
\cite[\href{https://kerodon.net/tag/03JF}{Tag 03JF}]{kerodon}. According
to \cite[Theorem 4.4]{AS23}, the square 
\[\begin{tikzcd}
	{\operatorname{Tw}(\operatorname{Fun}(\Delta^1,\mathcal{C}))} & {\operatorname{Fun}(\operatorname{Tw}(\Delta^1),\operatorname{Tw}(\mathcal{C}))} \\
	{\operatorname{Fun}(\Delta^1,\mathcal{C})^{\mathrm{op}}\times \operatorname{Fun}(\Delta^1,\mathcal{C})} & {\operatorname{Fun}(\operatorname{Tw}(\Delta^1),\mathcal{C}^\mathrm{op}\times \mathcal{C})}
	\arrow[from=1-1, to=1-2]
	\arrow[from=1-1, to=2-1]
	\arrow[from=1-2, to=2-2]
	\arrow[from=2-1, to=2-2]
\end{tikzcd}\]is homotopy cartesian. We now observe that $\Tw\pr{\Delta^{1}}$ is
the poset $00\to01\ot11$, i.e., a walking cospan. Consequently, the
claim follows from the following assertion:
\begin{itemize}
\item [($\ast$)]Let $p:\cal E\to\cal B$ be a left fibration classified
by a functor $f:\cal B\to\cal S$, and let $\cal X$ be an $\infty$-category.
The left fibration $\Fun\pr{\cal X,p}$ is classified by the composite
$\Fun\pr{\cal X,\cal B}\xrightarrow{\Fun\pr{\cal X,f}}\Fun\pr{\cal X,\cal S}\xrightarrow{\lim}\cal S$.
\end{itemize}

To prove ($\ast$), it suffices to consider the case where $p=p_{\univ}:\cal S_{\ast}\to\cal S$
is the universal left fibration. In this case, consider the following
diagram 
\[\begin{tikzcd}
	{\operatorname{Fun}(\mathcal{X},\mathcal{S}_\ast)} & {\operatorname{Fun}(\mathcal{X}^\triangleleft,\mathcal{S}_\ast)} & {\mathcal{S}_\ast} \\
	{\operatorname{Fun}(\mathcal{X},\mathcal{S})} & {\operatorname{Fun}(\mathcal{X}^\triangleleft,\mathcal{S})} & {\mathcal{S}.}
	\arrow["{\mathrm{Ran}}", from=1-1, to=1-2]
	\arrow[from=1-1, to=2-1]
	\arrow["{\mathrm{ev}_\infty}", from=1-2, to=1-3]
	\arrow[from=1-2, to=2-2]
	\arrow["{p_{\mathrm{univ}}}", from=1-3, to=2-3]
	\arrow["{\mathrm{Ran}}"', from=2-1, to=2-2]
	\arrow["{\mathrm{ev}_\infty}"', from=2-2, to=2-3]
\end{tikzcd}\]The left-hand square commutes up to natural equivalence because $p_{\univ}$
preserves limits. Combining this with the fact that the right Kan
extension functors are fully faithful \cite[Proposition 4.3.2.15]{HTT},
we find that the left-hand square is a pullback square of possibly
large $\infty$-categories. Also, since the inclusion $\{\infty\}\hookrightarrow\cal X^{\lcone}$
is left anodyne and $p_{\univ}$ is a left fibration, \cite[Corollary 2.1.2.7]{HTT}
implies that the right-hand square is also a pullback. The proof is
now complete, as the bottom horizontal composite is nothing but the
limit functor.
\end{proof}

\begin{cor}
\label{cor:map_slice}Let $\cal C$ be an $\infty$-category. For
each object $X\in\cal C$, there is a pullback square 
\[\begin{tikzcd}
	{\mathcal{C}_{/X}(a,b)} & {\mathcal{C}(A,B)} \\
	\ast & {\mathcal{C}(A,X)}
	\arrow[from=1-1, to=1-2]
	\arrow[from=1-1, to=2-1]
	\arrow["{b_\ast}", from=1-2, to=2-2]
	\arrow["a"', from=2-1, to=2-2]
\end{tikzcd}\]in $\cal S$, natural in $\pr{a:A\to X,b:B\to X}\in\pr{\cal C_{/X}}^{\op}\times\cal C_{/X}$.
\end{cor}

\begin{proof}
This follows from Proposition \ref{prop:map_arrow} and the equivalence
$\cal C_{/X}\xrightarrow{\simeq}\{X\}\times_{\Fun\pr{\{0\},\cal C}}\Fun\pr{\Delta^{1},\cal C}$
of \cite[Proposition 4.2.1.5]{HTT}.
\end{proof}

We will need the following two results in the main body of the paper:
\begin{prop}
\label{prop:final_map_into_slice}Let $\cal C$ be an $\infty$-category,
let $\cal C_{0}\subset\cal C$ be a full subcategory, let $X\in\cal C$
be an object, and let $\cal I$ be a small category. Let $f:\cal I^{\rcone}\to\cal C$
be a functor which carries $\cal I$ into $\cal C_{0}$ and the cone
point to $X$. The following conditions are equivalent:
\begin{enumerate}
\item The functor $\cal I\to\cal C^{0}_{/X}$ is final.
\item For each object $C\in\cal C_{0}$ admitting a morphism $\phi:C\to X$
in $\cal C$, the map
\[
\colim_{I\in\cal I}\cal C\pr{C,f\pr I}\to\cal C\pr{C,X}
\]
is an equivalence of $\infty$-groupoids.
\end{enumerate}
\end{prop}

\begin{proof}
Condition (1) is equivalent to the condition that, for each object
$\pr{c:C\to X}\in\cal C^{0}_{/X}$, the $\infty$-category $\pr{\cal C^{0}_{/X}}_{c/}\times_{\cal C^{0}}\cal I$
is weakly contractible. By \cite[Corollary 3.3.4.6]{HTT} and \cite[Theorem 4.2.4.1]{HTT},
the $\infty$-groupoid $\pr{\cal C^{0}_{/X}}_{c/}\times_{\cal C^{0}}\cal I$
has the homotopy type of $\colim_{I\in\cal I}\pr{\cal C^{0}_{/X}}_{c/}\pr{c,f\pr I}\in\cal S$.
Since colimits in $\cal S$ are universal \cite[Lemma 6.1.3.14]{HTT},
it follows from Corollary \ref{cor:map_slice} that $\colim_{I\in\cal I}\cal C_{/X}\pr{c,f\pr I}$
has the weak homotopy type of the fiber of the map
\[
\colim_{I\in\cal I}\cal C\pr{C,f\pr I}\to\cal C\pr{C,X}
\]
over $\alpha$. The claim follows.
\end{proof}

For later reference, we record a variant of Proposition \ref{prop:final_map_into_slice}. 
\begin{prop}
\label{prop:final_map_into_slice_2}Let $\cal C$ be an $\infty$-category,
let $\cal C_{0}\subset\cal C$ be a full subcategory, let $X\in\cal C$
be an object, and let $\cal I$ be an $\infty$-category. Let $f:\cal I^{\rcone}\to\cal C$
be a functor which carries $\cal I$ into $\cal C^{\simeq}_{0}$ and
the cone point to $X$. The following conditions are equivalent:
\begin{enumerate}
\item The functor $f':\cal I\to\pr{\pr{\cal C_{0}}_{/X}}^{\simeq}$ which
is adjoint to $f$ is a weak homotopy equivalence.
\item For each object $C\in\cal C_{0}$ which admits a morphism $C\to X$
in $\cal C$, the map
\[
\colim_{I\in\cal I}\cal C^{\simeq}\pr{C,f\pr I}\to\cal C\pr{C,X}
\]
is an equivalence of $\infty$-groupoids.
\end{enumerate}
\end{prop}

The proof of Proposition \ref{prop:final_map_into_slice_2} is nearly
identical to that of Proposition \ref{prop:final_map_into_slice},
noting that for a Kan complex $X$, a map $K\to X$ of simplicial
sets is final if and only if it is a weak homotopy equivalence. Details
are left to the reader.

\subsection{\label{subsec:Kan}Kan Extension}

In this subsection, we summarize a few key facts on Kan extensions.
The starting point is the following lemma.
\begin{lem}
\cite[\href{https://kerodon.net/tag/02ZM}{Tag 02ZM}]{kerodon}\label{lem:02ZM}
Let $\cal C,\cal D,$ and $\cal E$ be $\infty$-categories. Let $p:\cal C\to\cal D$,
$f:\cal C\to\cal E$, and $l:\cal D\to\cal E$ be functors, and let
$\alpha:f\to lp$ be a natural transformation. Suppose that $p$ is
a cocartesian fibration. The following conditions are equivalent:
\begin{enumerate}
\item The natural transformation $\alpha$ exhibits $l$ as a left Kan extension
of $f$ along $p$.
\item For each object $D\in\cal D$, the map $\alpha$ restricts to a colimit
diagram $f\vert\cal C_{D}\to\u{lp\pr D}$, where we set $\cal C_{D}=\cal C\times_{\cal D}\{D\}$.
\end{enumerate}
\end{lem}

\begin{rem}
\label{rem:02ZM}Informally, Lemma \ref{lem:02ZM} says that a left
Kan extension $\Lan_{p}f$ of $f$ along $p$, if it exists, is given
by the formula
\[
\Lan_{p}f\pr D\simeq\colim_{C\in\cal C_{D}}f\pr C.
\]
This informal notation is quite useful, and we will use it frequently.
Thus, for instance, if we say ``define a functor $g$ by $g\pr D=\colim_{C\in\cal C_{D}}f\pr C$,''
what we really mean is that ``we define $g$ to be a left Kan extension
of $f$ along $p$.'' (In the applications we have in mind, $\cal D$
will be the nerve of an ordinary category $\cal D_{0}$, and $p$
will be the relative nerve of an ordinary functor $\cal D_{0}\to\SS$.)
\end{rem}

Lemma \ref{lem:02ZM} has the following consequences.
\begin{prop}
\label{prop:fiberwise_finality}Let 
\[\begin{tikzcd}
	{\mathcal{A}} && {\mathcal{B}} \\
	& {\mathcal{C}}
	\arrow["f", from=1-1, to=1-3]
	\arrow["q", from=1-3, to=2-2]
	\arrow["p"', from=1-1, to=2-2]
\end{tikzcd}\]be a commutative diagram of $\infty$-categories. Suppose that $p$
and $q$ are cocartesian fibrations and that, for each object $C\in\cal C$,
the map
\[
\cal A\times_{\cal C}\{C\}=\cal A_{C}\to\cal B_{C}
\]
induced by $f$ is final. Then $f$ is final.
\end{prop}

\begin{proof}
We wish to show that, for each object $B\in\cal B$, the $\infty$-category
$\cal A_{B/}$ is weakly contractible. According to \cite[Corollary 3.3.4.6]{HTT},
the simplicial set $\cal A_{B/}$ has the weak homotopy type of the
colimit of the diagram $\cal A\xrightarrow{f}\cal B\xrightarrow{F_{B}}\cal S$,
where $F_{B}$ is the functor corepresented by $B$. By the same token,
$\cal B_{B/}$ has the weak homotopy type of the colimit of $F_{B}$.
Therefore, to prove that $\cal A_{B/}$ is weakly contractible, it
suffices to prove that precomposing $f$ does not change the colimit
of any diagram $\cal B\to\cal S$. In other words, it suffices to
show that, for each diagram $g:\cal B\to\cal S$, the map
\[
\cal S^{g/}\to\cal S^{gf/}
\]
is a covariant equivalence over $\cal S$. 

Find a functor $l:\cal C\to\cal S$ and a natural transformation $\alpha:g\to lq$
which exhibits $l$ as a left Kan extension of $g$ along $q$, and
find also a Kan complex $K$ and a natural transformation $\beta:l\to\underline{K}$
which exhibits $K$ as a colimit of $l$. Then a composite natural
transformation
\[
\theta:g\xrightarrow{\alpha}lq\xrightarrow{\beta q}\underline{K}
\]
exhibits $K$ as a colimit of $g$ (by the transitivity of Kan extensions\cite[\href{https://kerodon.net/tag/031M}{Tag 031M}]{kerodon}),
and we wish to show that the natural transformation $\theta f:gf\to\underline{K}$
exhibits $K$ as a colimit of $gf$. For this, it suffices to show
that the natural transformation $\alpha f:gf\to lp$ exhibits $l$
as a left Kan extension of $gf$ along $p$. By Lemma \ref{lem:02ZM},
we must show that, for each $C\in\cal C$, the natural transformation
\[
\alpha f\vert\cal A_{C}:gf\vert\cal A_{C}\to\underline{l\pr C}
\]
exhibits $l\pr C$ as a colimit of $gf\vert\cal A_{C}$. Since $f\vert\cal A_{C}:\cal A_{C}\to\cal B_{C}$
is final, we are reduced to showing that the natural transformation
\[
\alpha\vert\cal B_{C}:g\vert\cal B_{C}\to\underline{l\pr C}
\]
is a colimit diagram. This follows from Lemma \ref{lem:02ZM}.
\end{proof}

\begin{prop}
\label{prop:dec_colim}Let $\overline{p}:\overline{\cal E}\to\cal B^{\rcone}$
be a cocartesian fibration of $\infty$-categories classifying a colimit
diagram in $\cal B^{\rcone}\to\Cat_{\infty}$, and let $f:\overline{\cal E}\to\cal C$
be a diagram that carries $\overline{p}$-cocartesian edges to equivalences.
If $f$ admits a left Kan extension $F:\cal B^{\rcone}\to\cal C$
along $p$, then $F$ is a colimit diagram.
\end{prop}

\begin{rem}
In the situation of Proposition \ref{prop:dec_colim}, the functor
$F$ is given by $F\pr B=\colim_{E\in\cal E_{B}}f\pr E$ (Lemma \ref{lem:02ZM}).
Thus, informally, the proposition says that
\[
\colim_{E\in\colim_{B\in\cal B}\cal E_{B}}f\pr E\simeq\colim_{B\in\cal B}\colim_{E\in\cal E_{B}}f\pr E.
\]
\end{rem}

\begin{rem}
\label{rem:dec_lim}There is an obvious dual version of Proposition
\ref{prop:dec_colim}: Let $\overline{p}:\overline{\cal E}\to\pr{\cal B^{\rcone}}^{\op}$
is a cartesian fibration of $\infty$-categories classifying a colimit
diagram $\cal B^{\rcone}\to\Cat_{\infty}$, and let $f:\overline{\cal E}\to\cal C$
be a functor carrying $\overline{p}$-cartesian edges to equivalences.
If $f$ has a left Kan extension $F:\pr{\cal B^{\rcone}}^{\op}\to\cal C$
along $p$, then $F$ is a limit diagram. 
\end{rem}

\begin{proof}
[Proof of Proposition \ref{prop:dec_colim}]Consider the following
diagram and natural transformations:
\[\begin{tikzcd}
	{\mathcal{E}} && {\overline{\mathcal{E}}} && {\mathcal{C}} \\
	{\mathcal{B}} && {\mathcal{B}^\triangleright} \\
	& {\{\infty\}}
	\arrow["j", from=1-1, to=1-3]
	\arrow["p"', from=1-1, to=2-1]
	\arrow[""{name=0, anchor=center, inner sep=0}, "f", from=1-3, to=1-5]
	\arrow["{\overline{p}}"', from=1-3, to=2-3]
	\arrow[""{name=1, anchor=center, inner sep=0}, "i", from=2-1, to=2-3]
	\arrow[from=2-1, to=3-2]
	\arrow[""{name=2, anchor=center, inner sep=0}, "F"', from=2-3, to=1-5]
	\arrow["k"', from=3-2, to=2-3]
	\arrow["\alpha"', between={0.2}{0.8}, Rightarrow, from=0, to=2]
	\arrow["\beta"{description}, between={0.2}{1}, Rightarrow, from=1, to=3-2]
\end{tikzcd}\]Here the maps $i,j,k$ are inclusions, $p$ is the pullback of $\overline{p}$
along $i$, and $\alpha$ is the structure natural transformation
for a left Kan extension. Our goal is to show that $F\beta$ is a
colimit diagram. Since $\alpha j$ exhibits $Fi$ as a left Kan extension
of $fj$ along $p$ (by Lemma \ref{lem:02ZM}), we only need to show
that the composite $2$-cell
\[
\theta:fj\xrightarrow{\alpha j}F\overline{p}j=Fip\xrightarrow{F\beta p}\underline{F\pr{\infty}}
\]
is a colimit diagram. 

Let $j_{\infty}:\overline{\cal E}_{\infty}\to\overline{\cal E}$ denote
the inclusion, and let $\gamma:j\to j_{\infty}r$ denote the cocartesian
natural transformation covering $\beta$ (i.e., $\beta p=\overline{p}\gamma$),
where $r:\cal E\to\overline{\cal E}_{\infty}$ is the codomain restriction
of $\gamma\vert\cal E\times\{1\}$. The map $\theta$ is then homotopic
to the pasting of the $2$-cells
\[\begin{tikzcd}
	& {\overline{\mathcal{E}_{\infty}}} \\
	{\mathcal{E}} && {\overline{\mathcal{E}}} && {\mathcal{C}} \\
	{\mathcal{B}} && {\mathcal{B}^\triangleright}
	\arrow["{j_\infty}", from=1-2, to=2-3]
	\arrow["r", from=2-1, to=1-2]
	\arrow[""{name=0, anchor=center, inner sep=0}, "j"', from=2-1, to=2-3]
	\arrow["p"', from=2-1, to=3-1]
	\arrow[""{name=1, anchor=center, inner sep=0}, "f", from=2-3, to=2-5]
	\arrow["{\overline{p}}"', from=2-3, to=3-3]
	\arrow["i", from=3-1, to=3-3]
	\arrow[""{name=2, anchor=center, inner sep=0}, "F"', from=3-3, to=2-5]
	\arrow["\gamma"', between={0.2}{1}, Rightarrow, from=0, to=1-2]
	\arrow["\alpha"', between={0.2}{0.8}, Rightarrow, from=1, to=2]
\end{tikzcd}\]which can be written as a composite
\[
fj\xrightarrow{f\gamma}fj_{\infty}r\xrightarrow{\alpha j_{\infty}r}F\overline{p}j_{\infty}r.
\]
Since $f$ carries cocartesian morphisms to equivalences, the map
$f\gamma$ is an equivalence. Therefore, it will suffice to show that
$\alpha j_{\infty}r$ is a colimit diagram. Since $\overline{p}$
classifies a colimit diagram, \cite[Lemma 3.3.4.1 and Proposition 3.3.4.2]{HTT}
show that the map $r$ is a localization at $p$-cocartesian morphisms
(in the sense of Definition \ref{subsec:loc}). Since localizations
are final \cite[\href{https://kerodon.net/tag/02N9}{Tag 02N9}]{kerodon},
we are reduced to showing that $\alpha j_{\infty}$ is a colimit diagram.
This follows from Lemma \ref{lem:02ZM}, and we are done.
\end{proof}

\subsection{\label{subsec:t-st}t-Structure and Homotopy Groups}

Let $\cal C$ be a stable $\infty$-category. Recall that a (homological)\textbf{
t-structure} on $\cal C$ is a collection of full subcategories $\{\cal C_{\geq n}\}_{n\in\bb Z}$
and $\{\cal C_{\leq n}\}_{n\in\bb Z}$ of $\cal C$ that determine
the ordinary t-structure on the triangulated category $\ho\pr{\cal C}$
\cite[1.2.1]{HA}. If $\cal C$ is equipped with a t-structure, its
\textbf{heart} $\cal C^{\heartsuit}=\cal C_{\leq0}\cap\cal C_{\geq0}$
is equivalent to its homotopy category, which is an abelian category.
For each $n\geq0$, there is an \textbf{$n$th homotopy group} functor
$\pi_{n}:\cal C\to\cal C^{\heartsuit}$ \cite[Definition 1.2.1.11]{HA},
and every fiber sequence $X\to Y\to Z$ in $\cal C$ gives rise to
a long exact sequence 
\[
\cdots\to\pi_{n+1}\pr Z\to\pi_{n}\pr X\to\pi_{n}\pr Y\to\pi_{n}\pr Z\to\pi_{n-1}\pr X\to\cdots
\]
in $\cal C^{\heartsuit}$ \cite[IV.4. Theorem 11]{GM03}. In this
subsection, we record some results on homotopy groups that we will
use in the main body of the paper.

The datum of a t-structure give us a measurement of the failure of
a cube to be cartesian. The following sequences of results, up to
Corollary \ref{cor:GW99_1.22}, give an estimate of this sort using
estimates of smaller cubes. These results and arguments are due to
Goodwillie, but we record them for completeness.
\begin{defn}
Let $\cal C$ be a stable $\infty$-category with a t-structure, and
let $k$ be an integer. An object $X\in\cal C$ is said to be \textbf{$k$-connected}
if it belongs to $\cal C_{\geq k+1}$, or equivalently, $\pi_{i}\pr X\simeq0$
for all $i\leq k$. A map $f:X\to Y$ in $\cal C$ is said to be \textbf{$k$-connected}
if its fiber is $\pr{k-1}$-connected. Given a finite set $S$ and
an $S$-cube $X$ in $\cal C$, we say that $X$ is \textbf{$k$-cartesian}
if the map
\[
X\pr{\emptyset}\to\lim\pr{X\vert\cal P_{0}\pr S}
\]
is $k$-connected.
\end{defn}

\begin{prop}
\cite[Proposition 1.5]{Goo91}\label{prop:1.5} Let $\cal C$ be a
stable $\infty$-category with a t-structure, let $k$ be an integer,
and let $X\xrightarrow{f}Y\xrightarrow{g}Z$ be a pair of maps in
$\cal C$.
\begin{enumerate}
\item If $f$ and $g$ are $k$-connected, so is $gf$.
\item If $g$ is $\pr{k+1}$-connected and $gf$ is $k$-connected, then
$f$ is $k$-connected.
\item If $gf$ is $k$-connected and $f$ is $\pr{k-1}$-connected, then
$g$ is $k$-connected.
\end{enumerate}
\end{prop}

\begin{proof}
This follows from the long exact sequence associated with the fiber
sequence $\fib\pr f\to\fib\pr{gf}\to\fib\pr g$.
\end{proof}

\begin{prop}
\cite[Proposition 1.6]{Goo91}\label{prop:1.6} Let $\cal C$ be a
stable $\infty$-category with a t-structure, let $S$ be a finite
set, and let $F$ be an $\pr{S\amalg\{\ast\}}$-cube in $\cal C$.
Define $S$-cubes $X,Y$ by $X=F\vert\cal P\pr S$ and $Y=F\pr{-\cup\{\ast\}}$.

\begin{enumerate}[label=(\roman*)]

\item If $Y$ is $k$-cartesian and $F$ is $k$-cartesian, then
$X$ is $k$-cartesian.

\item If $X$ is $k$-cartesian and $Y$ is $\pr{k+1}$-cartesian,
then $F$ is $k$-cartesian.

\end{enumerate}
\end{prop}

\begin{proof}
We will apply Proposition \ref{prop:1.5} to the maps
\[
X\pr{\emptyset}\xrightarrow{f}\lim\pr{F\vert\cal P_{0}\pr{S\amalg\{\ast\}}}\xrightarrow{g}\lim\pr{X\vert\cal P_{0}\pr S}.
\]
To obtain the conclusion, it will suffice to show that $g$ is a pullback
of the map $Y\pr{\emptyset}\to\lim\pr{Y\vert\cal P_{0}\pr S}$. This
follows from Remark \ref{rem:dec_lim}, as the square 
\[\begin{tikzcd}
	{(\mathcal{P}(S)\times [1])\setminus \{(\emptyset ,0)\}} & {\mathcal{P}_0(S)\times [1]} \\
	{\mathcal{P}(S)\times \{1\}} & {\mathcal{P}_0(S)\times\{1\}}
	\arrow[from=1-2, to=1-1]
	\arrow[from=2-1, to=1-1]
	\arrow[from=2-2, to=1-2]
	\arrow[from=2-2, to=2-1]
\end{tikzcd}\]in a pushout in $\Cat_{\infty}$ (because the Joyal model structure
is left proper).
\end{proof}

\begin{prop}
\cite[Proposition 1.20]{Goo91}\label{prop:1.20} Let $\cal C$ be
a stable $\infty$-category with a t-structure, let $S$ and $T$
be finite sets, and let $X:\cal P\pr{S\amalg T}\to\cal D$ be an $S\amalg T$-cube.
Suppose that:
\begin{enumerate}
\item $X$ is $k$-cartesian.
\item For each $U\in\cal P_{0}\pr S$, the $T$-cube $X\pr{U\amalg-}$ is
$\pr{k+\abs U-1}$-cartesian. 
\end{enumerate}
Then the $T$-cube $X\pr{\emptyset\amalg-}$ is $k$-cartesian.
\end{prop}

\begin{proof}
We prove the claim by induction on $n=\abs S$. If $n=0$, there is
nothing to prove. For the inductive step, fix an element $\ast\in S$,
and set $S'=S\setminus\{\ast\}$. Applying the induction hypothesis
to the $S'\amalg\pr{\{\ast\}\amalg T}$-cube $X$, we deduce that
the $\pr{\{\ast\}\amalg T}$-cube $X\pr{\emptyset\amalg-\amalg-}$
is $k$-cartesian. Also, by condition (2), the $T$-cube $X\pr{\emptyset\amalg\{\ast\}\amalg-}$
is $k$-cartesian. It follows from part (i) of Proposition \ref{prop:1.6}
that the $T$-cube $X\pr{\emptyset\amalg\emptyset\amalg-}$ is $k$-cartesian,
as desired.
\end{proof}

\begin{cor}
\cite[Proposition 1.22]{Goo91}\label{cor:GW99_1.22} Let $\cal C$
be a stable $\infty$-category with a t-structure, let $S$ be a finite
set, and let $X:\cal P_{0}\pr S\times[1]\to\cal D$ be a diagram.
Suppose that, for each $U\in\cal P_{0}\pr S$, the map
\[
X\pr{U,0}\to X\pr{U,1}
\]
is $k_{U}$-connected. Then the map 
\[
\lim_{U\in\cal P_{0}\pr S}X\pr{U,0}\to\lim_{U\in\cal P_{0}\pr S}X\pr{U,1}
\]
is $\min\{1-\abs U+k_{U}\mid U\in\cal P_{0}\pr S\}$-connected.
\end{cor}

\begin{proof}
Let $\overline{X}:\cal P\pr S\times[1]\to\cal D$ denote the right
Kan extension of $X$. Since the inclusion $\cal P\pr S\to\cal P\pr{S\amalg\{\ast\}}$
is initial, the $\pr{S\amalg\{\ast\}}$-cube $\overline{X}$ is cartesian.
So the claim follows from applying Proposition \ref{prop:1.20} to
$\overline{X}$.
\end{proof}

We conclude this subsection with two more miscellaneous results.
\begin{lem}
[Milnor exact sequence]\label{lem:Milnor}Let $\cal C$ be a stable
$\infty$-category with a t-structure. Suppose that $\cal C$ has
countable products and that the functor $\pi_{0}:\cal C\to\cal C^{\heartsuit}$
preserves them. Given a diagram $X_{\bullet}:\bb Z\to\cal C$ depicted
as
\[
\cdots\xrightarrow{p_{2}}X_{2}\xrightarrow{p_{1}}X_{1}\xrightarrow{p_{0}}X_{0}\xrightarrow{p_{-1}}\cdots,
\]
there is an exact sequence
\[
0\to\lim^{1}_{i}\pi_{n+1}\pr{X_{i}}\to\pi_{n}\pr{\lim_{i}X_{i}}\to\lim_{i}\pi_{n}\pr{X_{i}}\to0
\]
in the abelian category $\cal C^{\heartsuit}\simeq\ho\pr{\cal C^{\heartsuit}}$,
which is natural in $X_{\bullet}$.
\end{lem}

\begin{proof}
Set $X=\lim_{i}X_{i}$. Since $\pi_{0}$ carries fiber sequences to
exact sequences \cite[IV.4., Theorem 11]{GM03}, it suffices to prove
the following two assertions:
\begin{enumerate}
\item $X$ is naturally equivalent to equalizer of the maps $\id,p:\prod_{i}X_{i}\rr\prod_{i}X_{i}$,
where $p$ is determined by the maps $\{p_{i}\}_{i\geq0}$.
\item Given a pair of maps $f,g:A\rr B$ in $\cal C$, the equalizer of
$f$ and $g$ is equivalent to naturally equivalent to the kernel
of $f-g$.
\end{enumerate}
Indeed, assertions (1) and (2) will give us a fiber sequence $X\to\prod_{i}X_{i}\xrightarrow{\id-p}\prod_{i}X_{i}$,
and then the associated long exact sequence of homotopy groups gives
us the desired conclusion.

To prove (1), let $\cal B=\{0\stackrel[b]{a}{\rr}1\}$ denote (the
nerve of) the free walking parallel arrows, and let $\cal E$ denote
the Grothendieck construction of the functor $\cal B\to\sf{Cat}$
corresponding to the maps $\id,-1:\bb Z\rr\bb Z$. Thus we can depict
$\cal E$ as as
\[\begin{tikzcd}
	\vdots & {(1,i-1)} \\
	{(0,i)} & {(1,i)} \\
	{(0,i+1)} & \vdots,
	\arrow[from=1-1, to=1-2]
	\arrow["{b_{i-1}}"{description}, from=2-1, to=1-2]
	\arrow["{a_i}"{description}, from=2-1, to=2-2]
	\arrow["{b_{i}}"{description}, from=3-1, to=2-2]
	\arrow[from=3-1, to=3-2]
\end{tikzcd}\]and there is a projection $\cal E\to\cal B$ carrying $a_{i}$ to
$a$ and $b_{i}$ to $b$. The assignment $\pr{n,i}\mapsto n$ determines
a functor $\pi:\cal E\to\bb Z$. Using the fact that the classifying
space of $[1]$ is contractible, we deduce that $\pi$ is a localization
at the maps $\{a_{i}\}_{i\geq0}$. In particular, $\pi$ is initial
\cite[\href{https://kerodon.net/tag/02N9}{Tag 02N9}]{kerodon}, so
\[
\lim_{i}X_{i}\simeq\lim\pr{\cal E\xrightarrow{\pi}\bb Z\xrightarrow{X_{\bullet}}\cal C}.
\]
The right-hand limit is equivalent to the limit of the right Kan extension
of $X_{\bullet}\circ\pi$ along the projection $\cal E\to\cal B$,
which, by Remark \ref{rem:dec_lim}, is the equalizer of $\id$ and
$p$.

Next, for (2), we consider the pullback squares
\[\begin{tikzcd}
	P & B & 0 \\
	A & {B\oplus B} & B
	\arrow[from=1-1, to=1-2]
	\arrow[from=1-1, to=2-1]
	\arrow[from=1-2, to=1-3]
	\arrow["\begin{array}{c} \left(\begin{smallmatrix}    1\\    1 \end{smallmatrix}\right) \end{array}", from=1-2, to=2-2]
	\arrow[from=1-3, to=2-3]
	\arrow["{(f,g)}"', from=2-1, to=2-2]
	\arrow["{(1\,-1)}"', from=2-2, to=2-3]
\end{tikzcd}\]The pullback $P$ is nothing but the equalizer of $f$ and $g$ \cite[\href{https://kerodon.net/tag/03HC}{Tag 03HC}]{kerodon},
and the pullback of the outer rectangle is the kernel of $f-g$. The
proof is now complete.
\end{proof}

\begin{lem}
\label{lem:conn_lim}Let $\cal C$ be a stable $\infty$-category
with a t-structure. Suppose that $\cal C$ has small limits and that
$0$-connected objects are stable under small products. Suppose we
are given a pushout diagram
\[\begin{tikzcd}
	{S^{d-1}\times I} & X \\
	{D^{d}\times I} & Y
	\arrow[from=1-1, to=1-2]
	\arrow[from=1-1, to=2-1]
	\arrow[from=1-2, to=2-2]
	\arrow[from=2-1, to=2-2]
\end{tikzcd}\]in $\cal S$, where $I$ is a set and $d\ge0$. For every diagram
$\phi:Y\to\cal C$ carrying each object to a $0$-connected object,
the map
\[
\theta:\lim\phi\to\lim\pr{\phi\vert X}
\]
is $\pr{-d+1}$-connected.
\end{lem}

\begin{proof}
By Remark \ref{rem:dec_lim}, the map $\theta$ is a pullback of the
map
\[
\prod_{i\in I}\lim\pr{\phi\vert\pr{D^{d}\times\{i\}}}\to\prod_{i\in I}\lim\pr{\phi\vert\pr{S^{d-1}\times\{i\}}}.
\]
Therefore, it suffices to show that the limit $\lim\pr{\phi\vert\pr{S^{d-1}\times\{i\}}}$
is $\pr{-d+1}$-connected for each $i\in I$. When $d=0$, the limit
is a zero object, so there is nothing to prove. If $d\geq1$, let
$C\in\cal C$ denote an image of a vertex of $S^{d-1}$. Since $D^{d}$
is contractible, the diagram $\phi\vert S^{d-1}\times\{i\}$ is equivalent
to the constant diagram at $C$. Thus
\[
\lim\phi\vert\pr{S^{d-1}\times\{i\}}\simeq C^{S^{d-1}}\simeq C^{\Sigma^{d-1}S^{0}}\simeq\Omega^{d-1}\pr{C^{S^{0}}}\simeq\Omega^{d-1}\pr{C\times C}.
\]
The right-hand side is $\pr{-d+1}$-connected, so we are done. 
\end{proof}

\subsection{\label{subsec:zero_ext}Extension by Zero Objects}

Consider the following problem: Let $\cal C$ be an $\infty$-category
and let $\cal C_{0},\cal C_{1}\subset\cal C$ be full subcategories
such that every object of $\cal C$ belongs to exactly one of $\cal C_{0}$
or $\cal C_{1}$. Let $\cal D$ be a pointed $\infty$-category and
suppose we are given a functor $f:\cal C\to\cal D$ carrying each
object of $\cal C_{0}$ to a zero object. One would expect that $f$
is equivalent to a smaller piece of data. In other words, we would
expect that there is a categorical equivalence
\[
\Fun'\pr{\cal C,\cal D}\simeq\{\text{something smaller}\}
\]
where $\Fun'\pr{\cal C,\cal D}\subset\Fun\pr{\cal C,\cal D}$ denotes
the full subcategory spanned by the objects carrying each object in
$\cal C_{0}$ to a zero object. We want to figure out what the smaller
thing is. 

A na\"ive guess of the ``smaller thing'' is $\Fun\pr{\cal C_{1},\cal D}$,
but this is too na\"ive. Indeed, there could be a morphism in $\cal C_{1}$,
say $\alpha$, which factors through a morphism of $\cal C_{0}$,
and there is no reason why an arbitrary functor $\cal C_{1}\to\cal D$
carries $\alpha$ to a zero map. So the restriction $\Fun'\pr{\cal C,\cal D}\to\Fun\pr{\cal C_{1},\cal D}$
is, in general, not essentially surjective. 

The next guess is that the essential image of the restriction functor
$\Fun'\pr{\cal C,\cal D}\to\Fun\pr{\cal C_{1},\cal D}$ will suffice.
If $\cal C$ and $\cal D$ are nerves of ordinary categories, this
is true. But for arbitrary $\infty$-categories, this is false; the
problem is that there could be a nontrivial homotopy between zero
maps, as one can see from the particular case where $\cal C=\Delta^{2}$
and $\cal C_{1}=\{1\}$. 

In this subsection, we will give one answer to this question in the
case where $\cal C_{1}$ is ``isolated'' in $\cal C$.

To state our main result, we need a bit of terminology.
\begin{defn}
Let $\cal C$ be an $\infty$-category, and let $\cal C'\subset\cal C$
be a full subcategory. A morphism $X\to Y$ of $\cal C'$ is said
to be \textbf{isolated in $\cal C$} if it does not factor through
any object of $\cal C$ which lies outside of $\cal C'$. We say that
$\cal C'$ is \textbf{isolated in $\cal C$ }if the morphisms of $\cal C'$
that are isolated in $\cal C$ are closed under compositions. 
\end{defn}

Here is the main result of this subsection.
\begin{prop}
\label{prop:zero_extension}Let $\cal C$ be an $\infty$-category,
and let $\cal C_{0},\cal C_{1}\subset\cal C$ be full subcategories
such that every object of $\cal C$ belongs to exactly one of $\cal C_{0}$
or $\cal C_{1}$. Let $\cal D$ be a pointed $\infty$-category, and
let $\cal Z\subset\cal D$ denote the full subcategory spanned by
the zero objects. Suppose that $\cal C_{1}$ is isolated in $\cal C$,
and let $\cal C^{{\rm isltd}}_{1}\subset\cal C_{1}$ denote the subcategory
spanned by the morphisms isolated in $\cal C$. The functor

\[
\theta:\Fun\pr{\cal C,\cal D}\times_{\Fun\pr{\cal C_{0},\cal D}}\Fun\pr{\cal C_{0},\cal Z}\to\Fun\pr{\cal C^{{\rm isltd}}_{1},\cal D}
\]
is a trivial fibration.
\end{prop}

\begin{proof}
Let us say that a morphism $f$ of $\cal C$ is \textbf{basic} if
either $f$ is a morphism of $\cal C^{{\rm isltd}}_{1}$ or $f$ is
not a morphism of $\cal C^{1}$. We will say that a simplex $\sigma:\Delta^{d}\to\cal C$
is \textbf{basic }if for each $0\leq i<d$, the restriction $\sigma\vert\Delta^{\{i,i+1\}}$
is basic. We let $\{\sigma_{\alpha}:\Delta^{d_{\alpha}}\to\cal C\}_{\alpha\in A}$
denote the set of all basic simplices of $\cal C$. Note that the
map $\coprod_{\alpha\in A}\Delta^{d_{\alpha}}\to\cal C$ is an epimorphism
of simplicial sets because every simplex of $\cal C$ is a face of
a basic simplex.

For each $\alpha\in A$, set $P_{\alpha}=\Delta^{d_{\alpha}}\times_{\cal C}\cal C_{0}$,
$Q_{\alpha}=\Delta^{d_{\alpha}}\times_{\cal C}\cal C_{1}$, and $Q'_{\alpha}=\Delta^{d_{\alpha}}\times_{\cal C}\cal C^{{\rm isltd}}_{1}$.
Since $\theta$ is a pullback of the map
\[
\Fun\pr{\coprod_{\alpha\in A}\Delta^{d_{\alpha}},\cal D}\times_{\Fun\pr{\coprod_{\alpha\in A}P_{\alpha},\cal D}}\Fun\pr{\coprod_{\alpha\in A}P_{\alpha},\cal Z}\to\Fun\pr{\coprod_{\alpha\in A}Q_{\alpha}',\cal D},
\]
it suffices to show that for each $\alpha\in A$, the map
\[
\theta_{\alpha}:\Fun\pr{\Delta^{d_{\alpha}},\cal D}\times_{\Fun\pr{P_{\alpha},\cal D}}\Fun\pr{P_{\alpha},\cal Z}\to\Fun\pr{Q'_{\alpha},\cal D}
\]
is a trivial fibration. Now since $\sigma_{\alpha}$ is isolated,
the subcategory $Q'_{\alpha}\subset Q_{\alpha}$ is spanned by the
morphisms of $Q_{\alpha}$ that are isolated in $\Delta^{d_{\alpha}}$.
Therefore, we are reduced to the case where $\cal C=\Delta^{d}$ for
some $d\geq0$. In this case, we will prove the assertion by induction
on $d$.

If $d=0$, the claim is obvious. For the inductive step, suppose we
have proved the assertion up to $d-1$. If $\cal C_{0}$ is empty,
the claim is trivial, so assume that $\cal C_{0}$ is nonempty. Let
$m$ be the maximal integer which belongs to $\cal C_{0}$. There
are two cases to consider.

Suppose first that $m=d$. In this case, \cite[Proposition 4.3.2.15]{HTT}
shows that the functor
\[
\Fun\pr{\cal C,\cal D}\times_{\Fun\pr{\cal C_{0},\cal D}}\Fun\pr{\cal C_{0},\cal Z}\to\Fun\pr{\Delta^{d-1},\cal D}\times_{\Fun\pr{\cal C_{0}\cap\Delta^{d-1},\cal D}}\Fun\pr{\cal C_{0}\cap\Delta^{d-1},\cal Z}
\]
is a trivial fibration. By the induction hypothesis, the functor
\[
\Fun\pr{\cal C\cap\Delta^{d-1},\cal D}\times_{\Fun\pr{\cal C_{0}\cap\Delta^{d-1},\cal D}}\Fun\pr{\cal C_{0}\cap\Delta^{d-1},\cal Z}\to\Fun\pr{\cal C^{{\rm isltd}}_{1},\cal D}
\]
is a trivial fibration. It follows that $\theta$ is a composition
of trivial fibrations and hence is itself a trivial fibration.

Suppose next that $m<d$. Set $X=\Delta^{\{0,\dots,m\}}\amalg_{\{m\}}\Delta^{\{m\dots,d\}}$.
The map $\theta$ factors as 
\begin{align*}
\Fun\pr{\cal C,\cal D}\times_{\Fun\pr{\cal C_{0},\cal D}}\Fun\pr{\cal C,\cal Z} & \xrightarrow{f}\Fun\pr{X,\cal D}\times_{\Fun\pr{\cal C_{0},\cal D}}\Fun\pr{\cal C_{0},\cal Z}\\
 & \xrightarrow{g}\Fun\pr{\cal C^{{\rm isltd}}_{1},\cal D}.
\end{align*}
The map $f$ is a trivial fibration because the inclusion $X\to\cal C$
is a weak categorical equivalence \cite[Lemma 5.4.5.10]{HTT}. It
will therefore suffice to show that $g$ is a trivial fibration. Since
$g$ is a categorical fibration, it suffices to show that it is a
categorical equivalence. We may identify $g$ with the map between
the pullbacks of the rows of the following commutative diagram: 
\[\begin{tikzcd}[column sep=tiny, scale cd=.75]
	{\operatorname{Fun}(\Delta^m,\mathcal{D})\times _{\operatorname{Fun}(\mathcal{C}_0\cap\Delta^m,\mathcal{D})}\operatorname{Fun}(\mathcal{C}_0\cap\Delta^m,\mathcal{Z})} & {\operatorname{Fun}(\{m\},\mathcal{Z})} & {\operatorname{Fun}(\Delta^{\{m,\dots,d\}},\mathcal{D})\times _{\operatorname{Fun}(\{m\},\mathcal{D})}\operatorname{Fun}(\{m\},\mathcal{Z})} \\
	{\operatorname{Fun}(\Delta^m\cap\mathcal{C}_1^{\mathrm{isltd}},\mathcal{D})} & {\Delta^0} & {\operatorname{Fun}(\Delta^{\{m,\dots,d\}}\cap\mathcal{C}_{1}^{\mathrm{isltd}},\mathcal{D}).}
	\arrow[from=1-1, to=1-2]
	\arrow[from=1-3, to=1-2]
	\arrow[from=1-2, to=2-2]
	\arrow[from=1-1, to=2-1]
	\arrow[from=1-3, to=2-3]
	\arrow[from=2-1, to=2-2]
	\arrow[from=2-3, to=2-2]
\end{tikzcd}\]By the induction hypothesis, the vertical arrows of the above diagrams
are trivial fibrations. Moreover, the horizontal arrows are categorical
fibrations of $\infty$-categories. Hence $g$ is also a categorical
equivalence, and the proof is complete. 
\end{proof}

\section{Miscellany}

In this section, we briefly recall some key definitions and results
that did not quite fit into the main body of the paper, but are nonetheless
used in it.

\subsection{\label{subsec:germ}Germs}

In this subsection, we will establish a few basic results on \textit{germs}
on manifolds. Intuitively, a germ of an $n$-manifold $M$ is an infinitesimal
embedding of $\bb R^{n}$ into $M$. We will see that the collection
of germs can be organized into a principal bundle\footnote{Principal fibrations of simplicial sets are an analog of principal
bundles of topological spaces. See \cite[ChapterV, $\S$2]{GoerssJardine}
for an overview.} over $M$ (or more precisely, a principal fibration over $\Sing M$),
which roughly plays the role of the frame bundle for smooth manifolds.
\begin{defn}
\cite[Definition 5.4.1.6]{HA} Let $n\geq0$. For each positive number
$r>0$, let $B^{n}\pr r\subset\bb R^{n}$ denote the open ball of
radius $r$ centered at the origin. Given a finite set $S$ and an
$n$-manifold $M$, we define the simplicial set $\Germ\pr{S,M}$
of \textbf{$S$-germs} of $M$ as the colimit 
\[
\Germ\pr{S,M}=\colim_{r\in\pr{\bb R_{>0}}^{\op}}\Sing\Emb\pr{S\times B^{n}\pr r,M}.
\]
If $S$ is a singleton, we simply write $\Germ\pr{S,M}=\Germ\pr M$.
The evaluation at the origin determines a map $\Germ\pr{S,M}\to\Sing\Conf\pr{S,M}$.
For each $p\in M$, we set $\Germ_{p}\pr M=\Germ\pr M\times_{\Sing M}\{p\}$.
\end{defn}

\begin{rem}
Let $n\geq0$. The simplicial set $\Germ_{0}\pr{\bb R^{n}}$ has the
structure of a simplicial monoid. To see this, note that its $k$-simplex
is represented by a map $\sigma:\abs{\Delta^{k}}\times B^{n}\pr r\to\bb R^{n}$,
where $r>0$ is a positive number and for each point $x\in\abs{\Delta^{k}}$,
the map $\sigma\pr{x,-}:B^{n}\pr r\to\bb R^{n}$ is an embedding which
fixes the origin. Given another such map $\tau:\abs{\Delta^{k}}\times B^{n}\pr{r'}\to\bb R^{n}$,
the product $[\tau]\circ[\sigma]$ is represented by the map
\[
\abs{\Delta^{k}}\times B^{n}\pr{r''}\xrightarrow{\pr{\id,\sigma}}\abs{\Delta^{k}}\times B^{n}\pr{r'}\xrightarrow{\tau}\bb R^{n},
\]
where $0<r''<r$ is chosen so that $\sigma\vert\Delta^{k}\times B^{n}\pr{r''}$
takes values in $B^{n}\pr{r'}$. (Such a number $r''$ exists because
$\abs{\Delta^{k}}$ is compact.)

More generally, if $M$ is an $n$-manifold and $S$ is a finite set,
then $\Germ_{0}\pr{\bb R^{n}}^{S}$ acts on $\Germ\pr{S,M}$ (from
the right) as follows: A $k$-simplex of $\Germ\pr{S,M}\times\Germ_{0}\pr{\bb R^{n}}^{S}$
is represented by maps $\tau:\abs{\Delta^{k}}\times\coprod_{s\in S}B^{n}\pr r\to M$
and $\pr{\sigma_{s}:\abs{\Delta^{k}}\times B^{n}\pr{r'}\to\bb R^{n}}_{s\in S}$
where $r,r'>0$. Its image in $\Germ\pr{S,M}$ is represented by the
composite 
\[
\abs{\Delta^{k}}\times\coprod_{s\in S}B^{n}\pr{r''}\xrightarrow{\pr{\id,\coprod_{s}\sigma_{s}}}\abs{\Delta^{k}}\times\coprod_{s\in S}B^{n}\pr{r'}\xrightarrow{\tau}M,
\]
where $r''>0$ is chosen so that the above composite is defined. 
\end{rem}

\begin{rem}
\label{rem:5.4.1.8}Let $n\geq0$, let $M$ be an $n$-manifold, and
let $S$ be a finite set. As a filtered colimit of Kan complexes,
the simplicial set $\Germ\pr{S,M}$ is a Kan complex. Moreover, the
map
\[
\Sing\Emb\pr{\bb R^{n}\times S,M}\to\Germ\pr{S,M}
\]
is a homotopy equivalence of Kan complexes \cite[Proposition 5.4.1.8]{HA}.
\end{rem}

The importance of the germ construction lies in the following propositions:

\begin{prop}
\label{prop:germ_group}Let $n\geq0$. The simplicial monoid $\Germ_{0}\pr{\bb R^{n}}$
is a simplicial group.
\end{prop}

\begin{prop}
\label{prop:germ_bundle}Let $n\geq0$, let $M$ be an $n$-manifold,
and let $S$ be a finite set. The map
\[
\Germ\pr{S,M}\to\Sing\Conf\pr{S,M}
\]
is a principal $\Germ_{0}\pr{\bb R^{n}}^{S}$-fibration.
\end{prop}

The proofs of Propositions \ref{prop:germ_group} and \ref{prop:germ_bundle}
require a few preliminaries.
\begin{lem}
\label{lem:cube_in_cube}Let $n\geq1$, let $X$ be a locally path-connected
topological space, let $M$ be an $n$-manifold, and let
\[
f,g:X\times D^{n}\to M
\]
be continuous maps. Suppose that, for each $x\in X$, the map $f_{x}=f\pr{x,-}:D^{n}\to M$
is injective. If there is a point $x_{0}\in X$ such that $g_{x_{0}}\pr{D^{n}}\subset f_{x_{0}}\pr{\opn{Int}D^{n}}$,
then there is a neighborhood $U$ of $x_{0}$ such that $g_{x}\pr{D^{n}}\subset f_{x}\pr{\opn{Int}D^{n}}$
for all $x\in U$.
\end{lem}

\begin{proof}
Since $f_{x_{0}}\vert\opn{Int}D^{n}$ is an injective continuous map
between $n$-manifolds, it is open. In particular, the set $V=f_{x_{0}}\pr{\opn{Int}D^{n}}$
is open. Thus, replacing $X$ by a neighborhood of $x_{0}$ if necessary,
we may assume that $g$ takes values in $f_{x_{0}}\pr{\opn{Int}D^{n}}$.
The subset $f^{-1}_{x_{0}}\pr{g_{x_{0}}\pr{D^{n}}}\subset D^{n}$
is compact and lies in the interior of $D^{n}$, so we can find some
number $0<r<1$ such that $f^{-1}_{x_{0}}\pr{g_{x_{0}}\pr{D^{n}}}\subset B^{n}\pr r$.
Replacing $X$ by a neighborhood of $x_{0}$ if necessary, we may
assume that $f\pr{X\times\overline{B^{n}\pr r}}\subset V$.

Choose a metric on $V$ and set $\varepsilon=\opn{dist}\pr{V\setminus f_{x_{0}}\pr{B^{n}\pr r},g{}_{x_{0}}\pr{D^{n}}}$.
Since $D^{n}$ and $\partial\overline{B^{n}\pr r}$ are compact, and
since $X$ is locally path-connected, there is a path-connected neighborhood
$U$ of $x_{0}$ such that
\begin{align*}
\sup_{x\in U}\opn{dist}\pr{f_{x}\pr{\partial\overline{B^{n}\pr r}},f_{x_{0}}\pr{\partial\overline{B^{n}\pr r}}} & <\varepsilon/2,\\
\sup_{x\in U}\opn{dist}\pr{g_{x}\pr{D^{n}},g_{x_{0}}\pr{D^{n}}} & <\varepsilon/2.
\end{align*}
We claim that $U$ has the desired properties. In fact, we will show
that $g_{x}\pr{D^{n}}\subset f_{x}\pr{B^{n}\pr r}$ for all $x\in U$.

Let $x_{1}\in U$. We will derive a contradiction by assuming that
there is a point $p\in D^{n}$ such that $q=g_{x_{1}}\pr p\not\in f_{x_{1}}\pr{B^{n}\pr r}$.
By construction, for each $x\in U$, the point $q$ does not belong
to $f_{x}\pr{\partial\overline{B^{n}\pr r}}$. Since $U$ is path-connected,
the maps $f_{x_{0}}:\partial\overline{B^{n}\pr r}\to V\setminus\{q\}$
and $f_{x_{1}}:\partial\overline{B^{n}\pr r}\to V\setminus\{q\}$
are homotopic. The latter map is null homotopic because it extends
to all of $\overline{B^{n}\pr r}$. Hence the map $f_{x_{0}}:\partial\overline{B^{n}\pr r}\to V\setminus\{q\}$
is also nullhomotopic. On the other hand, the distance from $q$ to
$g_{x_{0}}\pr{D^{n}}$ is less than $\varepsilon/2$, so $q$ belongs
to $f_{x_{0}}\pr{B_{n}\pr r}$. It follows that the map $f_{x_{0}}:\partial\overline{B^{n}\pr r}\to V\setminus\{q\}$
is a homotopy equivalence. Thus $\partial\overline{B^{n}\pr r}$ is
contractible, which is a contradiction.
\end{proof}

\begin{prop}
\label{prop:ev_serre}Let $n\geq0$. For any $n$-manifold $M$ and
any finite set $S$, the evaluation at the origin determines a Serre
fibration
\[
\Emb\pr{\bb R^{n}\times S,M}\to\Conf\pr{S,M}.
\]
\end{prop}

The following proof of Proposition \ref{prop:ev_serre} is due to
Alexander Kupers \cite{MO452495}.
\begin{proof}
Since Serre fibrations can be recognized locally \cite[6.3.3]{DieckAT},
it suffices to prove the following: 
\begin{itemize}
\item [(A)]For any collection $\{U_{s}\}_{s\in S}$ of pairwise disjoint
open sets of $M$ such that each $U_{s}$ is homeomorphic to $\bb R^{n}$,
the map
\[
\Emb\pr{\bb R^{n}\times S,M}\times_{\Conf\pr{S,M}}\prod_{s\in S}U_{s}\to\prod_{s\in S}U_{s}
\]
is a Serre fibration. 
\end{itemize}
We begin with a preliminary assertion. Let $\Homeo\pr M\subset\Emb\pr{M,M}$
denote the subspace consisting of the self-homeomorphisms of $M$.
We prove the following:
\begin{itemize}
\item [(B)]Let $U\subset M$ be an open set homeomorphic to $\bb R^{n}$,
and let $K\subset U$ be a compact subset. There is a continuous map
\[
\chi:K\times K\to\Homeo\pr M
\]
with the following properties:

\begin{enumerate}[label=$\bullet$]

\item For each $\pr{x,y}\in K\times K$, the support of $\chi\pr{x,y}$
is compact and lies in $U$ and, and moreover $\chi\pr{x,y}\pr x=y$.

\item For each $x\in K$, we have $\chi\pr{x,x}=\id_{\bb R^{n}}$.

\end{enumerate}

\end{itemize}
To prove (B), it suffices to consider the case where $U=M=\bb R^{n}$
and $K=D^{n}$. Choose a compactly supported smooth function $\phi:\bb R\to\bb R$
such that $\phi\pr t=1$ if $\abs t\leq1$ and $\sup_{t\in\bb R}\abs{\phi'\pr t}<1/2$.
We define $\chi:D^{n}\times D^{n}\to\Homeo\pr{\bb R^{n}}$ by
\[
\chi\pr{x,y}\pr z=\pr{\pr{1-\phi\pr{z_{i}}}z_{i}+\phi\pr{z_{i}}\pr{z_{i}+y_{i}-x_{i}}}^{n}_{i=1}.
\]
Note that this is well-defined. To see this, it suffices to show that,
for any real number $r\in[-2,2]$, the function
\[
z\mapsto\pr{1-\phi\pr z}z+\phi\pr z\pr{z+r}=z+\phi\pr zr
\]
is a self-homeomorphism of $\bb R$. This is clear, as the right-hand
side is an increasing function in $z$ (because $\sup_{t\in\bb R}\abs{\phi'\pr t}<1/2$).
The function $\chi$ has the desired properties, so we have proved
(B).

Now we prove assertion (A). Let $k\geq0$ and consider a solid commutative
diagram 
\[\begin{tikzcd}
	{I^k\times \{0\}} & {\operatorname{Emb}(\mathbb{R}^n\times S,M)\times _{\operatorname{Conf}(S,M)}\prod_{s\in S}U_s} \\
	{I^k\times I} & {\prod_{s\in S}U_s}.
	\arrow["F", from=1-1, to=1-2]
	\arrow[from=1-2, to=2-2]
	\arrow[from=1-1, to=2-1]
	\arrow["G"', from=2-1, to=2-2]
	\arrow[dashed, from=2-1, to=1-2]
\end{tikzcd}\]We wish to find a map $I^{k}\times I\to\Emb\pr{\bb R^{n}\times S,M}\times_{\Conf\pr{S,M}}\prod_{s\in S}U_{s}$
rendering the diagram commutative. Using assertion (B), we can construct
a map
\[
\Phi:I^{k}\times I\to\Homeo\pr M
\]
such that $\Phi\pr{x,0}=\id_{M}$ and $\Phi\pr{x,t}\pr{G_{s}\pr{x,0}}=G_{s}\pr{x,t}$
for all $s\in S$ and $\pr{x,t}\in I^{k}\times I$; here $G_{s}$
denotes the $s$th component of $G$. The map
\[
I^{k}\times I\to\Emb\pr{\bb R^{n}\times S,M},\,\pr{x,t}\mapsto\Phi\pr{x,t}\circ F\pr{x,0}
\]
gives the desired filler.
\end{proof}

\begin{cor}
\label{cor:germ_fib}Let $n\geq0$. For any $n$-manifold $M$ and
any finite set $S$, the map 
\[
\Germ\pr{S,M}\to\Sing\Conf\pr{S,M}
\]
is a Kan fibration.
\end{cor}

\begin{proof}
This follows from Proposition \ref{prop:ev_serre}, for Kan fibrations
are stable under filtered colimits.
\end{proof}

We can now prove Propositions \ref{prop:germ_group} and \ref{prop:germ_bundle}.
\begin{proof}
[Proof of Proposition \ref{prop:germ_group}]The claim is trivial
if $n=0$, so we will assume that $n\geq1$. Let $\sigma:\abs{\Delta^{k}}\times B^{n}\pr r\to\bb R^{n}$
be a representative of a $k$-simplex $[\sigma]\in\Germ_{0}\pr{\bb R^{n}}$.
We wish to construct an inverse of $[\sigma]$.

For each $x\in\abs{\Delta^{k}}$, let $\sigma_{x}=\sigma\pr{x,-}$
denote the restriction of $\sigma$. We will first show that there
is some $s>0$ which satisfies $\overline{B^{n}\pr s}\subset\sigma_{x}\pr{\overline{B^{n}\pr{r/2}}}$
for every $x\in\abs{\Delta^{k}}$. Since $\abs{\Delta^{k}}$ is compact,
it will suffice to prove the following:
\begin{itemize}
\item [($\ast$)]For each $x_{0}\in\abs{\Delta^{k}}$, there are neighborhoods
$U_{x_{0}}$ of $x_{0}$ and some number $s\pr{x_{0}}>0$ such that
$\overline{B^{n}\pr{s\pr{x_{0}}}}\subset\sigma_{x}\pr{\overline{B^{n}\pr{r/2}}}$
for every $x\in U_{x_{0}}$.
\end{itemize}
Assertion ($\ast$) follows from Lemma \ref{lem:cube_in_cube}. 

Now we construct the inverse of $[\sigma]$. Choose $s>0$ as in the
previous paragraph. The map
\[
F:\sigma^{-1}\pr{\overline{B^{n}\pr s}}\to\abs{\Delta^{k}}\times\overline{B^{n}\pr s},\,\pr{x,p}\mapsto\pr{x,\sigma_{x}\pr p}
\]
is a continuous bijection of compact Hausdorff spaces, so it is a
homeomorphism. Consider the composite
\[
\tau:\abs{\Delta^{k}}\times B^{n}\pr s\xrightarrow{F^{-1}}\sigma^{-1}\pr{\overline{B^{n}\pr s}}\xrightarrow{\opn{pr}}\bb R^{n}.
\]
We have $[\tau][\sigma]=1=[\sigma][\tau]$ in $\Germ_{0}\pr{\bb R^{n}}_{k}$,
so $[\tau]$ is the desired inverse of $[\sigma]$.
\end{proof}

\begin{proof}
[Proof of Proposition \ref{prop:germ_bundle}]Since the map $\Germ\pr{S,M}\to\Sing\pr{\Conf\pr{S,M}}$
is a pullback of the map $\Germ\pr M^{S}\to\Sing M^{S}$, it suffices
to consider the case where $S$ is a singleton. Let $k\geq0$. The
action of $\Germ_{0}\pr{\bb R^{n}}_{k}$ on $\Germ\pr M_{k}$ is clearly
free. It will therefore suffice to show that the induced map 
\[
\theta:\Germ\pr M_{k}/\Germ_{0}\pr{\bb R^{n}}_{k}\to\Sing\pr M_{k}
\]
is bijective. Surjectivity of $\theta$ is immediate from Corollary
\ref{cor:germ_fib}. To prove that $\theta$ is injective, we must
prove the following:
\begin{itemize}
\item [($\ast$)]Let $r>0$ and let $\sigma,\tau:\abs{\Delta^{k}}\times B^{n}\pr r\to M$
be continuous maps such that, for each $x\in\abs{\Delta^{k}}$, the
maps $\sigma_{x}=\sigma\pr{x,0}:B^{n}\pr r\to M$ and $\tau_{x}=\tau\pr{x,0}:B^{n}\pr r\to M$
are embeddings satisfying $\sigma_{x}\pr 0=\tau_{x}\pr 0$. Then there
are some number $r'>0$ and a map $g:\abs{\Delta^{k}}\times B^{n}\pr{r'}\to B^{n}\pr r$
such that
\[
\sigma\pr{x,g\pr{x,p}}=\tau\pr{x,p}
\]
for every $\pr{x,p}\in\abs{\Delta^{k}}\times B^{n}\pr{r'}$.
\end{itemize}
Using Lemma \ref{lem:cube_in_cube} and the compactness of $\abs{\Delta^{k}}$,
we can find some $0<r'<r$ which satisfies $\tau_{x}\pr{\overline{B^{n}\pr{r'}}}\subset\sigma_{x}\pr{B^{n}\pr{r/2}}$
for every $x\in\abs{\Delta^{k}}$. The map
\begin{align*}
F:\abs{\Delta^{k}}\times\overline{B^{n}\pr{r/2}} & \to\abs{\Delta^{k}}\times M\\
\pr{x,p} & \mapsto\pr{x,\sigma_{x}\pr p}
\end{align*}
is a continuous injection from a compact space to a Hausdorff space,
so it is a homeomorphism onto its image. We define $g:\abs{\Delta^{k}}\times B^{n}\pr{r'}\to B^{n}\pr r$
as the composite
\[
\abs{\Delta^{k}}\times B^{n}\pr{r'}\xrightarrow{\tau'}F\pr{\abs{\Delta^{k}}\times\overline{B^{n}\pr{r/2}}}\xrightarrow{F^{-1}}\abs{\Delta^{k}}\times\overline{B^{n}\pr{r/2}}\xrightarrow{\opn{pr}}B^{n}\pr r,
\]
where $\tau'\pr{x,p}=\pr{x,\tau_{x}\pr p}$. The map $g$ has the
desired properties.
\end{proof}

We conclude this subsection with another consequence of Corollary
\ref{cor:germ_fib}.
\begin{prop}
\label{prop:emb_conf_cart}Let $n\geq0$, let $S$ be a finite set,
let $M$ and $N$ be $n$-manifolds, and let $\phi:M\to N$ be an
embedding. The square 
\[\begin{tikzcd}
	{\operatorname{Emb}(\mathbb{R}^n\times S,M)} & {\operatorname{Emb}(\mathbb{R}^n\times S,N)} \\
	{\operatorname{Conf}(S,M)} & {\operatorname{Conf}(S,N),}
	\arrow[from=1-1, to=1-2]
	\arrow[from=1-1, to=2-1]
	\arrow[from=1-2, to=2-2]
	\arrow[from=2-1, to=2-2]
\end{tikzcd}\]determined by the evaluation at the origin, is homotopy cartesian
(with respect to the Quillen model structure for topological spaces).
\end{prop}

\begin{proof}
It will suffice to show that the square becomes homotopy cartesian
after applying the singular complex functor. Using Remark \ref{rem:5.4.1.8},
we are reduced to showing that the square 
\[\begin{tikzcd}
	{\operatorname{Germ}(S,M)} & {\operatorname{Germ}(S,N)} \\
	{\operatorname{Sing}\operatorname{Conf}(S,M)} & {\operatorname{Sing}\operatorname{Conf}(S,N)}
	\arrow[from=1-1, to=1-2]
	\arrow[from=1-1, to=2-1]
	\arrow[from=1-2, to=2-2]
	\arrow[from=2-1, to=2-2]
\end{tikzcd}\]is homotopy cartesian. By Corollary \ref{cor:germ_fib}, the vertical
arrows of this diagram are Kan fibrations. It will therefore suffice
to show that the square is strictly cartesian, which is clear.
\end{proof}

\begin{rem}
\label{rem:emb_conf_cart}Proposition \ref{prop:emb_conf_cart} remains
valid if $\phi$ is a smooth embedding of smooth manifolds of the
same dimension and $\Emb$ is replaced by $\Emb_{\sm}$. This is because
the smooth version of Proposition \ref{prop:ev_serre} is also true,
with the same proof.
\end{rem}

\subsection{Simplicial Complexes}

The term \textit{simplicial complex} can refer to several closely
related notions depending on context. Since definitions are not entirely
standardized, we briefly recall the conventions and basic properties
relevant to this paper.
\begin{defn}
\label{def:scomplex}A \textbf{simplicial complex} $K$ is a locally
finite collection of simplices in a Euclidean space, which is closed
under finite intersection and the operation of taking faces. A\textbf{
subcomplex} of $K$ is a subset $L$ of $K$ which is itself a simplicial
complex.

The union of the simplices in $K$ is called its \textbf{underlying
polyhedron} and is denoted by $\abs K$. (In the main body of the
paper, we will often blur the distinction between simplicial complexes
and their underlying polyhedra.) A \textbf{subdivision} of $K$ is
a simplicial complex $K'$ with the same underlying polyhedron as
$K'$, such that every simplex of $K'$ lies in $K$. In this situation,
we write $K'\lcone K$. 

A \textbf{triangulation} of a topological space $X$ is a homeomorphism
$t:\abs K\to X$ with $K$ a simplicial complex.
\end{defn}

\begin{example}
\label{exa:triangulation}Let $K$ be a simplicial complex, and let
$k\geq0$. There is a triangulation of $\SP^{k}\pr K=K^{k}/\Sigma_{k}$
containing $\blacktriangle_{k}\pr K$ as its subcomplex. To see this,
choose a well-ordering of the vertices of $K$. The products $K^{k}$
is a union of cells of the form $\sigma_{1}\times\cdots\times\sigma_{k}$,
where each $\sigma_{i}$ is a simplex of $M$. We subdivide each such
cell by the simplices spanned by the vertices of form $\{\pr{a^{\pr i}_{1},\dots,a^{\pr i}_{k}}\}_{0\leq i\leq d}$,
where $\pr{a^{\pr 0}_{1},\dots,a^{\pr 0}_{k}}<\cdots<\pr{a^{\pr d}_{1},\dots,a^{\pr d}_{k}}$
are vertices of $\sigma_{1}\times\cdots\times\sigma_{k}$. (The ordering
on the product is given by the categorical product of posets.) This
gives rise to a triangulation on $\SP^{k}\pr K$ with the desired
properties.
\end{example}

\begin{defn}
Let $K$ be a simplicial complex, and let $L\subset K$ be a subcomplex.
The \textbf{simplicial neighborhood} of $L$ in $K$, denoted by $N\pr{L,K}$,
is the subcomplex of $K$ generated by the simplices intersecting
$\abs L$ nontrivially. The \textbf{simplicial complement} of $L$
in $K$, denoted by $C\pr{L,K}$, is the subcomplex of $K$ consisting
of the simplices not intersecting $\abs L$. 

A subdivision $K'\lcone K$ is said to be obtained from $K$ by \textbf{deriving
it near $L$} if it is obtained by replacing the simplices of $K\setminus\pr{L\cup C\pr{L,K}}$
by the following procedure: For each $\sigma\in K\setminus\pr{L\cup C\pr{L,K}}$,
we choose a point $a_{\sigma}$ in the interior of $\sigma$. We then
adjoin $K\setminus\pr{L\cup C\pr{L,K}}$ to the faces simplices of
the form $\sigma\star a_{\tau_{1}}\star\cdots\star a_{\tau_{k}}$,
where $k\geq1$, $\sigma\in L\cup C\pr{L,K}$, and $\tau_{1}\subset\cdots\subset\tau_{k}$
are simplices in $K\setminus\pr{L\cup C\pr{L,K}}$ with $\sigma\subset\tau_{1}$.
(Here $\star$ denotes the join operation.) We call $N\pr{L,K'}$
a \textbf{derived neighborhood} of $L$ in $K$.
\end{defn}

\begin{rem}
Let $K$ be a simplicial complex, and let $L\subset K$ be a subcomplex.
Then $\abs{N\pr{L,K}}$ is a neighborhood of $\abs L$ in $\abs K$.
Moreover, $N\pr{L,K}$ is the smallest subcomplex of $K$ having this
property.
\end{rem}

\begin{rem}
\label{rem:scomplement}Let $K$ be a simplicial complex, and let
$L\subset K$ be a subcomplex. Then $\abs{C\pr{L,K}}$ is a deformation
retract of $\abs K\setminus\abs L$. Indeed, every point of $\abs K\setminus\pr{\abs L\cup\abs{C\pr{L,K}}}$
can be written uniquely as $tx+\pr{1-t}y$, where $x\in\abs L$, $y\in\abs{C\pr{L,K}}$,
and $t\in\pr{0,1}$, and an inverse homotopy equivalence is given
by mapping such a point to $y$.
\end{rem}

\begin{rem}
\label{rem:snbhd_compat}Simplicial neighborhoods are compatible with
the operation of taking intersections with subcomplexes, in the following
sense: Let $K$ be a simplicial complex, and let $K_{0},L\subset K$
be subcomplexes. Set $L_{0}=K_{0}\cap L$. We have 
\[
N\pr{K,L}\cap K_{0}=N\pr{K_{0},L_{0}}.
\]
\end{rem}

\begin{rem}
\label{rem:dnbhd_compat}Derived neighborhoods are compatible with
the operation of taking intersections with subcomplexes, in the following
sense: Let $K$ be a simplicial complex, and let $K_{0},L\subset K$
be subcomplexes. Suppose we are given a subdivision $K'\lcone K$
is obtained by deriving it near $L$, and let $K'_{0}\lcone K_{0}$
be the subdivision consisting of the simplices of $K'$ lying in $\abs{K_{0}}$.
Then $K'_{0}$ is obtained by deriving $K_{0}$ near $L_{0}$, and
$L_{0}=K'_{0}\cap L$. So by Remark \ref{rem:snbhd_compat}, we have
\[
N\pr{K',L}\cap K'_{0}=N\pr{K'_{0},L_{0}}.
\]
In particular, we have $\abs{N\pr{K',L}}\cap\abs{K_{0}}=\abs{N\pr{K'_{0},L_{0}}}$.
\end{rem}

\begin{lem}
\label{lem:small_reg_nbhd}Let $K$ be a simplicial complex, let $L\subset K$
be a subcomplex, and let $U\subset\abs K$ be a neighborhood of $\abs L$.
There is a derived neighborhood of $L$ in $K$ whose underlying polyhedron
lies in $U$.
\end{lem}

\begin{proof}
Define $d_{L}:\abs K\to[0,1]$ by mapping each vertex in $L$ to $0$
and the remaining vertices to $1$, and then extending linearly on
each simplex. Since simplices are compact, for each simplex $\sigma\in K\setminus\pr{L\cup C\pr{L,K}}$,
there is some $\varepsilon_{\sigma}\in\pr{0,1}$ such that $d^{-1}_{L}\pr{[0,\varepsilon_{\sigma}]}\cap\sigma\subset U$.
By choosing $\varepsilon_{\sigma}$ in the order of decreasing dimension,
we may assume that if $\sigma\subset\tau$, then $\varepsilon_{\sigma}\leq\varepsilon_{\tau}$.
We then choose a point $a_{\sigma}\in\sigma^{\circ}\cap d^{-1}_{L}\pr{\varepsilon_{\sigma}}$,
and form a derived subdivision $K'$ of $K$ by using the points $a_{\sigma}$.
We claim that $\abs{N\pr{L,K'}}\subset U$.

The simplicial complex $N\pr{L,K'}$ is generated by the simplices
of the form $\sigma\star a_{\sigma_{1}\star\tau_{1}}\star\cdots\star a_{\sigma_{k}\star\tau_{k}}$,
where $\sigma\subset\sigma_{1}\subset\cdots\subset\sigma_{k}$ are
simplices of $L$ and $\tau_{1}\subset\cdots\subset\tau_{k}$ are
simplices of $C\pr{K,L}$. Such a cell is contained in $d^{-1}_{L}\pr{[0,\varepsilon_{\sigma_{k}\star\tau_{k}}]}\cap\pr{\sigma_{k}\star\tau_{k}}$
and hence in $U$. Hence $\abs{N\pr{L,K'}}\subset U$, as required.
\end{proof}

\providecommand{\bysame}{\leavevmode\hbox to3em{\hrulefill}\thinspace}
\providecommand{\MR}{\relax\ifhmode\unskip\space\fi MR }
\providecommand{\MRhref}[2]{%
  \href{http://www.ams.org/mathscinet-getitem?mr=#1}{#2}
}
\providecommand{\href}[2]{#2}

\end{document}